\documentclass[a4paper,10pt]{amsart}
\usepackage{graphicx, amssymb, amsmath, amscd}
\usepackage[arrow,tips,matrix]{xy}
\usepackage{pst-node, hhline, pst-plot}
\usepackage{subfigure}

\newgray{lightgray}{0.9}

\theoremstyle{plain}
\newtheorem{theorem}{Theorem}[section]
\newtheorem{main}{Theorem} 
\newtheorem{proposition}[theorem]{Proposition}
\newtheorem{lemma}[theorem]{Lemma}
\newtheorem{corollary}[theorem]{Corollary}

\theoremstyle{definition}

\newtheorem{definition}[theorem]{Definition}
\newtheorem{remark}[theorem]{Remark}
\newtheorem{notation}[theorem]{Notation}

\numberwithin{table}{section} \numberwithin{figure}{section}
\numberwithin{equation}{section}

\DeclareMathOperator{\id}{id} \DeclareMathOperator{\D}{D}
 \DeclareMathOperator{\orthogonal}{O}
\DeclareMathOperator{\SO}{SO} 
\DeclareMathOperator{\res}{res} \DeclareMathOperator{\ext}{ext}
\DeclareMathOperator{\Iso}{Iso} \DeclareMathOperator{\Isom}{Isom}
 \DeclareMathOperator{\Irr}{Irr}
\DeclareMathOperator{\Vect}{Vect} \DeclareMathOperator{\Rep}{Rep}

\newcommand{\complex}[1]{\mathcal{#1}}
\newcommand{\complexK}{\complex{K}}
\newcommand{\complexL}{\complex{L}}

\newcommand{\vect}{\mathrm{vect}}
\newcommand{\sq}{\mathrm{sq}}
\newcommand{\eq}{\mathrm{eq}}
\newcommand{\Aff}{\mathrm{Aff}}
\newcommand{\aff}{\mathrm{aff}}
\newcommand{\F}{\mathfrak{F}}
\newcommand{\Area}{\mathrm{Area}}

\newcommand{\lineK}{{\bar{\complexK}}}
\newcommand{\lineL}{{\bar{\complexL}}}
\newcommand{\hatL}{{\hat{\complexL}}}

\newcommand{\field}[1]{\mathbb{#1}}

\newcommand{\R}{\field{R}}
\newcommand{\Q}{\field{Q}}
\newcommand{\Z}{\field{Z}}
\newcommand{\N}{\field{N}}

\begin{document}

\title[equivariant vector bundles over two-torus]{Classification of equivariant vector bundles over two-torus}
\author{Min Kyu Kim}
\address{Department of Mathematics Education,
Gyeongin National University of Education, San 59-12, Gyesan-dong,
Gyeyang-gu, Incheon, 407-753, Korea}

\email{mkkim@kias.re.kr}

\subjclass[2000]{Primary 57S25, 55P91 ; Secondary 20C99}

\keywords{equivariant vector bundle, equivariant homotopy, representation}

\begin{abstract}
We exhaustively classify topological equivariant complex vector
bundles over two-torus under a compact Lie group (not necessarily
effective) action. It is shown that inequivariant Chern classes and
isotropy representations at (at most) six points are sufficient to
classify equivariant vector bundles except a few cases. To do it, we
calculate homotopy of the set of equivariant clutching maps.
\end{abstract}

\maketitle

\section{Introduction} \label{section: introduction}

In the previous paper \cite{Ki}, we have exhaustively classified
topological equivariant complex vector bundles over two-sphere under
a compact Lie group (not necessarily effective) action. To do it, we
have calculated homotopy of equivariant clutching maps. Two-sphere
has relatively small number of effective finite group actions and
those actions deliver typical equivariant simplicial complex
structures(for example, regular polyhedron). By using this
simplicial complex structure, we have reduced the calculation from
the whole $1$-skeleton to an edge, and this makes the calculation
possible. But, readers might have doubted if our method would apply
to other spaces.

In this paper, we would do the same thing over two-torus. However, a
finite group action on two-torus does not have a typical equivariant
simplicial complex structure. Also, the calculation of homotopy of
equivariant clutching maps is not reduced to an edge, but to a union
of edges(sometimes a loop). This is a general situation which we
should face in classification of equivariant vector bundles. So,
results of this paper will make readers confident that our method is
general and applies to other spaces.

To state our results, let us introduce notations. Let $\Lambda$ be a
lattice of $\R^2.$ Let $\Iso (\R^2, \Lambda)$ be the group $\{ A \in
\Iso (\R^2) | A (\Lambda) = \Lambda \}.$ Let $\Aff (\R^2)$ and $\Aff
(\R^2, \Lambda)$ be sets of affine isomorphisms $\bar{x} \mapsto A
\bar{x} + \bar{c}$ for $\bar{x}, \bar{c} \in \R^2,$ $A \in
\Iso(\R^2)$ and $\in \Iso(\R^2, \Lambda),$ respectively. Let $\Aff_t
(\R^2, \Lambda)$ be the set of all translations in $\Aff (\R^2,
\Lambda).$ Let $\pi : \R^2 \rightarrow \R^2/\Lambda,$ $\bar{x}
\mapsto [\bar{x}]$ be the quotient map, and let $\Aff
(\R^2/\Lambda)$ be the group of affine isomorphisms of
$\R^2/\Lambda,$ i.e. the group of all $[\bar{x}] \mapsto [A \bar{x}
+ \bar{c}]$'s for $A \in \Iso (\R^2, \Lambda).$ Let $\Aff (\R^2)$
and $\Aff (\R^2 / \Lambda)$ (and their subgroups) act naturally on
$\R^2$ and $\R^2/\Lambda,$ respectively. Let $\pi_\aff : \Aff (\R^2,
\Lambda) \rightarrow \Aff (\R^2/\Lambda), A \bar{x} + \bar{c}
\mapsto [A \bar{x} + \bar{c}].$ Hereafter, we identify $\Aff_t
(\R^2, \Lambda)$ with $\R^2$ so that $\ker \pi_\aff = \Lambda.$ Let
$\Aff_t (\R^2/\Lambda)$ be the group $\pi_\aff( \Aff_t (\R^2,
\Lambda) ).$ Let a compact Lie group $G$ act affinely (not
necessarily effectively) on two torus $\R^2/\Lambda$ through a
homomorphism $\rho : G \rightarrow \Aff (\R^2/\Lambda).$ For a
bundle $E$ in $\Vect_G (\R^2/\Lambda)$ and a point $x$ in
$\R^2/\Lambda,$ denote by $E_x$ the isotropy $G_x$-representation on
the fiber at $x.$ Put $H = \ker \rho,$ i.e. the kernel of the
$G$-action on $\R^2/\Lambda.$ Let $\Irr(H)$ be the set of characters
of irreducible $H$-representations which has a $G$-action defined as
\begin{equation*}
(g \cdot \chi) (h) = \chi (g^{-1} h g)
\end{equation*}
for $h \in H,$ $g \in G,$ and $\chi \in \Irr(H).$ For $\chi \in
\Irr(H)$ and a compact Lie group $K$ satisfying $H \lhd K < G$ and
$K \cdot \chi = \chi,$ a $K$-representation $W$ is called
$\chi$-\textit{isotypical} if $\res_H^K W$ is $\chi$-isotypical, and
denote by $\Vect_K (\R^2/\Lambda, \chi)$ the set
\begin{equation*}
\{ E \in \Vect_K (\R^2/\Lambda) ~ | ~ E_x \text{ is
$\chi$-isotypical for each $x \in \R^2/\Lambda$} \}
\end{equation*}
for $\chi \in \Irr(H)$ where $\R^2/\Lambda$ has the restricted
$K$-action. In the previous paper \cite{Ki}, classification of
$\Vect_G (\R^2/\Lambda)$ has been reduced to classification of
$\Vect_{G_\chi} (\R^2/\Lambda, \chi)$ for each $\chi \in \Irr(H).$

\begin{table}[ht]
{\footnotesize
\begin{tabular}{c||c|c|c}
 $R$                                         & $R/R_t$                           & $\Lambda_t$        & $\Lambda$         \\
\hhline{=#=|=|=}
 $R_t \rtimes R/R_t$                         & $\Z_2, \D_{2,2}, \D_2,$           & $\Lambda^\sq$      & a square lattice  \\
                                             & $\Z_4, \D_4$                      &                    &  \\
\hline
 $R_t \rtimes R/R_t$                         & $\D_2, \D_{2,3}, \D_3, \Z_6,$     & $\Lambda^\eq$      & an equilateral  \\
                                             & $\D_6, \Z_3, \D_{3,2}$            &                    & triangle lattice  \\
\hline
 $R_t \rtimes R/R_t$                         & $\langle \id \rangle$             & $\Lambda^\sq$      & $\langle (m_1,0), (0,m_2) \rangle$  \\
                                             &                                   &                    & for $m_1, m_2 \in \N$             \\
\hline
 $\langle R_t, [A\bar{x}+\bar{c}] \rangle,$  & $\D_1, \D_{1,4}$                  & $\Lambda^\sq$      & \\
 $A\bar{c}+\bar{c}=0$                        &                                   &                    & \\
\hline
 $\langle R_t, [A\bar{x}+\bar{c}] \rangle,$  & $\D_1, \D_{1,4}$                  & $\Lambda^\sq$      & \\
 $A\bar{c}+\bar{c}=\bar{l}_0$                &                                   &                    & \\
\end{tabular}}
\caption{\label{table: finite group of aff} Finite subgroups of
$\Aff (\R^2/\Lambda)$}
\end{table}

The classification of $\Vect_{G_\chi} (\R^2/\Lambda, \chi)$ is
highly dependent on the $G_\chi$-action on the base space
$\R^2/\Lambda,$ i.e. on the image $\rho(G_\chi),$ and classification
is actually given case by case according to $\rho(G_\chi).$ So, we
need to describe $\rho(G_\chi)$ in a moderate way. For this, we
would list all possible finite $\rho(G_\chi)$'s up to conjugacy, and
then assign a fundamental domain to each $\rho(G_\chi).$ Cases of
nonzero-dimensional $\rho(G_\chi)$ are relatively simple and
separately dealt with as special cases.

Now, we would list all possible finite $\rho(G_\chi)$ up to
conjugacy. Let $R$ be a finite subgroup of $\Aff (\R^2/\Lambda),$
i.e. $R$ acts effectively and affinely on $\R^2/\Lambda$ through the
natural action of $\Aff (\R^2/\Lambda)$ on $\R^2/\Lambda.$ Let $R_t$
be the translation part $R \cap \Aff_t (\R^2/\Lambda)$ of $R,$ and
put $\Lambda_t = \pi_\aff^{-1} ( R_t ).$ Then, $R / R_t$ is
isomorphic to the subgroup $\{ A \in \Iso (\R^2, \Lambda) ~ | ~
[A\bar{x}+\bar{c}] \in R \},$ and we will henceforward use the
notation $R/R_t$ to denote the subgroup. For simplicity, $R/R_t$ is
confused with the subgroup $\pi_\aff ( R/R_t )$ of $\Aff
(\R^2/\Lambda)$ according to context. Let $\Lambda^\sq$ and
$\Lambda^\eq$ be lattices $\langle (1,0), (0,1) \rangle$ and
$\langle (1,0), (1/2,\sqrt{3}/2) \rangle$ in $\R^2,$ respectively.
Let $\Z_n$ be the cyclic group $\langle a_n \rangle$ which is the
subgroup of $\Iso(\R^2)$ where $a_n$ is the rotation through the
angle $2\pi/n.$ Let $\D_n $ and $\D_{n, l}$ be dihedral groups
$\langle a_n, b \rangle$ and $\langle a_n, a_{l n} b \rangle$ with
$l \in \Q^*$ which are subgroups of $\Iso (\R^2)$ where $b$ is the
reflection in the $x$-axis and $a_{l n}$ is the rotation through the
angle $2\pi/nl.$ If $\Lambda$ in $\R^2$ is preserved by one of these
cyclic and dihedral groups, then the cyclic or dihedral group is
also regarded as a subgroup of $\Aff ( \R^2 / \Lambda ).$ In Section
\ref{section: closed subgroup}, it is shown that each finite group
in $\Aff ( \R^2 / \Lambda )$ for some $\Lambda$ is conjugate to an
$R$ entry of Table \ref{table: finite group of aff}. Let us explain
for the table except cases when $R/R_t = \D_1, \D_{1,4}.$ Cases of
$R/R_t = \D_1, \D_{1,4}$ are explained in Section \ref{section:
closed subgroup}. In each row, pick one of groups in the $R/R_t$
entry, and call it $Q.$ And, denote by $L_t$ the $\Lambda_t$ entry
in the row. Also, pick any sublattice in $L_t$ which is preserved by
$Q$ and satisfies the restriction stated in the $\Lambda$ entry in
the row, and call it just $\Lambda.$ Then, $Q, L_t, \Lambda$
determine the semidirect product $\pi_\aff (L_t) \rtimes \pi_\aff
(Q)$ in $\Aff(\R^2/\Lambda).$ Denote it by $R.$ So defined $R$
satisfies $R_t = \pi_\aff (L_t), \Lambda_t = L_t, R/R_t = Q.$ In
this way, $R/R_t, \Lambda_t, \Lambda$ entries of each row determine
the unique $R.$ Since each finite group in $\Aff ( \R^2 / \Lambda )$
for some $\Lambda$ is conjugate to one of these $R$'s, it is assumed
that $\rho(G_\chi)=R$ for some $R$ of Table \ref{table: finite group
of aff} in this section. In the below, it is observed that two
groups with the same $R/R_t$- and $\Lambda_t$-entries behave
similarly even though their $\Lambda$'s are different.

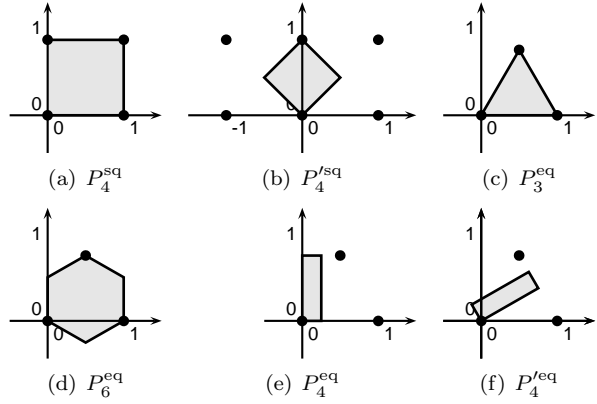
\begin{figure}[!ht]
\begin{center}

\mbox{

\subfigure[$P_4^\sq$]{
\begin{pspicture}(-0.5,-0.5)(1.5,1.5)\footnotesize

\psgrid[gridcolor=white, subgridcolor=white,
gridlabels=7pt](0,0)(0,0)(1,1)
\pspolygon[fillstyle=solid,fillcolor=lightgray,linestyle=none](0,0)(1,0)(1,1)(0,1)
\psline[linewidth=1pt](0,0)(1,0)(1,1)(0,1)

\psaxes[labels=none, ticks=none]{->}(0,0)(-0.5,-0.5)(1.5,1.5)
\psdots[dotsize=4pt](0,0)(1,0)(1,1)(0,1)

\end{pspicture}
}

\subfigure[$P_4^{\prime \sq}$]{
\begin{pspicture}(-1.5,-0.5)(1.5,1.5)\footnotesize

\psgrid[gridcolor=white, subgridcolor=white,
gridlabels=7pt](0,0)(-1,0)(1,1)
\pspolygon[fillstyle=solid,fillcolor=lightgray,linestyle=none](0,0)(0.5,0.5)(0,1)(-0.5,0.5)(0,0)
\psline[linewidth=1pt](0,0)(0.5,0.5)(0,1)(-0.5,0.5)(0,0)
\psaxes[labels=none, ticks=none]{->}(0,0)(-1.5,-0.5)(1.5,1.5)
\psdots[dotsize=4pt](0,0)(1,0)(-1,0)(-1,1)(0,1)(1,1)

\end{pspicture}
}

\subfigure[$P_3^\eq$]{
\begin{pspicture}(-0.5,-0.5)(1.5,1.5)\footnotesize

\psgrid[gridcolor=white, subgridcolor=white,
gridlabels=7pt](0,0)(-0.5,0.5)(1.4,1)

\psaxes[labels=none, ticks=none]{->}(0,0)(-0.5,-0.5)(1.5,1.5)
\pspolygon[fillstyle=solid,fillcolor=lightgray,linestyle=none](0,0)(1,0)(0.5,0.866)(0,0)
\psline[linewidth=1pt](0,0)(1,0)(0.5,0.866)(0,0)

\psdots[dotsize=4pt](0,0)(1,0)
\rput{60}(0,0){\psdots[dotsize=4pt](1,0)}
\end{pspicture}
}

}

\mbox{

\subfigure[$P_6^\eq$]{
\begin{pspicture}(-0.5,-0.5)(1.5,1.5)\footnotesize

\psgrid[gridcolor=white, subgridcolor=white,
gridlabels=7pt](0,0)(-0.5,-0.5)(1.4,1)
\pspolygon[fillstyle=solid,fillcolor=lightgray,linestyle=none](0,0)(0.5,-0.288)(1,0)(1,0.578)(0.5,0.866)(0,0.578)(0,0)
\psline[linewidth=1pt](0,0)(0.5,-0.288)(1,0)(1,0.578)(0.5,0.866)(0,0.578)(0,0)
\psaxes[labels=none, ticks=none]{->}(0,0)(-0.5,-0.5)(1.5,1.5)
\psdots[dotsize=4pt](0,0)(1,0)
\rput{60}(0,0){\psdots[dotsize=4pt](1,0)}

\end{pspicture}
}

\subfigure[$P_4^\eq$]{
\begin{pspicture}(-1.5,-0.5)(1.5,1.5)\footnotesize

\psgrid[gridcolor=white, subgridcolor=white,
gridlabels=7pt](0,0)(0,0)(1,1) \psaxes[labels=none,
ticks=none]{->}(0,0)(-0.5,-0.5)(1.5,1.5)
\pspolygon[fillstyle=solid,fillcolor=lightgray,linestyle=none](0,0)(0.25,0)(0.25,0.866)(0,0.866)
\psline[linewidth=1pt](0,0)(0.25,0)(0.25,0.866)(0,0.866)(0,0)
\psdots[dotsize=4pt](0,0)(1,0)
\rput{60}(0,0){\psdots[dotsize=4pt](1,0)}

\end{pspicture}
}

\subfigure[$P_4^{\prime \eq}$]{
\begin{pspicture}(-0.5,-0.5)(1.5,1.5)\footnotesize

\psgrid[gridcolor=white, subgridcolor=white,
gridlabels=7pt](0,0)(0,0)(1,1)
\pspolygon[fillstyle=solid,fillcolor=lightgray,linestyle=none](0,0)(-0.125,0.216)(0.625,0.649)(0.75,0.433)
\psline[linewidth=1pt](0,0)(-0.125,0.216)(0.625,0.649)(0.75,0.433)(0,0)
\psaxes[labels=none, ticks=none]{->}(0,0)(-0.5,-0.5)(1.5,1.5)
\psaxes[labels=none, ticks=none]{->}(0,0)(-0.5,-0.5)(1.5,1.5)
\psdots[dotsize=4pt](0,0)(1,0)
\rput{60}(0,0){\psdots[dotsize=4pt](1,0)}

\end{pspicture}
} }

\end{center}
\caption{\label{figure: polygonal area} Some polygonal domain}
\end{figure}

To each finite group $R$ in Table \ref{table: finite group of aff},
we assign a polygonal area $|P_R|$ in Table \ref{table:
two-dimensional fundamental domain} called a \textit{two-dimensional
fundamental domain} according to its $R/R_t$ and $\Lambda_t.$ Let
$P_4^\sq, P_4^{\prime \sq},$ $P_3^\eq, P_6^\eq$ be gray colored
areas defined by regular polygons in Figure \ref{figure: polygonal
area} where dots are points of $\Lambda^\sq$ or $\Lambda^\eq.$ And,
let $P_4^\eq, P_4^{\prime \eq}$ be convex hulls of
\begin{equation*}
\Big\{ O, \Big(0, \frac {\sqrt{3}} 2 \Big), \Big( \frac 1 4, \frac
{\sqrt{3}} 2 \Big), \Big( \frac 1 4, 0 \Big) \Big\}, \quad \Big\{ O,
\Big(- \frac 1 8, \frac {\sqrt{3}} 8 \Big), \Big( \frac 5 8, \frac
{3\sqrt{3}} 8 \Big), \Big( \frac 3 4, \frac {\sqrt{3}} 4 \Big)
\Big\},
\end{equation*}
respectively.

\begin{table}[!ht]
{\footnotesize
\begin{tabular}{c|c||c}
$R/R_t$                       & $\Lambda_t$    & $|P_R|$      \\
\hhline{=|=#=}
$\Z_2 $                       & $\Lambda^\sq$  & $\frac 1 2 P_4^\sq \cup ((\frac 1 2, 0)+\frac 1 2 P_4^\sq)$  \\
$\D_{2,2}$                    &                & $\frac 1 2 P_4^{\prime \sq} \cup ((\frac 1 4, \frac 1 4)+\frac 1 2 P_4^{\prime \sq})$  \\
$\D_2, \Z_4, \D_4$            &                & $\frac 1 2 P_4^\sq$  \\
\hline
$\D_2$                        & $\Lambda^\eq$  & $P_4^\eq$  \\
$\D_{2,3}$                    &                & $P_4^{\prime \eq}$  \\
$\D_3, \Z_6, \D_6$            &                & $P_3^\eq$  \\
$\Z_3, \D_{3,2}$              &                & $P_6^\eq$  \\
\hline
$\langle \id \rangle$         & $\Lambda^\sq$  & $P_4^\sq$  \\
\hline
$\D_1,$                       & $\Lambda^\sq$  & $\frac 1 2 \bar{c} + \big( \frac 1 2 P_4^\sq \cup (\frac 1 2, 0)+\frac 1 2 P_4^\sq \big)$  \\
$A\bar{c}+\bar{c}=0$          &                &             \\
\hline
$\D_{1,4},$                   & $\Lambda^\sq$  & $\frac 1 2 \bar{c} + P_4^{\prime \sq}$  \\
$A\bar{c}+\bar{c}=0$          &                &             \\
\hline
$\D_1,$                       & $\Lambda^\sq$  & $\frac 1 2 \bar{c} + \big ( \frac 1 2 P_4^\sq \cup (0, -\frac 1 2)+\frac 1 2 P_4^\sq \big)$  \\
$A\bar{c}+\bar{c}=\bar{l}_0$  &                &             \\
\hline
$\D_{1,4},$                   & $\Lambda^\sq$  & $\frac 1 2 \bar{c} + (\frac 1 4, -\frac 1 4)+P_4^{\prime \sq}$  \\
$A\bar{c}+\bar{c}=\bar{l}_0$  &                &             \\
\end{tabular}}
\caption{\label{table: two-dimensional fundamental domain}
Two-dimensional fundamental domain}
\end{table}

\begin{table}[!ht]
{\footnotesize
\begin{tabular}{c|c||c|l}
$R/R_t$                 & $\Lambda_t$    & $\bar{D}_R$                                                   & $\sim$      \\
\hhline{=|=#=|=}
 $\Z_2$                 & $\Lambda^\sq$  & $[b(\bar{e}^1), \bar{v}^2, \bar{v}^3, b(\bar{e}^3))]$         &   \\
 $\D_{2,2}$             &                & $[\bar{v}^0, \bar{v}^1, \bar{v}^2, \bar{v}^3, \bar{v}^0]$     & $v^1 \sim v^2$    \\
 $\D_2$                 &                & $[\bar{v}^0, \bar{v}^1, \bar{v}^2, \bar{v}^3, \bar{v}^0]$     &   \\
 $\Z_4$                 &                & $[\bar{v}^0, \bar{v}^1, \bar{v}^2]$                           &   \\
 $\D_4$                 &                & $[\bar{v}^0, \bar{v}^1, \bar{v}^2]$                           &   \\
\hline
 $\D_2$                 & $\Lambda^\eq$  & $[b(\bar{e}^2), \bar{v}^3, \bar{v}^0, \bar{v}^1, \bar{v}^2]$  & $v^2 \sim v^3$ \\
 $\D_{2,3}$             &                & $[b(\bar{e}^1), \bar{v}^2, \bar{v}^3, \bar{v}^0, \bar{v}^1]$  & $v^1 \sim v^2$ \\
 $\D_3$                 &                & $[b(\bar{e}^2), \bar{v}^0, b(\bar{e}^0)]$                     &   \\
 $\Z_6$                 &                & $[\bar{v}^0, b(\bar{e}^0)]$                                   &   \\
 $\D_6$                 &                & $[\bar{v}^0, b(\bar{e}^0)]$                                   &   \\
\hline
 $\Z_3$                 & $\Lambda^\eq$  & $[\bar{v}^0, \bar{v}^1]$                                      &   \\
 $\D_{3,2}$             &                & $[\bar{v}^0, \bar{v}^1]$                                      &   \\
\hline
 $\langle \id \rangle$  & $\Lambda^\sq$  & $[\bar{v}^2, \bar{v}^3, \bar{v}^0]$                           & $v^i \sim v^{i^\prime}$  \\
\hline
 $\D_1,$                & $\Lambda^\sq$  & $[\bar{v}^1, \bar{v}^2, \bar{v}^3, \bar{v}^0]$                & $v^0 \sim v^3, v^1 \sim v^2$  \\
 $A\bar{c}+\bar{c}=0$               &                &                                                               &    \\
\hline
 $\D_{1,4},$            & $\Lambda^\sq$  & $[\bar{v}^1, \bar{v}^2, \bar{v}^3, \bar{v}^0]$                & $v^0 \sim v^2, v^1 \sim v^3$ \\
 $A\bar{c}+\bar{c}=0$               &                &                                                               &    \\
\hline
 $\D_1,$                & $\Lambda^\sq$  & $[\bar{v}^2, \bar{v}^3, \bar{v}^0]$                           & $v^i \sim v^{i^\prime}$ \\
 $A\bar{c}+\bar{c}=\bar{l}_0$             &                &                                                               &    \\
\hline
 $\D_{1,4},$            & $\Lambda^\sq$  & $[\bar{v}^1, \bar{v}^2, \bar{v}^3, \bar{v}^0]$                & $v^0 \sim v^2, v^1 \sim v^3$ \\
 $A\bar{c}+\bar{c}=\bar{l}_0$             &                &                                                               &    \\

\end{tabular}}
\caption{\label{table: one-dimensional fundamental domain}
One-dimensional fundamental domain}
\end{table}

In dealing with equivariant vector bundles over two-torus, we need
to consider isotropy representations at a few (at most six) points
of $\R^2/\Lambda.$ We would specify those points. For this, we endow
$|P_R|$ with the natural simplicial complex structure denoted by
$P_R,$ and we would label its vertices and edges. For simplicity, we
allow that a face of a two-dimensional simplicial complex need not
be a triangle, and we consider an $n$-gon as the simplicial complex
with one face, $n$ edges, $n$ vertices. Define $i_R$ as the integer
such that $|P_R|$ is an $i_R$-gon. Put $\bar{v}^0 = O$ when $R/R_t
\ne \D_1, \D_{1,4}$ where $O$ is the origin of $\R^2.$ Put
$\bar{v}^0 = 1/2 \bar{c}$ when $R/R_t = \D_1, \D_{1,4}$ with
$A\bar{c}+\bar{c}=0,$ and put $\bar{v}^0 = 1/2 \bar{c} - 1/2
\bar{l}_0^\perp$ when $R/R_t = \D_1, \D_{1,4}$ with
$A\bar{c}+\bar{c}=\bar{l}_0.$ Label vertices of $P_R$ as $\bar{v}^i$
so that $\bar{v}^0, \bar{v}^1, \bar{v}^2, \cdots, \bar{v}^{i_R-1}$
are consecutive vertices arranged in the clockwise way around $P_R.$
We use $\Z_{i_R}$ as the index set of vertices. Also, denote by
$\bar{e}^i$ the edge of $P_R$ connecting $\bar{v}^i$ and
$\bar{v}^{i+1}$ for $i \in \Z_{i_R}.$ For any two points $\bar{a},
\bar{b}$ in $|P_R|,$ denote by $[\bar{a}, \bar{b}]$ the line
connecting $\bar{a}$ and $\bar{b}.$ Also for any points
$\bar{a}^i$'s for $i=0, \cdots, n$ in $|P_R|,$ we denote by
$[\bar{a}^0, \cdots ,\bar{a}^n]$ the union of $[\bar{a}^i,
\bar{a}^{i+1}]$'s for $i=0, \cdots ,n-1.$ To each finite group $R$
in Table \ref{table: finite group of aff}, we assign the union
$\bar{D}_R$ of lines in Table \ref{table: one-dimensional
fundamental domain} called a \textit{one-dimensional fundamental
domain} according to its $R/R_t$ and $\Lambda_t.$ Here, $b(\sigma)$
is the barycenter of $\sigma$ for any simplex $\sigma.$ For each
$\bar{D}_R$ in Table \ref{table: one-dimensional fundamental
domain}, express $\bar{D}_R$ as $[\bar{a}^0, \cdots, \bar{a}^i,
\cdots, \bar{a}^{n_R-1}]$ where we respect the order of
$\bar{a}^i$'s and $n_R$ is the number of $\bar{a}^i$'s. If
$[\bar{a}^0, \bar{a}^1] \subset |\bar{e}^{i_0}|$ with $0 \le i_0 \le
i_G -1,$ then put $\bar{d}^i = \bar{a}^{i-i_0}$ for each $i_0 \le i
\le i_0+n_R-1.$ In addition to these, put $\bar{d}^{-1} = b(P_R).$
Denote $\{ i \in \Z | i_0 \le i \le i_0 + n_R-1 \}$ by $I_R,$ and
denote $I_R-\{ i_0 + n_R-1 \},$ $I_R \cup \{ -1 \}$ by $I_R^-,$
$I_R^+,$ respectively. So, the set of all $\bar{d}^i$'s and
$\bar{D}_R$ are expressed as $(\bar{d}^i)_{i \in I_R^+}$ and
$\bigcup_{i \in I_R^-} [\bar{d}^i, \bar{d}^{i+1}],$ respectively.
For example, in the case of $R/R_t = \D_{2,3},$
$\bar{d}^1=b(\bar{e}^1),$ $\bar{d}^2=\bar{v}^2,$
$\bar{d}^3=\bar{v}^3,$ $\bar{d}^4=\bar{v}^0,$ $\bar{d}^5=\bar{v}^1,$
and $I_R=\{ 1, 2, 3, 4, 5 \}.$ Also, denote by $d^i$ the image
$\pi(\bar{d}^i)$ for each $i \in I_R^+.$ These $d^i$'s are wanted
points. The $\sim$ entry of Table \ref{table: one-dimensional
fundamental domain} is explained later.

When $R=\rho(G_\chi),$ we simply denote $P_R, \bar{D}_R, i_R, I_R$
by $P_\rho, \bar{D}_\rho, i_\rho, I_\rho,$ respectively. Next, we
define a semigroup with which we would describe isotropy
representations at $(d^i)_{i \in I_\rho^+}$ of bundles in
$\Vect_{G_\chi} (\R^2/\Lambda, \chi).$

\begin{definition} \label{definition: A_G}
Let $A_{G_\chi} (\R^2/\Lambda, \chi)$ be the set of
$(|I_\rho^+|)$-tuple $(W_{d^i})_{i \in I_\rho^+}$'s such that
\begin{enumerate}
  \item $W_{d^i} \in \Rep((G_\chi)_{d^i})$ for each $i \in I_\rho^+,$
  \item $W_{d^{-1}}$ is $\chi$-isotypical,
  \item $W_{d^{i^\prime}} \cong ~ ^g W_{d^i}$
        if $d^{i^\prime} = g \cdot d^i$ for
        some $g \in G_\chi,$ $i, i^\prime \in I_\rho,$
  \item $\res_{(G_\chi)_{C(\bar{d}^i)}}^{(G_\chi)_{d^{-1}}} W_{d^{-1}}
        \cong \res_{(G_\chi)_{C(\bar{d}^i)}}^{(G_\chi)_{d^i}} W_{d^i}$
        for each $i \in I_\rho,$
  \item $\res_{(G_\chi)_{|e^i|}}^{(G_\chi)_{d^i}} W_{d^i}
        \cong \res_{(G_\chi)_{|e^i|}}^{(G_\chi)_{d^{i+1}}} W_{d^{i+1}}$
        for each $i \in I_\rho^-.$
\end{enumerate}
Here, $C(\bar{x}) = \pi([\bar{x}, b(P_\rho)])$ for $\bar{x} \in
\partial |P_\rho|,$ and $|e^i|$ is the image $\pi(|\bar{e}^i|).$
Also, $(G_\chi)_{C(\bar{x})}$ and $(G_\chi)_{|e^i|}$ are subgroups
of $G_\chi$ fixing $C(\bar{x})$ and $|e^i|,$ respectively. And, let
$p_{\vect} : \Vect_{G_\chi} (\R^2/\Lambda, \chi) \rightarrow
A_{G_\chi} ( \R^2/\Lambda, \chi )$ be the map defined as $[E]
\mapsto ( E_{d^i} )_{i \in I_\rho^+}.$
\end{definition}
Groups $(G_\chi)_{C(\bar{x})}$ and $(G_\chi)_{|e^i|}$ are calculated
in Section \ref{section: equivariant simplicial complex}, and
$p_{\vect}$ is well defined.

Now, we can state our results. Let $c_1 : \Vect_{G_\chi}
(\R^2/\Lambda, \chi) \rightarrow H^2 (\R^2/\Lambda)$ be the map
defined as $[E] \mapsto c_1 (E).$ Denote by $l_\rho$ the number
$|\rho(G_\chi)|.$

\begin{main} \label{main: by isotropy and chern}
Assume that $R/R_t = \Z_n$ for $n=1, 2, 3, 4, 6$ when $\rho(G_\chi)
= R$ for some $R$ of Table \ref{table: finite group of aff}. Then,
\begin{equation*}
p_\vect \times c_1 : \Vect_{G_\chi} (\R^2/\Lambda, \chi) \rightarrow
A_{G_\chi} ( \R^2/\Lambda, \chi ) \times H^2 (\R^2/\Lambda)
\end{equation*}
is injective. For each element in $A_{G_\chi} ( \R^2/\Lambda, \chi
),$ the set of the first Chern classes of bundles in its preimage
under $p_{\vect}$ is equal to the set $\{ \chi( \id ) ( l_\rho k +
k_0 ) ~ | ~ k \in \Z \}$ where $k_0$ is dependent on it.
\end{main}

\begin{main} \label{main: not by isotropy and chern class}
Assume that $R/R_t = \D_1$ with $A\bar{c}+\bar{c}=\bar{l}_0$ when
$\rho(G_\chi) = R$ for some $R$ of Table \ref{table: finite group of
aff}. For each element in $A_{G_\chi} (\R^2/\Lambda, \chi),$ its
preimage under $p_{\vect}$ has two elements which have the same
Chern class.
\end{main}

\begin{main} \label{main: only by isotropy}
Assume that finite $\rho(G_\chi)$ is not equal to one of groups
appearing in the above two theorems. Then, $p_{\vect}$ is an
isomorphism.
\end{main}

\begin{main} \label{main: nonzero-dimensional}
Assume that $\rho(G_\chi)$ is one-dimensional. Then, there exists a
circle $C$ in $\R^2/\Lambda$ such that the map $G_\chi
\times_{(G_\chi)_C} C \rightarrow \R^2/\Lambda,$ $[g, x] \mapsto gx$
for $g \in G_\chi, x \in C$ is an equivariant isomorphism where
$(G_\chi)_C$ is the subgroup of $G_\chi$ preserving $C.$ And,
\begin{equation*}
\Vect_{G_\chi} (\R^2/\Lambda, \chi) \rightarrow \Vect_{(G_\chi)_C}
(C, \chi), \quad E \mapsto E|_C
\end{equation*}
is an isomorphism.
\end{main}

In Theorem \ref{main: not by isotropy and chern class}, two bundles
in each preimage are not distinguished by Chern classes and isotropy
representations, so it might be regarded as an exceptional case. In
\cite{Ki2}, we explain for this case by showing that the two bundles
are pull backs of equivariant vector bundles over Klein bottle which
are distinguished by Chern classes and isotropy representations.
Since equivariant complex vector bundles over circle are classified
in \cite{CKMS}, Theorem \ref{main: nonzero-dimensional} gives
classification of $\Vect_{G_\chi} (\R^2/\Lambda, \chi)$ when
$\rho(G_\chi)$ is one-dimensional.

This paper is organized as follows. In Section \ref{section: closed
subgroup}, we list all finite subgroups of $\Aff (\R^2/\Lambda)$ up
to conjugacy. In Section \ref{section: nonzero dimensional case}, we
prove Theorem \ref{main: nonzero-dimensional}. In Section
\ref{section: equivariant simplicial complex}, we consider
$\R^2/\Lambda$ as a quotient space of the underlying space of an
equivariant simplicial complex which will be denoted by
$\lineK_\rho.$ And, we investigate equivariance of $\R^2/\Lambda$ by
calculating isotropy groups at some points which come from vertices
or barycenters of $|\lineK_\rho|.$ In Section \ref{section:
preclutching map}, we consider an equivariant vector bundle over
$\R^2/\Lambda$ as an equivariant clutching construction of an
equivariant vector bundle over $|\lineK_\rho|.$ For this, we define
equivariant clutching map. In Section \ref{section: relation}, we
investigate relations between homotopy of equivariant clutching
maps, $\Vect_{G_\chi} (\R^2/\Lambda, \chi),$ and $A_{G_\chi}
(\R^2/\Lambda, \chi).$ From these relations, it is shown that our
classifications in most cases are obtained by calculation of
homotopy of equivariant clutching maps. In Section \ref{section:
pointwise clutching map}, we review the concept and results of
equivariant pointwise clutching map from the previous paper
\cite{Ki}, and supplement these with two more cases. In calculation
of homotopy of equivariant clutching maps, equivariant pointwise
clutching map plays a key role because an equivariant clutching map
can be considered as a continuous collection of equivariant
pointwise clutching maps. In Section \ref{section: lemmas on
fundamental group}, we recall two useful lemmas on fundamental
groups from \cite{Ki}. In Section \ref{section: equivariant
clutching maps}, we prove technical results needed in dealing with
equivariant clutching maps through equivariant pointwise clutching
maps. In Section \ref{section: not id and D_1 cases} and
\ref{section: id and D_1 cases}, we prove Theorem \ref{main: by
isotropy and chern}, \ref{main: not by isotropy and chern class},
\ref{main: only by isotropy}.

\section{finite subgroups of $\Aff (\R^2/\Lambda)$}
 \label{section: closed subgroup}

In this section, we list all finite subgroups of $\Aff
(\R^2/\Lambda)$ up to conjugacy, and explain for two-dimensional
fundamental domain.

First, we introduce notations on the group of affine isomorphisms on
two torus. Let $\Lambda$ be a lattice of $\R^2.$ Let $\Iso (\R^2,
\Lambda)$ be the group $\{ A \in \Iso (\R^2) | A (\Lambda) = \Lambda
\},$ and let $\Aff (\R^2, \Lambda)$ be the set of affine
isomorphisms $\bar{x} \mapsto A \bar{x} + \bar{c}$ for $\bar{x},
\bar{c} \in \R^2$ and $A \in \Iso(\R^2, \Lambda).$ Let $\Aff_t
(\R^2, \Lambda)$ and $\Aff_0 (\R^2, \Lambda)$ be sets of all
translations and all orientation preserving affine isomorphisms in
$\Aff (\R^2, \Lambda),$ respectively. Let $\pi : \R^2 \rightarrow
\R^2/\Lambda, \bar{x} \mapsto [\bar{x}]$ be the quotient map, and
let $\Aff (\R^2/\Lambda)$ be the group of affine isomorphisms of
$\R^2/\Lambda,$ i.e. the group of all $[\bar{x}] \mapsto [A \bar{x}
+ \bar{c}]$'s for $A \in \Iso (\R^2, \Lambda).$ Let $\Aff (\R^2)$
and $\Aff (\R^2 / \Lambda)$ (and their subgroups) act naturally on
$\R^2$ and $\R^2/\Lambda,$ respectively. And, a compact Lie group
$R$ is said to act \textit{naturally} on $\R^2$ or $\R^2/\Lambda$ if
the action satisfies the following two conditions:
\begin{enumerate}
  \item $R$ is a subgroup of $\Aff (\R^2)$ or
  $\Aff (\R^2 / \Lambda),$
  \item the $R$-action is the action as a subgroup
  of $\Aff (\R^2)$ or $\Aff (\R^2 / \Lambda),$ respectively.
\end{enumerate}
Let $\pi_\aff : \Aff (\R^2, \Lambda) \rightarrow \Aff
(\R^2/\Lambda), ~ A \bar{x} + \bar{c} \mapsto [A \bar{x} +
\bar{c}].$ We identify $\Aff_t (\R^2, \Lambda)$ with $\R^2$ so that
$\ker \pi_\aff = \Lambda.$ Let $\Aff_t (\R^2/\Lambda)$ and $\Aff_0
(\R^2/\Lambda)$ be groups $\pi_\aff( \Aff_t (\R^2, \Lambda) )$ and
$\pi_\aff( \Aff_0 (\R^2, \Lambda) ),$ respectively.

Next, we introduce notations on finite subgroups of $\Aff
(\R^2/\Lambda).$ Let $R$ be a finite subgroup of $\Aff
(\R^2/\Lambda),$ i.e. the finite group $R$ acts naturally on
$\R^2/\Lambda.$ Let $R_0$ and $R_t$ be $R \cap \Aff_0
(\R^2/\Lambda)$ and $R \cap \Aff_t (\R^2/\Lambda),$ respectively.
Put $\Lambda_t = \pi_\aff^{-1} ( R_t ).$ Note that $R_t \lhd R$ and
$R_t \cong \Lambda_t / \Lambda$ so that $R_t$ is isomorphic to
$\Z_{m_1} \times \Z_{m_2}$ because $\Lambda_t$ has two generators
$\bar{l}_1, \bar{l}_2$ such that $\Lambda = \langle m_1 \bar{l}_1,
m_2 \bar{l}_2 \rangle$ for some $m_1, m_2 \in \N.$ For each element
$g=[A\bar{x}+\bar{c}]$ of $R,$ it can be shown that $A ( \Lambda_t )
= \Lambda_t$ by normality of $R_t$ in $R,$ i.e. $R / R_t$ preserves
$\Lambda_t$ in addition to $\Lambda.$ Through the map
\begin{equation*}
R \rightarrow \Iso (\R^2, \Lambda), \quad [A\bar{x}+\bar{c}] \mapsto
A,
\end{equation*}
$R / R_t$ is isomorphic to the subgroup $\{ A \in \Iso (\R^2,
\Lambda) ~ | ~ [A\bar{x}+\bar{c}] \in R \},$ and we will use the
notation $R/R_t$ to denote the subgroup. For simplicity, $R/R_t$ is
confused with the subgroup $\pi_\aff ( R/R_t )$ of $\Aff
(\R^2/\Lambda)$ according to context.

For a compact subgroup $R$ of $\Aff (\R^2/\Lambda)$ and its natural
action on $\R^2/\Lambda,$ we define conjugacy. Given an affine
isomorphism
\begin{equation*}
\bar{\eta} : \R^2 \rightarrow \R^2, \quad \bar{x} \mapsto A^\prime
\bar{x}+ \bar{c}^\prime
\end{equation*}
for some $A^\prime \in \Iso(\R^2),$ $\bar{c}^\prime \in \R^2,$ it
induces the affine isomorphism
\begin{equation*}
\eta: \R^2/\Lambda \rightarrow \R^2/ A^\prime (\Lambda), \quad
[\bar{x}] \mapsto [A^\prime \bar{x}+ \bar{c}^\prime]
\end{equation*}
and the group isomorphism
\begin{equation*}
\eta^*: \Aff(\R^2/\Lambda) \rightarrow \Aff(\R^2/ A^\prime
(\Lambda)), \quad B \mapsto \eta B \eta^{-1}.
\end{equation*}
Then, $\eta^*(R)$ is called \textit{conjugate} to $R,$ and the
natural action of $\eta^*(R)$ on $\R^2/ A^\prime (\Lambda)$ is
called the \textit{conjugate action} induced by $\bar{\eta}.$ It can
be easily checked that
\begin{equation}
\label{equation: conjugate action} \eta^*(R)_t = \eta^*(R_t) \quad
\text{and} \quad \Lambda_t^\prime = A^\prime (\Lambda_t)
\end{equation}
where $\Lambda_t^\prime$ is equal to $\pi_\aff^{-1}(\eta^*(R)_t)$
for $\pi_\aff : \Aff(\R^2, A^\prime(\Lambda)) \rightarrow \Aff(\R^2
/ A^\prime(\Lambda)).$ For a subgroup in $\Aff (\R^2, \Lambda)$ and
the isomorphism $\bar{\eta},$ we can similarly define the conjugate
subgroup in $\Aff (\R^2, A^\prime(\Lambda)).$

Up to conjugacy, it suffices to consider only isometries on two
torus instead of all affine isomorphisms on two torus. Let $\Isom
(\R^2/\Lambda)$ be the isometry group of $\R^2/\Lambda$ where
$\R^2/\Lambda$ delivers the metric induced by the usual metric on
$\R^2.$ And, put $\Isom_0 (\R^2/\Lambda) =$ $\Isom (\R^2/\Lambda)
\cap \Aff_0 (\R^2/\Lambda)$ and $\Isom_t (\R^2/\Lambda) =$ $\Aff_t
(\R^2/\Lambda).$

Now, we would list all finite subgroup of $\Aff (\R^2/\Lambda)$ up
to conjugacy. First, we explain for upper two rows of Table
\ref{table: finite group of aff}.

\begin{proposition} \label{proposition: finite group not D_1}
Let a finite group $R$ act naturally on $\R^2/\Lambda$ for some
$\Lambda.$ Assume that $R_0/R_t$ is nontrivial. Then, the $R$-action
is conjugate to one of the following:
\begin{enumerate}
  \item the subgroup $R$ in $\Isom(\R^2 / \Lambda)$
        such that $\Lambda_t = \Lambda^\sq$ and $R/R_t$
        is one of
        \begin{equation*}
        \Z_2, \D_{2,2}, \D_2, \Z_4, \D_4
        \end{equation*}
        where $\Lambda$ is a square sublattice of $\Lambda^\sq$
        preserved by $R/R_t,$
  \item the subgroup $R$ in $\Isom(\R^2 / \Lambda)$
        such that $\Lambda_t = \Lambda^\eq$ and $R/R_t$
        is one of
        \begin{equation*}
        \Z_2, \D_2, \D_{2,3}, \Z_3, \D_3, \D_{3,2},
        \Z_6, \D_6
        \end{equation*}
        where $\Lambda$ is an equilateral triangle sublattice
        of $\Lambda^\eq$ preserved by $R/R_t.$
\end{enumerate}
In both cases, $R$ is equal to the semidirect product $R_t \rtimes
R/R_t.$ Two $\Z_2$-actions of (1) and (2) are conjugate.

\end{proposition}

\begin{proof}
We may assume that $R$ is a subgroup of $\Isom(\R^2 / \Lambda)$ for
some lattice $\Lambda.$ So, $R/R_t \subset \orthogonal(2).$ We start
with $R_0$ instead of the whole $R.$ Since $R_0 / R_t$ is a
nontrivial subgroup of $\SO(2),$ it is a nontrivial cyclic group.
Pick $g_0$ in $R_0$ such that $g_0$ gives a generator in $R_0 /
R_t,$ and put $g_0 [\bar{x}] = [A_0 \bar{x} + \bar{c}_0]$ for some
$A_0 \in \SO(2),$ $\bar{c}_0 \in \R^2.$ Put $\bar{c}_1 = - (\id -
A_0)^{-1} \bar{c}_0,$ and put
\begin{equation*}
\bar{\eta} : \R^2 \rightarrow \R^2, \quad \bar{x} \mapsto A^\prime
\bar{x}+ \bar{c}^\prime \quad \text{with } A^\prime = \id, ~
\bar{c}^\prime = \bar{c}_1
\end{equation*}
where $\id - A_0$ is invertible because $A_0$ is a nontrivial
rotation. Here, note that $A^\prime(\Lambda) = \Lambda,$
$A^\prime(\Lambda_t) = \Lambda_t,$ and $\eta^*(R_t) = R_t.$ In the
remaining proof, we use the conjugate $\eta^*(R)$-action induced by
$\bar{\eta},$ and denote $\eta^*(R)$ just by $R.$ Since the action
by $\eta^*(g_0)$ is equal to $[A_0 \bar{x}],$ $R_0$ contains $R_0 /
R_t.$ Since $R_t \cap (R_0 / R_t)$ is trivial and $R_t \lhd R_0,$
$R_0$ is the semidirect product of $R_0 / R_t$ and $R_t.$ It is well
known that if a lattice in $\R^2$ is preserved by a nontrivial
rotation, then it is a square or an equilateral triangle lattice.
Since $R_0 / R_t$ is nontrivial and $R_0 / R_t$ preserves $\Lambda$
and $\Lambda_t,$ both lattices are simultaneously square or
equilateral triangle lattices. Here, we use the fact that a
sublattice of a square lattice can not be an equilateral triangle
lattice and vice versa. Take a suitable conformal linear isomorphism
$\bar{\eta}^\prime : \R^2 \rightarrow \R^2$ such that
$\bar{\eta}^\prime (\Lambda_t)=\Lambda^\sq$ or $\Lambda^\eq.$ In the
remaining proof, we use the conjugate $\eta^{\prime
*}(R)$-action induced by $\bar{\eta}^\prime,$ and denote $\eta^{\prime
*}(R),$ $\bar{\eta}^\prime (\Lambda),$ $\bar{\eta}^\prime (\Lambda_t)$
just by $R,$ $\Lambda,$ $\Lambda_t,$ respectively. Since $R_0 / R_t$
preserves $\Lambda_t,$ $R_0 / R_t$ becomes one of $\Z_2, \Z_4,$ or
$\Z_2, \Z_3, \Z_6$ according to $\Lambda_t = \Lambda^\sq$ or
$\Lambda^\eq,$ respectively.

If $R=R_0,$ then proof is done. So, assume that $R_0 \lneq R.$ Let
$g [\bar{x}] = [ A \bar{x} + \bar{c} ]$ be an arbitrary element of
$R - R_0.$ Then, $g^2 [\bar{x}] = [\bar{x} + A\bar{c} + \bar{c}]$
because $R / R_t$ is a dihedral group. Here, note that $A\bar{c} +
\bar{c}$ is in $\Lambda_t,$ and that $\eta^{\prime *}( \eta^*(g_0) )
= \eta^*(g_0)$ because $\bar{\eta}^\prime$ is conformal. Similar to
$g^2 [\bar{x}],$ we have $(g \eta^*(g_0))^2 [\bar{x}] = [ \bar{x} +
A A_0 \bar{c} + \bar{c} ]$ so that $A A_0 \bar{c} + \bar{c}$ is in
$\Lambda_t.$ By these two elements $A\bar{c} + \bar{c}$ and $A A_0
\bar{c} + \bar{c}$ of $\Lambda_t,$ we have $A (\id-A_0)\bar{c} \in
\Lambda_t.$ Since $A (\Lambda_t) = \Lambda_t,$ this gives us
$(\id-A_0)\bar{c} \in \Lambda_t.$ So, $\bar{c} \in \Lambda_t$
because $(\id-A_0)^{-1} = \id + A_0 + \cdots + A_0^{n-1}$ and $A_0
(\Lambda_t) = \Lambda_t$ when $A_0$ has order $n.$ Since $g
[\bar{x}] = [ A \bar{x} + \bar{c} ]$ with $\bar{c} \in \Lambda_t,$
the affine map $[A \bar{x}]$ is in $R.$ So, we obtain that $R / R_t$
is in $R.$ Since $R_t \cap (R / R_t)$ is trivial and $R_t \lhd R,$
$R$ is the semidirect product of $R / R_t$ and $R_t.$ Since $R/R_t$
preserves $\Lambda_t,$ possible actions of $R/R_t$ are very
restrictive and are conjugate to one of $\D_{2,2}, \D_2, \D_4,$ or
$\D_2, \D_{2,3}, \D_3, \D_{3,2}, \D_6$ induced by some suitable
isometry on $\R^2$ preserving $\Lambda_t$ according to $\Lambda_t =
\Lambda^\sq$ or $\Lambda^\eq,$ respectively.

Conjugacy of two $\Z_2$-actions of (1) and (2) is easy.
\end{proof}

To each finite group $R$ in the previous proposition, we assign the
two-dimensional fundamental domain $|P_R|$ with which we describe
the $R$-action.

\begin{corollary} \label{corollary: finite group not D_1}
To each finite group $R$ in the previous proposition, we assign the
area $|P_R|$ in Table \ref{table: two-dimensional fundamental
domain}. Then, the $\pi_\aff^{-1}(R)$-orbit of $|P_R|$ cover $\R^2$
and there is no interior intersection between different areas in the
orbit. Also, the $R$-orbit of $\pi(|P_R|)$ cover $\R^2 / \Lambda$
and there is no interior intersection between different areas in the
orbit.
\end{corollary}

\begin{proof}
For each $R \subset \Aff(\R^2/\Lambda)$ in the previous proposition,
$\pi_\aff^{-1} (R)$ is equal to $\Lambda_t \rtimes R/R_t$ because
$R$ is equal to $R_t \rtimes R/R_t$ and $\Lambda_t = \pi_\aff^{-1}
(R_t).$ So, $\pi_\aff^{-1} (R) = \Lambda_t \cdot R/R_t$ and proof of
the first argument is done easily case by case. And, the second
argument follows from the first.
\end{proof}

Next, we need investigate the case of trivial $R_0 / R_t.$ Before
it, we need an elementary lemma.

\begin{lemma} \label{lemma: reflection}
Let $\Lambda$ be a lattice of $\R^2.$ Let $A$ be an orient reversing
element of $\Iso (\R^2, \Lambda)$ such that $A^2 = \id.$ Then, the
group $\langle A \rangle$ in $\Iso (\R^2, \Lambda)$ is linearly
conjugate to one of $\D_1$ or $\D_{1,4}$ in $\Iso (\R^2,
\Lambda^\sq).$
\end{lemma}

\begin{proof}
We may assume that $\langle A \rangle$ acts naturally on $\R^2$ with
$A \in \orthogonal(2).$ If $\Lambda$ is a square lattice, then it
might be assumed that it is $\Lambda^\sq$ by a linear conformal map.
For $\bar{u}_1=(1,0)$ and $\bar{u}_2=(0,1),$ $A$ is determined by
$A(\bar{u}_1),$ and possible $A(\bar{u}_1)$'s are $\pm \bar{u}_1,
\pm \bar{u}_2$ because $A$ is a reflection and preserves
$\Lambda^\sq.$ By a suitable rotation, the action of $\langle A
\rangle$ is conjugate to $\D_1$ or $\D_{1,4}$ in $\Iso (\R^2,
\Lambda^\sq).$

Next, if $\Lambda$ is not a square lattice but has an orthogonal
basis, then it might be assumed that $\{ \bar{u}_1, c \cdot
\bar{u}_2 \}$ with $c>1$ is a basis of $\Lambda$ by a linear
conformal map. So, possible $A(\bar{u}_1)$'s are $\pm \bar{u}_1$ and
the action of $\langle A \rangle$ is conjugate to $\D_1$ in $\Iso
(\R^2, \Lambda^\sq)$ induced by a suitable linear isomorphism.

Before we go further, we state an elementary fact. If two elements
$\bar{w}_1$ and $\bar{w}_2$ in $\Lambda^* = \Lambda - O$ are
linearly independent and have the same smallest length, then $\{
\bar{w}_1, \bar{w}_2 \}$ is a basis of $\Lambda.$ We return to our
proof. Assume that $\Lambda$ does not have an orthogonal basis. Let
$\bar{w}_0$ be an element of $\Lambda^*$ with the smallest length.
If $A \bar{w}_0 \ne \pm \bar{w}_0,$ then $\bar{w}_0$ and $A
\bar{w}_0$ are linearly independent so that they become a basis by
the previous elementary fact. By a suitable linear isomorphism, the
action is conjugate to $\D_{1,4}$ in $\Iso (\R^2, \Lambda^\sq).$
Next, assume that $A \bar{w}_0 = \pm \bar{w}_0.$ Let $\bar{w}$ be an
element of $\Lambda - \R \cdot \bar{w}_0$ which has the smallest
distance to $\R \cdot \bar{w}_0.$ Let $L$ be the line parallel to
$\R \cdot \bar{w}_0$ containing $\bar{w}.$ Pick an element
$\bar{w}_1$ of $L \cap \Lambda$ which has the smallest length. Put
$\bar{w}_2^\prime = \mp A \bar{w}_1.$ Since $\Lambda$ does not have
an orthogonal basis, $\bar{w}_1 \ne \bar{w}_2^\prime.$ Also,
$|\bar{w}_1 - \bar{w}_2^\prime| = |\bar{w}_0|$ because $\bar{w}_1$
has the smallest length in $L \cap \Lambda.$ And, the convex hull of
$\{ O, \bar{w}_1, \bar{w}_2^\prime \}$ contains no element of
$\Lambda$ except $O, \bar{w}_1, \bar{w}_2^\prime$ so that $\{
\bar{w}_1, \bar{w}_2^\prime \}$ is a basis for $\Lambda.$ If we put
$\bar{w}_2 = A \bar{w}_1 = \mp \bar{w}_2^\prime,$ then $\{
\bar{w}_1, \bar{w}_2 \}$ is a basis of $\Lambda$ with the same
length such that $A \bar{w}_1 = \bar{w}_2.$ By a suitable linear
isomorphism, the action of $\langle A \rangle$ is conjugate to the
action of the $\D_{1,4}$ in $\Iso (\R^2, \Lambda^\sq).$ Therefore,
we obtain a proof.
\end{proof}

To deal with the case of trivial $R_0/R_t,$ we define some
notations. When $R/R_t = \D_1$ or $\D_{1,4},$ let $\bar{l}_0$ be
$(1,0)$ or $(1,1),$ and let $\bar{l}_0^\perp$ be $(0,1)$ or
$(-1/2,1/2),$ respectively. For those two cases, let $L$ in $\R^2$
be the line fixed by $R/R_t,$ and let $L^\perp$ be the line
perpendicular to $L$ passing through $O.$ Here, $\bar{l}_0$ is an
element of $(\Lambda^\sq)^* \cap L$ with the smallest length, and
the length of $\bar{l}_0^\perp$ is the distance from $L$ to
$\Lambda^\sq-L.$ For a line $L^\prime$ and a vector
$\bar{c}^\prime$(possibly zero) in $\R^2$ with $\bar{c}^\prime
\parallel L^\prime,$ the composition of the reflection in $L^\prime$
and the translation by $\bar{c}^\prime$ is called the \textit{glide}
through $L^\prime$ and $\bar{c}^\prime.$ Now, we can explain for
lower three rows of Table \ref{table: finite group of aff}.

\begin{proposition} \label{proposition: finite group D_1}
Let a finite group $R$ act naturally on $\R^2/\Lambda$ for some
$\Lambda.$ Assume that $R_0/R_t$ is trivial. Then, the $R$-action is
conjugate to one of the following:
\begin{enumerate}
  \item the subgroup $R$ in $\Isom (\R^2 / \Lambda)$ with
  $R = R_t$ and $\Lambda_t = \Lambda^\sq$ where
  $\Lambda$ is equal to $\langle (m_1, 0), (0, m_2) \rangle$
  in $\Lambda^\sq$
  for some $m_1, m_2 \in \N.$

  \item the subgroup $R$ in $\Isom (\R^2 / \Lambda)$ with
  $R = \langle R_t, [A\bar{x}+\bar{c}] \rangle$ and $\Lambda_t = \Lambda^\sq$
  where $R/R_t = \langle A \rangle = \D_1$ or $\D_{1,4}.$
  Here, $\bar{c} \in L^\perp$ or $\frac 1 2 \bar{l}_0+ L^\perp,$ and $A\bar{x}+\bar{c}$ is
  the glide through $\frac 1 2 \bar{c}+L$ and $\frac 1 2 (A\bar{c}+\bar{c}).$
\end{enumerate}
\end{proposition}

\begin{proof}
First, if $R/R_t$ is trivial, then $R=R_t$ and we easily obtain (1)
by a suitable linear isomorphism. For (2), assume that $R/R_t$ is
nontrivial. Then, $R/R_t$ in $\Aff(\R^2, \Lambda)$ is order two, and
there exists a linear isomorphism $\bar{\eta} : \R^2 \rightarrow
\R^2$ such that $\bar{\eta}(\Lambda_t) = \Lambda^\sq$ and the
conjugate action on $\R^2$ induced by $\bar{\eta}$ is $\D_1$ or
$\D_{1,4}$ in $\Aff (\R^2, \Lambda^\sq)$ by Lemma \ref{lemma:
reflection}. In the remaining proof, we use the conjugate
$\eta^*(R)$-action on $\R^2 / \bar{\eta}(\Lambda)$ induced by
$\bar{\eta},$ and denote $\eta^*(R),$ $\bar{\eta} (\Lambda),$
$\bar{\eta} (\Lambda_t)$ just by $R,$ $\Lambda,$ $\Lambda_t,$
respectively. Then, $\Lambda_t=\Lambda^\sq$ is obtained by
(\ref{equation: conjugate action}). Pick an element $g$ in $R-R_t$
which is $[A\bar{x}+\bar{c}]$ with $\langle A \rangle = \D_1$ or
$\D_{1,4}.$ By $g^2 \in R_t,$ we obtain $A\bar{c}+\bar{c} \in
\Lambda^\sq$ which is fixed by $A,$ i.e. $A\bar{c}+\bar{c} \in L.$
Since $[A\bar{x}+(\bar{c}+\bar{\lambda})] \in R$ for each
$\bar{\lambda} \in \Lambda^\sq,$
$A(\bar{c}+\bar{\lambda})+(\bar{c}+\bar{\lambda})$ is also in
$\Lambda^\sq$ so that we may assume that $A\bar{c}+\bar{c}$ is equal
to $0$ or $\bar{l}_0.$ Then, $\bar{c}$ should be in $L^\perp$ or
$\frac 1 2 \bar{l}_0 + L^\perp$ because $A$ is the reflection in
$L,$ and (2) is easily obtained.
\end{proof}

\begin{figure}[ht]
\begin{center}

\mbox{

\subfigure[$R/R_t=D_1$ with $\bar{c}=(0,0)$]{
\begin{pspicture}(-1.5,-2)(4,2)\footnotesize

\pspolygon[fillstyle=solid,fillcolor=lightgray,linestyle=none](0,0)(0,0.5)(1,0.5)(1,0)(0,0)

\psaxes[labels=none,ticks=none]{->}(0,0)(-1.5,-2)(3.5,2)

\psgrid[gridwidth=2pt, subgridwidth=2pt, gridcolor=black,
subgridcolor=black, subgriddiv=2, griddots=1, subgriddots=1,
gridlabels=7pt](0,0)(-0.4,-1.4)(2.4,1.4)

\psline[linewidth=1.5pt](-1,0)(3,0)

\uput[ur](3,0){$\frac 1 2 \bar{c}+L$}

\psline[linewidth=1.5pt](0,0)(0,0.5)(1,0.5)(1,0)(0,0)

\psline[linewidth=0.3pt](0,0)(0,0.5)(2,0.5)(2,0)(0,0)
\psline[linewidth=0.3pt](1,0.5)(1,0)

\rput(0,0.5){ \psline[linewidth=0.3pt](0,0)(0,0.5)(2,0.5)(2,0)(0,0)
\psline[linewidth=0.3pt](1,0.5)(1,0) }

\rput(0,-0.5){ \psline[linewidth=0.3pt](0,0)(0,0.5)(2,0.5)(2,0)(0,0)
\psline[linewidth=0.3pt](1,0.5)(1,0) }

\rput(0,-1){ \psline[linewidth=0.3pt](0,0)(0,0.5)(2,0.5)(2,0)(0,0)
\psline[linewidth=0.3pt](1,0.5)(1,0) }

\end{pspicture}
}

\subfigure[$R/R_t=D_{1,4}$ with $\bar{c}=(0,0)$]{
\begin{pspicture}(-1.5,-1)(4,3)\footnotesize

\pspolygon[fillstyle=solid,fillcolor=lightgray,linestyle=none](0,0)(0.5,0.5)(0,1)(-0.5,0.5)(0,0)

\psaxes[labels=none,ticks=none]{->}(0,0)(-1.5,-1)(3.5,3)

\psgrid[gridwidth=2pt, subgridwidth=2pt, gridcolor=black,
subgridcolor=black, subgriddiv=2, griddots=1, subgriddots=1,
gridlabels=7pt](0,0)(-0.5,-0.5)(2.4,2.4)

\psline[linewidth=1.5pt](-1,-1)(3,3)

\uput[r](2.7,2.5){$\frac 1 2 \bar{c}+L$}

\psline[linewidth=1.5pt](0,0)(0.5,0.5)(0,1)(-0.5,0.5)(0,0)

\psline[linewidth=0.3pt](0,0)(1,1)(0.5,1.5)(-0.5,0.5)(0,0)
\psline[linewidth=0.3pt](0.5,0.5)(0,1)

\rput(1,1){
\psline[linewidth=0.3pt](0,0)(1,1)(0.5,1.5)(-0.5,0.5)(0,0)
\psline[linewidth=0.3pt](0.5,0.5)(0,1) }

\rput(0.5,-0.5){
\psline[linewidth=0.3pt](0,0)(1,1)(0.5,1.5)(-0.5,0.5)(0,0)
\psline[linewidth=0.3pt](0.5,0.5)(0,1) }

\rput(1.5,0.5){
\psline[linewidth=0.3pt](0,0)(1,1)(0.5,1.5)(-0.5,0.5)(0,0)
\psline[linewidth=0.3pt](0.5,0.5)(0,1) }

\end{pspicture}
} }

\mbox{

\subfigure[$R/R_t=D_1$ with $\bar{c}=(1/2,0)$]{
\begin{pspicture}(-1.5,-2)(4,2)\footnotesize

\pspolygon[fillstyle=solid,fillcolor=lightgray,linestyle=none](0.25,0.5)(0.75,0.5)(0.75,-0.5)(0.25,-0.5)(0.25,0.5)

\psaxes[labels=none,ticks=none]{->}(0,0)(-1.5,-2)(3.5,2)

\psgrid[gridwidth=2pt, subgridwidth=2pt, gridcolor=black,
subgridcolor=black, subgriddiv=2, griddots=1, subgriddots=1,
gridlabels=7pt](0,0)(-0.4,-1.4)(2.4,1.4)

\psline[linewidth=1.5pt](-1,0)(3,0)

\uput[ur](3,0){$\frac 1 2 \bar{c}+L$}

\psline[linewidth=1.5pt](0.25,0.5)(0.75,0.5)(0.75,-0.5)(0.25,-0.5)(0.25,0.5)

\rput(0.25,0){

\psline[linewidth=0.3pt](0,0.5)(1,0.5)(1,-0.5)(0,-0.5)(0,0.5)
\psline[linewidth=0.3pt](0.5,0.5)(0.5,-0.5)

\rput(1,0){
\psline[linewidth=0.3pt](0,0.5)(1,0.5)(1,-0.5)(0,-0.5)(0,0.5)
\psline[linewidth=0.3pt](0.5,0.5)(0.5,-0.5) }

\rput(0,1){
\psline[linewidth=0.3pt](0,0.5)(1,0.5)(1,-0.5)(0,-0.5)(0,0.5)
\psline[linewidth=0.3pt](0.5,0.5)(0.5,-0.5) }

\rput(1,1){
\psline[linewidth=0.3pt](0,0.5)(1,0.5)(1,-0.5)(0,-0.5)(0,0.5)
\psline[linewidth=0.3pt](0.5,0.5)(0.5,-0.5) }

\rput(0,-1){
\psline[linewidth=0.3pt](0,0.5)(1,0.5)(1,-0.5)(0,-0.5)(0,0.5)
\psline[linewidth=0.3pt](0.5,0.5)(0.5,-0.5) }

\rput(1,-1){
\psline[linewidth=0.3pt](0,0.5)(1,0.5)(1,-0.5)(0,-0.5)(0,0.5)
\psline[linewidth=0.3pt](0.5,0.5)(0.5,-0.5) } }

\end{pspicture}
}

\subfigure[$R/R_t=D_{1,4}$ with $\bar{c}=(1/2,1/2)$]{
\begin{pspicture}(-1.5,-1)(4,3)\footnotesize

\pspolygon[fillstyle=solid,fillcolor=lightgray,linestyle=none](0,0.5)(0.5,1)(1,0.5)(0.5,0)(0,0.5)

\psaxes[labels=none,ticks=none]{->}(0,0)(-1.5,-1)(3.5,3)

\psgrid[gridwidth=2pt, subgridwidth=2pt, gridcolor=black,
subgridcolor=black, subgriddiv=2, griddots=1, subgriddots=1,
gridlabels=7pt](0,0)(-0.5,-0.5)(2.4,2.4)

\psline[linewidth=1.5pt](-1,-1)(3,3)

\uput[r](2.7,2.5){$\frac 1 2 \bar{c}+L$}

\psline[linewidth=1.5pt](0,0.5)(0.5,1)(1,0.5)(0.5,0)(0,0.5)

\rput(0.25,0.25){

\psline[linewidth=0.3pt](-0.25,0.25)(0.75,1.25)(1.25,0.75)(0.25,-0.25)(-0.25,0.25)
\psline[linewidth=0.3pt](0.25,0.75)(0.75,0.25)

\rput(1,1){
\psline[linewidth=0.3pt](-0.25,0.25)(0.75,1.25)(1.25,0.75)(0.25,-0.25)(-0.25,0.25)
\psline[linewidth=0.3pt](0.25,0.75)(0.75,0.25) }

\rput(-0.5,0.5){
\psline[linewidth=0.3pt](-0.25,0.25)(0.75,1.25)(1.25,0.75)(0.25,-0.25)(-0.25,0.25)
\psline[linewidth=0.3pt](0.25,0.75)(0.75,0.25) }

\rput(0.5,1.5){
\psline[linewidth=0.3pt](-0.25,0.25)(0.75,1.25)(1.25,0.75)(0.25,-0.25)(-0.25,0.25)
\psline[linewidth=0.3pt](0.25,0.75)(0.75,0.25) }

\rput(0.5,-0.5){
\psline[linewidth=0.3pt](-0.25,0.25)(0.75,1.25)(1.25,0.75)(0.25,-0.25)(-0.25,0.25)
\psline[linewidth=0.3pt](0.25,0.75)(0.75,0.25) }

\rput(1.5,0.5){
\psline[linewidth=0.3pt](-0.25,0.25)(0.75,1.25)(1.25,0.75)(0.25,-0.25)(-0.25,0.25)
\psline[linewidth=0.3pt](0.25,0.75)(0.75,0.25) } }

\end{pspicture}
} }

\end{center}
\caption{\label{figure: polygonal area with reflection} Some
polygonal areas with glides}
\end{figure}
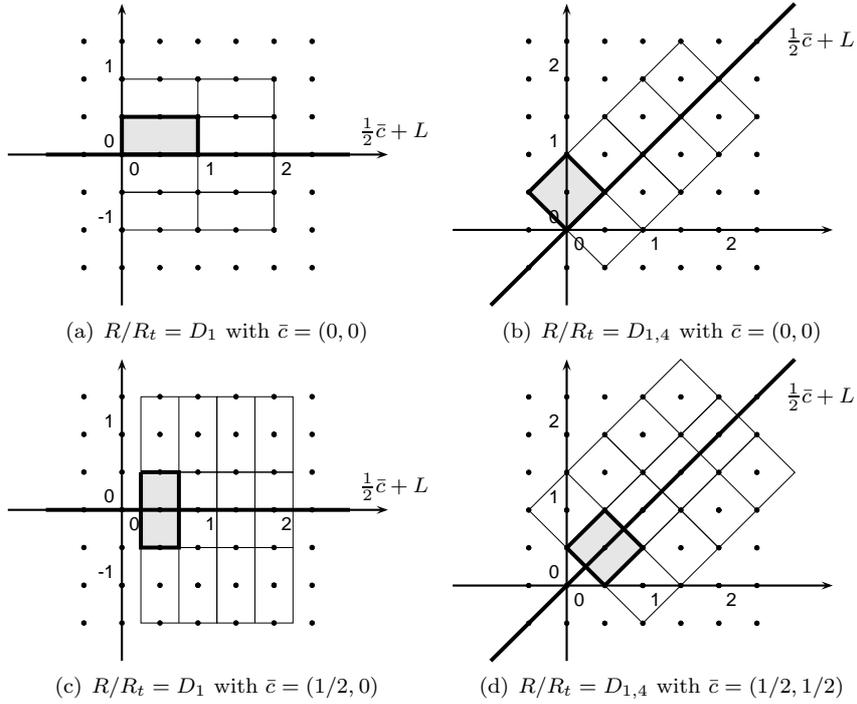


\begin{corollary} \label{corollary: finite group D_1}
To a finite group $R$ in the previous proposition, we assign the
two-dimensional fundamental domain $|P_R|$ in Table \ref{table:
two-dimensional fundamental domain}. Then, the
$\pi_\aff^{-1}(R)$-orbit of $|P_R|$ cover $\R^2$ and there is no
interior intersection between different areas in the orbit. Also,
the $R$-orbit of $\pi(|P_R|)$ cover $\R^2 / \Lambda$ and there is no
interior intersection between different areas in the orbit. In
Figure \ref{figure: polygonal area with reflection}, we illustrate
gray colored $|P_R|$ for some cases.
\end{corollary}

\begin{proof}
If $R=R_t,$ then proof is easy. Otherwise, $\pi_\aff^{-1} (R)$ is
equal to $\langle \Lambda_t, A\bar{x}+\bar{c} \rangle,$ and proof is
done similar to Corollary \ref{corollary: finite group not D_1}.
\end{proof}

In summary, any finite subgroup $R$ in $\Aff(\R^2/\Lambda)$ for some
$\Lambda$ is conjugate to one of groups in Table \ref{table: finite
group of aff}.

\section{one-dimensional $\rho(G_\chi)$ case}
 \label{section: nonzero dimensional case}

In this section, we prove Theorem \ref{main: nonzero-dimensional}.
First, we would list all one-dimensional subgroups of $\Aff(\R^2 /
\Lambda)$ for some $\Lambda$ up to conjugacy. Let $S^1$ be the
rotation group $\{ [\bar{x} +(t,0)] \in \Aff(\R^2 / \Lambda^\sq) ~|~
t \in [0,1] \}.$ Let $M(a)$ be the matrix $\left(
             \begin{array}{cc}
               1 & a  \\
               0 & -1 \\
             \end{array}
           \right)$
for $a \in \Z.$

\begin{lemma}
 \label{lemma: one-dimensional subgroups}
Let a compact one-dimensional group $R$ act naturally on
$\R^2/\Lambda$ for some $\Lambda.$ Then, $R$ is conjugate to one of
the following subgroups in $\Aff (\R^2 / \Lambda^\sq)$:
\begin{enumerate}
  \item $\langle S^1, [\bar{x}+(0, \frac 1 n)] \rangle$ for
  some $n \in \N,$
  \item $\langle S^1, [\bar{x}+(0, \frac 1 n)],
  [-\bar{x}+(0, t_1)] \rangle$ for some $n \in \N,$ $t_1 \in [0,1],$
  \item $\langle K, [M(a) \bar{x} + (0, t_2)] \rangle$
  or $\langle K, [-M(a) \bar{x} + (0, t_2)] \rangle$ for some $a \in
  \Z,$ $t_2 \in [0,1]$ where $K$ is equal to one of (1), (2)
  and $[\pm M(a) \bar{x} + (0, t_2)]^2 \in K.$
\end{enumerate}
\end{lemma}

\begin{proof}
If $R$ is connected, then it is easy that $R$ is conjugate to $S^1$
in $\Aff(\R^2 / \Lambda^\sq).$ So, if $R$ is not connected, then we
may assume that $R$ is a subgroup in $\Aff (\R^2 / \Lambda^\sq)$
whose identity component is $S^1.$ For arbitrary $g = [A \bar{x} +
\bar{c}] \in R$ and $z \in S^1,$
\begin{equation}
\tag{*} g z [ \bar{x} ] = (g z g^{-1}) g [ \bar{x} ]
\end{equation}
where $g z g^{-1} \in S^1$ by normality of $S^1$ in $R.$ Since
$S^1$-orbits in $\R^2 / \Lambda^\sq$ are horizontal, the right term
of (*) is horizontal when $g z g^{-1}$ moves in $S^1.$ So, $g$ moves
the horizontal $S^1$-orbit $z [ \bar{x} ]$ for $z \in S^1$ to a
horizontal $S^1$-orbit. From this, we can observe that $A$ sends the
$x$-axis to the $x$-axis itself, i.e. the $(2,1)$ entry of $A$
should be zero. Moreover, observe that $A$ has a finite order. From
this and $A(\Lambda^\sq)=\Lambda^\sq,$ we can show that possible
$A$'s are $\pm \id, \pm M(a)$ for $a \in \Z$ by simple calculation.
Also, we may assume that the first entry of $\bar{c}$ is 0 because
$zg \in R$ for each $z \in S^1.$ By these arguments, we would show
that possible $R$ is one of (1)$\sim$(3).

First, assume that the $R$-action is orient preserving. So, $A$ of
each $g = [A \bar{x} + \bar{c}]$ in $R$ is equal to $\pm \id.$ Since
the subgroup of elements of the form $[\bar{x} + (0,t)]$ in $R$ is
cyclic, it is generated by an element of the form $[\bar{x} +
(0,1/n)]$ for some $n \in \N.$ So, if $R$ contains no element of the
form $[-\bar{x} + (0, t_1)],$ then $R$ is equal to (1). And, if $R$
contains an element of the form $[-\bar{x} + (0, t_1)],$ then $R$ is
equal to (2).

Next, assume that the $R$-action is not orient preserving. Then, $R$
contains an element $[\pm M(a) \bar{x} + (0, t_2)]$ for some $a \in
\Z,$ $t_2 \in [0,1].$ So, $R$ is equal to $\langle R_0, [\pm M(a)
\bar{x} + (0, t_2)] \rangle.$ From this, we obtain this lemma
because $R_0$ is equal to (1) or (2).
\end{proof}

For each action of the above lemma, we can find a circle in
two-torus by which equivariance of the torus is simply expressed.

\begin{proposition}
 \label{proposition: affine circle}
Let a compact one-dimensional group $R$ act naturally on
$\R^2/\Lambda$ for some $\Lambda.$ Then, there exists a circle $C$
in the torus such that the map $R \times_{R_C} C \rightarrow \R^2 /
\Lambda,$ $[g, x] \mapsto gx$ for $g \in R, x \in C$ is an
equivariant isomorphism where $R_C$ is the subgroup of $R$
preserving $C.$
\end{proposition}

\begin{proof}
It suffices to find $C$ such that $R \cdot C = \R^2 / \Lambda$ and
$g C \cap C = \varnothing$ for any $g \in R - R_C.$ As in the
previous lemma, we may assume that $R$ is contained in $\Aff
(\R^2/\Lambda^\sq),$ and that its identity component is $S^1.$ When
the $R$-action is orient-preserving, we may assume that $R$ is equal
to (1) or (2) in Lemma \ref{lemma: one-dimensional subgroups}. Let
$C$ be the circle which is the image of the $y$-axis by $\pi.$ Then,
it is checked that $R$-orbits of $C$ are all translations of $C,$
and from this it is easily shown that $C$ becomes a wanted circle.
Therefore, we obtain a proof.

When the $R$-action is not orient-preserving, we prove it only for
the case of $R = \langle K, [ M(a) \bar{x} + (0, t_2)] \rangle$ when
$K$ is equal to (2) of Lemma \ref{lemma: one-dimensional subgroups}.
Proof for the other case is similar. Let $C$ be the image of $y = -
\frac 2 a x$ (or $x=0$ when $a=0$) by $\pi.$ Then, it is checked
that $R$-orbits of $C$ are all translations of $C,$ and from this it
is easily shown that $C$ becomes a wanted circle. Therefore, we
obtain a proof.
\end{proof}

By using this proposition, we can prove Theorem \ref{main:
nonzero-dimensional}.

\begin{proof}[Proof of Theorem \ref{main: nonzero-dimensional}]
We may assume that $\Lambda = \Lambda^\sq$ and $\rho(G_\chi)$ is one
of groups in Lemma \ref{lemma: one-dimensional subgroups}. The torus
$\R^2/\Lambda$ is equivariantly isomorphic to $G_\chi
\times_{(G_\chi)_C} C$ for the circle $C$ of Proposition
\ref{proposition: affine circle}. So, the map
\begin{equation*}
\Vect_{(G_\chi)_C} (C, \chi) \rightarrow \Vect_{G_\chi}
(\R^2/\Lambda, \chi), \quad [F] \mapsto [ G_\chi \times_{(G_\chi)_C}
F]
\end{equation*}
is the inverse of our map.
\end{proof}

\section{$\R^2/\Lambda$ as the quotient of an equivariant simplicial complex
for a finite subgroup of $\Aff (\R^2/\Lambda)$}
 \label{section: equivariant simplicial complex}

In this section, we consider $\R^2/\Lambda$ as the quotient of the
underlying space of an equivariant simplicial complex which will be
denoted by $\lineK_\rho.$ Given a finite subgroup $R$ of
$\Aff(\R^2/\Lambda)$ in Table \ref{table: finite group of aff} and
its natural action on $\R^2/\Lambda,$ we investigate equivariance of
$\R^2/\Lambda$ by calculating isotropy groups at some points of it.

Let $\F_R$ be a set of $\bar{g} \cdot |P_R|$'s for $\bar{g} \in
\pi_\aff^{-1} (R)$ such that
\begin{enumerate}
  \item $|P_R| \in \F_R,$
  \item $\bigcup_{\bar{g} \cdot |P_R| \in \F_R} ~ \pi(\bar{g} \cdot |P_R|) =
\R^2/\Lambda,$
  \item $\pi(\bar{g} \cdot |P_R|) ~ \cap ~ \pi(\bar{g}^\prime \cdot |P_R|)$
has no interior point for any different $\bar{g} \cdot |P_R|$ and
$\bar{g}^\prime \cdot |P_R|$ in $\F_R.$
\end{enumerate}
Such an $\F_R$ exists by Corollary \ref{corollary: finite group not
D_1} and Corollary \ref{corollary: finite group D_1}. We pick an
$\F_R$ for each $R.$ We would consider the disjoint union
$\amalg_{\bar{g} \cdot |P_R| \in \F_R} \bar{g} \cdot |P_R|$ as the
underlying space of a two-dimensional simplicial complex. As in
Introduction, we allow that a face of a two-dimensional simplicial
complex need not be a triangle, and we consider an $n$-gon as the
simplicial complex with one face, $n$ edges, $n$ vertices. Then,
denote by $\lineK_R$ the natural simplicial complex structure of
$\amalg_{\bar{g} \cdot |P_R| \in \F_R} ~ \bar{g} \cdot |P_R|.$ We
denote simply by $\pi$ the quotient map from $|\lineK_R|$ to
$\R^2/\Lambda$ sending each point $\bar{x}$ in some $\bar{g} \cdot
|P_R| \subset |\lineK_R|$ to $\pi(\bar{x})$ in $\R^2/\Lambda$ where
$\bar{x}$ is regarded as a point in $\R^2.$ By definition,
$\pi|_{(\bar{g} \cdot |P_R|)^\circ}$ is bijective for each $\bar{g}
\cdot |P_R| \in \F_R$ so that the $R$-action on $\R^2/\Lambda$
induces the $R$-action on $\lineK_R$ and $|\lineK_R|$ such that
$\pi$ is equivariant. In general, $\pi$ does not induce the
equivariant simplicial complex structure on $\R^2/\Lambda$ from
$\lineK_\rho$ as the example illustrated in Figure \ref{figure: i
and p}.

Next, we define some notations on $\lineK_R.$ We use notations
$\bar{v}, \bar{e}, \bar{f}$ to denote a vertex, an edge, a face of
$\lineK_R,$ respectively. We use the notation $\bar{x}$ to denote an
arbitrary point of $|\lineK_R|.$ And, we use notations $v,$ $x,$
$|e|,$ $b(e),$ $|f|,$ $b(f)$ to denote images $\pi (\bar{v}), \pi
(\bar{x}), \pi(|\bar{e}|), \pi (b(\bar{e})),$ $\pi(|\bar{f}|),
\pi(b(\bar{f})),$ respectively. Denote by $\bar{f}^{-1}$ the face of
$\lineK_R$ such that $|\bar{f}^{-1}|=|P_R|.$ In Introduction, we
have already defined vertices $\bar{v}^i$'s and edges $\bar{e}^i$'s
of $\bar{f}^{-1}.$ For simplicity, denote $\pi(\bar{v}^i),$
$\pi(|\bar{e}^i|),$ $\pi(b(\bar{e}^i)),$ $\pi(b(\bar{f}^{-1}))$ by
$v^i,$ $|e^i|,$ $b(e^i),$ $b(f^{-1}),$ respectively. Define the
integer $j_R$ as the cardinality of $\pi^{-1}(v^i)$ for $i \in
\Z_{i_R}.$ Let $B$ be the set of barycenters of faces in
$\lineK_\rho$ on which $R$ acts transitively by definition of
$\F_R,$ and $B$ is confused with $\pi(B).$

We would calculate isotropy subgroups at some points of
$\R^2/\Lambda.$ For this, we define some notations on isotropy
subgroups. For each $x \in \R^2/\Lambda,$ $R_x$ is considered as a
subgroup of $\Iso (T_x ~ \R^2/\Lambda).$ Since the tangent space
$T_x ~ \R^2/\Lambda$ at each $x \in \R^2/\Lambda$ inherits the
vector space structure and the usual coordinate system from $\R^2,$
$\Iso(\R^2)$ might be identified with $\Iso(T_x ~ \R^2/\Lambda).$
With this identification, $R/R_t$ is also regarded as a subgroup of
$\Iso (T_x ~ \R^2/\Lambda)$ according to context. Here, we can
observe that $R_x \subset R/R_t$ in $\Iso (T_x ~ \R^2/\Lambda)$ for
each $x \in \R^2/\Lambda.$ Similarly, we can define subgroups $\Z_n,
\D_n, \D_{n,l}$ in $\Iso (T_x ~ \R^2/\Lambda).$ For the calculation,
we give the natural simplicial complex structure on $\R^2$ by
considering $\R^2$ as the union $\bigcup_{\bar{g} \in
\pi_\aff^{-1}(R)} ~ \bar{g} |P_R|,$ and we denote it by
$\complexK_{\R^2}.$

First, we calculate $R_{b(f^{-1})}.$ For a lattice $\Lambda^\prime$
in $\R^2,$ denote by $\Area (\Lambda^\prime)$ the area of the
parallelogram spanned by two basis vectors in $\Lambda^\prime.$
Then, the area of $\R^2/\Lambda$ is equal to $\Area(\Lambda_t) \cdot
|R_t|$ so that we have
\begin{equation}
\label{equation: R_t and B} |R_t| = \frac {\Area(|P_R|)}
{\Area(\Lambda_t)} \cdot |B|
\end{equation}
because the area of $\R^2/\Lambda$ is equal to $\Area(|P_R|) \cdot
|B|$ where $\Area(|P_R|)$ is the area of $|P_R|.$

\begin{proposition} \label{proposition: isotropy of face}
For each $R$ in Table \ref{table: finite group of aff},
$R_{b(f^{-1})}$ is listed in Table \ref{table: isotropy of face}.
\end{proposition}

\begin{table}[!ht]
{\footnotesize
\begin{tabular}{c|c||c|c}
 $R/R_t$                & $\Lambda_t$    & $\Area(|P_R|)/\Area(\Lambda_t)$ & $R_{b(f^{-1})}$      \\
\hhline{=|=#=|=}
 $\Z_2$                 & $\Lambda^\sq$  & 1/2                           & $\langle \id \rangle$  \\
 $\D_{2,2}$             &                & 1/4                           & $\langle \id \rangle$  \\
 $\D_2$                 &                & 1/4                           & $\langle \id \rangle$  \\
 $\Z_4$                 &                & 1/4                           & $\langle \id \rangle$  \\
 $\D_4$                 &                & 1/4                           & $\D_{1,4}$ \\

\hline

 $\D_2$                 & $\Lambda^\eq$  & 1/4                           & $\langle \id \rangle$  \\
 $\D_{2,3}$             &                & 1/4                           & $\langle \id \rangle$  \\
 $\D_3$                 &                & 1/2                           & $\Z_3$  \\
 $\Z_6$                 &                & 1/2                           & $\Z_3$  \\
 $\D_6$                 &                & 1/2                           & $\D_{3,2}$  \\

\hline

 $\Z_3, \D_{3,2}$       & $\Lambda^\eq$  & 1                             & $R/R_t$  \\

\hline

 $\langle \id \rangle$  & $\Lambda^\sq$  & 1                             & $\langle \id \rangle$  \\

\hline
 $\D_1, \D_{1,4},$      & $\Lambda^\sq$  & 1/2                           & $\langle \id \rangle$ \\
 $A\bar{c}+\bar{c}=0$   &                &                               &     \\

\hline

 $\D_1, \D_{1,4},$      & $\Lambda^\sq$  & 1/2                           & $\langle \id \rangle$ \\
 $A\bar{c}+\bar{c}=\bar{l}_0$             &                &                               &     \\

\end{tabular}}
\caption{\label{table: isotropy of face} $R_{b(f^{-1})}$}
\end{table}

\begin{proof}
As in proofs of Corollary \ref{corollary: finite group not D_1},
\ref{corollary: finite group D_1}, we can calculate $|B/R_t| =
|\pi^{-1}(B) / \Lambda_t|$ case by case because $\pi^{-1}(B)$ is the
set of barycenters of faces of $\complexK_{\R^2}.$ By using this, we
can calculate $R_{b(f^{-1})}$ in the below.

By definition of $|P_R|$ in Table \ref{table: two-dimensional
fundamental domain}, $\Area(|P_R|)/\Area(\Lambda_t)=$ 1, 1/2, or
1/4. In the case when $\Area(|P_R|)/\Area(\Lambda_t)=1,$ it is
observed that $R_t$ acts transitively on $B,$ and the formula
(\ref{equation: R_t and B}) gives $|R_t|=|B|.$ From these, $R_t$
acts freely and transitively on $B.$ So, we can obtain that $R =
R_{b(f^{-1})} \cdot R_t$ and $R_{b(f^{-1})} \cap R_t = \langle \id
\rangle.$ That is, $R_{b(f^{-1})} \cong R/R_t$ because $R_t$ is
normal in $R.$ Also, since $R_{b(f^{-1})}$ is contained in $R/R_t$
as subgroups of $\Iso (T_{b(f^{-1})} ~ \R^2/\Lambda),$ $R/R_t$ and
$R_{b(f^{-1})}$ are equal as subgroups of $\Iso (T_{b(f^{-1})} ~
\R^2/\Lambda).$

When $\Area(|P_R|)/\Area(\Lambda_t)=1/2,$ the formula
(\ref{equation: R_t and B}) gives $|R_t|= \frac 1 2 |B|$ and $R_t$
acts freely and non-transitively on $B.$ In the below, we will use
the notation $R^\prime$ to denote a subgroup of $R$ such that
$R^\prime$ acts freely and transitively on $B.$ If such a subgroup
$R^\prime$ exists, then $R=R_{b(f^{-1})} \cdot R^\prime$ and
$R_{b(f^{-1})} \cap R^\prime = \langle \id \rangle$ so that
$|R_{b(f^{-1})}|=|R|/|R^\prime|.$ Moreover, $R_{b(f^{-1})} \cong
R/R^\prime$ if $R^\prime$ is normal in $R.$ In the case when $R/R_t
= \Z_2, \Z_6, \D_6,$ put $R^\prime = R_t \rtimes \Z_2$ which is
normal in $R.$ It is easily observed that $R^\prime$ acts freely and
transitively on $B.$ From this, $R_{b(f^{-1})} \cong$ $R/R^\prime
\cong$ $(R/R_t)/\Z_2.$ In fact, it is checked that the
$R_{b(f^{-1})}$-action on $T_{b(f^{-1})} \R^2/\Lambda$ is $\langle
\id \rangle, \Z_3, \D_{3,2}$ according to $R/R_t = \Z_2, \Z_6,
\D_6,$ respectively. In the case when $R/R_t=\D_3,$ put $R^\prime =
\langle b, R_t \rangle$ so that $R_{b(f^{-1})}$ has order 3 and
$R_{b(f^{-1})}=\Z_3.$ In the case when $R/R_t = \D_1, \D_{1,4},$ we
obtain $|R| = |B|$ from $|\D_1| = |\D_{1,4}| = 2$ and $|R_t|= \frac
1 2 |B|$ so that $R$ acts freely and transitively on $B,$ i.e.
$R_{b(f^{-1})}$ is trivial.

When $\Area(|P_R|)/\Area(\Lambda_t)=1/4,$ the formula
(\ref{equation: R_t and B}) gives $|R_t|= \frac 1 4 |B|$ and $R_t$
acts freely and non-transitively on $B.$ In cases when $R/R_t =
\D_{2,2}, \D_2, \Z_4$ with $\Lambda_t=\Lambda^\sq$ and $R/R_t =
\D_2, \D_{2,3}$ with $\Lambda_t=\Lambda^\eq,$ we have $|R/R_t|=4$ so
that $|R| = |B|.$ Since $R$ acts freely and transitively on $B,$
$R_{b(f^{-1})}$ is trivial. In the case when $R/R_t = \D_4,$ we
similarly have $|R| = 2|B|.$ Since $R$ acts transitively on $B,$ we
have $|R_{b(f^{-1})}|=2$ and it is checked that
$R_{b(f^{-1})}=\D_{1,2}.$
\end{proof}

Now, we calculate $R_{v^i}$ and $R_{b(e^i)}.$ Denote by
$\bar{\mathcal{V}}$ the set of vertices of $\lineK_R,$ and by
$\mathcal{V}$ the set $\pi(\bar{\mathcal{V}}).$ For two points $x$
and $x^\prime$ in $\R^2/\Lambda,$ $x \sim x^\prime$ means that $x$
and $x^\prime$ are in the same $R$-orbit.

\begin{proposition} \label{proposition: isotropy of vertex}
For each $R$ in Table \ref{table: finite group of aff}, $R_{v^i}$ is
listed in Table \ref{table: isotropy of vertex}. In $\sim$ entry, we
list all $v^i$'s in the same $R$-orbit. By $v^i \sim v^{i^\prime},$
we mean that all $v^i$'s are in an orbit.
\end{proposition}

\begin{table}[!ht]
{\footnotesize
\begin{tabular}{c|c||c|c|c}
 $R/R_t$               & $\Lambda_t$    & $|\mathcal{V}/R|$ & $\sim$                                         & $R_{v^i}$ \\
\hhline{=|=#=|=|=}
 $\Z_2$                & $\Lambda^\sq$  & 2                 & $v^0 \sim v^3, ~ v^1 \sim v^2$                 & $R/R_t$  \\
 $\D_{2,2}$            &                & 3                 & $v^1 \sim v^2$                                 & $R_{v^0}=R_{v^3}=R/R_t,$ $R_{v^1}=R_{v^2}=\D_{1,-4}$ \\
 $\D_2$                &                & 4                 &                                                & $R/R_t$  \\
 $\Z_4$                &                & 3                 & $v^1 \sim v^3$                                 & $R_{v^0}=R_{v^2}=R/R_t,$ $R_{v^1}=R_{v^3}=\Z_2$ \\
 $\D_4$                &                & 3                 & $v^1 \sim v^3$                                 & $R_{v^0}=R_{v^2}=R/R_t,$ $R_{v^1}=R_{v^3}=\D_2$ \\

\hline
 $\D_2$                & $\Lambda^\eq$  & 3                 & $v^2 \sim v^3$                                 & $R_{v^0}=R_{v^1}=R/R_t,$ $R_{v^2}=R_{v^3}=\D_1$  \\
 $\D_{2,3}$            &                & 3                 & $v^1 \sim v^2$                                 & $R_{v^0}=R_{v^3}=R/R_t,$ $R_{v^1}=R_{v^2}=\D_{1,-3}$  \\
 $\D_3$                &                & 1                 & $v^i \sim v^{i^\prime}$                        & $R/R_t$  \\
 $\Z_6$                &                & 1                 & $v^i \sim v^{i^\prime}$                        & $R/R_t$  \\
 $\D_6$                &                & 1                 & $v^i \sim v^{i^\prime}$                        & $R/R_t$  \\

\hline
 $\Z_3, \D_{3,2}$      & $\Lambda^\eq$  & 2                 & $v^i \sim v^{i+2}$                             & $R/R_t$  \\

\hline
 $\langle \id \rangle$ & $\Lambda^\sq$  & 1                 & $v^i \sim v^{i^\prime}$                        & $\langle \id \rangle$  \\

\hline
 $\D_1,$               & $\Lambda^\sq$  & 2                 & $v^0 \sim v^3, ~ v^1 \sim v^2$                 & $R/R_t$ \\
 $A\bar{c}+\bar{c}=0$  &                &                   &                                                &     \\

\hline
 $\D_{1,4},$           & $\Lambda^\sq$  & 2                 & $v^0 \sim v^2, ~ v^1 \sim v^3$                 & $R/R_t$ \\
 $A\bar{c}+\bar{c}=0$  &                &                   &                                                &     \\

\hline
 $\D_1,$               & $\Lambda^\sq$  & 1                 & $v^i \sim v^{i^\prime}$                        & $\langle \id \rangle$ \\
 $A\bar{c}+\bar{c}=\bar{l}_0$ &         &                   &                                                &     \\

\hline
 $\D_{1,4},$           & $\Lambda^\sq$  & 2                 & $v^0 \sim v^2, ~ v^1 \sim v^3$                 & $R/R_t$ \\
 $A\bar{c}+\bar{c}=\bar{l}_0$ &         &                   &                                                &     \\

\end{tabular}}
\caption{\label{table: isotropy of vertex} $R_{v^i}$}
\end{table}

\begin{proof}
Since $|\bar{\mathcal{V}}|=i_R |B|$ and $|\bar{\mathcal{V}}| = j_R
|\mathcal{V}|,$ we obtain
\begin{equation}
\label{equation: R_t and V} |R_t| = \frac {\Area(|P_R|)}
{\Area(\Lambda_t)} \cdot \frac {j_R} {i_R} \cdot |\mathcal{V}|
\end{equation}
by (\ref{equation: R_t and B}). As in proofs of Corollary
\ref{corollary: finite group not D_1}, \ref{corollary: finite group
D_1}, we can calculate
\begin{equation*}
|\mathcal{V}/R| = |\pi^{-1}(\mathcal{V}) / \pi_\aff^{-1}(R)| \quad
\text{ and } \quad |\mathcal{V}/R_t| = |\pi^{-1}(\mathcal{V}) /
\Lambda_t|
\end{equation*}
case by case because $\pi^{-1}(\mathcal{V})$ is the set of vertices
of $\complexK_{\R^2}.$ By using this, we can calculate $R_{v^i}$ in
the below.

In the case of $R/R_t = \Z_2$ with $\Lambda_t = \Lambda^\sq,$ the
formula (\ref{equation: R_t and V}) gives $|R_t|= \frac 1 2
|\mathcal{V}|.$ From this, we conclude that $R_t$ acts freely on
each $R_t$-orbit in $\mathcal{V}$ because $|\mathcal{V}/R_t|=2.$ Of
course, $R_t$ acts transitively on each $R_t$-orbit in
$\mathcal{V}.$ Since $|\mathcal{V}/R|=2,$ each $R_t$-orbit is also
$R$-invariant. From this, we have $R= R_{v^i} \cdot R_t$ and
$R_{v^i} \cap R_t = \langle \id \rangle$ so that $R_{v^i}=R/R_t.$
Here, $v^0, v^3$ are in the same orbit, and $v^1, v^2$ are in the
same orbit.

In the case of $R/R_t =\D_{2,2}$ with $\Lambda_t = \Lambda^\sq,$ the
formula (\ref{equation: R_t and V}) gives $|R_t|= \frac 1 4
|\mathcal{V}|.$ From this, we conclude that $R_t$ acts freely on
each $R_t$-orbit in $\mathcal{V}$ because $|\mathcal{V}/R_t|=4$ and
$R_t$ acts transitively on each $R_t$-orbit in $\mathcal{V}.$ Since
$|\mathcal{V}/R|=3$ and $v^1, v^2$ are in the same orbit, $R_t$ acts
freely and transitively on both $R_t$-orbits of $v^0, v^3$ in
$\mathcal{V}$ which are $R$-invariant. So, $R_{v^i}=R/R_t$ for $i=0,
3.$ But, each $R_t$-orbit of $v^i$ for $i=1, 2$ is not
$R$-invariant. If we take the index 2 subgroup $S$ of $R$ preserving
each $R_t$-orbit of $v^i$ for $i=1, 2,$ then $S= R_{v^i} \cdot R_t$
and $R_{v^i} \cap R_t = \langle \id \rangle$ so that $|R_{v^i}|$
should be 2 for $i=1, 2,$ and it is checked that $R_{v^i}=\D_{1,-4}$
for $i=1, 2.$

In cases when $R/R_t= \D_2, \Z_4, \D_4$ with $\Lambda_t =
\Lambda^\sq,$ the formula (\ref{equation: R_t and V}) gives $|R_t|=
\frac 1 4 |\mathcal{V}|.$ From this, $R_t$ acts freely and
transitively on each $R_t$-orbit in $\mathcal{V}$ because
$|\mathcal{V}/R_t|=4.$ Since $|\mathcal{V}/R|=4$ for $R/R_t=\D_2,$
$R_{v^i}=R/R_t.$ Since $|\mathcal{V}/R|=3$ and $v^1, v^3$ are in an
$R$-orbit for $R/R_t=\Z_4, \D_4,$ we have $R_{v^i}=R/R_t$ for
$i=0,2.$ Also, it is checked that $R_{v^i}= \Z_2$ for $R/R_t=\Z_4,
i=1,3.$ And, $R_{v^i}= \D_2$ for $R/R_t=\D_4, i=1,3.$

In the case of $R/R_t = \D_2, \D_{2,3}$ with
$\Lambda_t=\Lambda^\eq,$ the formula (\ref{equation: R_t and V})
gives $|R_t|= \frac 1 4 |\mathcal{V}|.$ From this, $R_t$ acts freely
and transitively on each $R_t$-orbit in $\mathcal{V}$ because
$|\mathcal{V}/R_t|=4.$ Since $|\mathcal{V}/R|=3$ and $v^2, v^3$ are
in a $R$-orbit for $R/R_t=\D_2,$ we have $R_{v^i}=R/R_t$ for
$i=0,1,$ and it is checked that $R_{v^i}= \D_1$ for $i=2,3.$ Since
$|\mathcal{V}/R|=3$ and $v^1, v^2$ are in an $R$-orbit for
$R/R_t=\D_{2,3},$ we have $R_{v^i}=R/R_t$ for $i=0,3,$ and it is
checked that $R_{v^i}= \D_{1,-3}$ for $i=1,2.$

In the case of $|P_R|=P_3^\eq,$ i.e. $R/R_t = \D_3, \Z_6, \D_6$ with
$\Lambda=\Lambda^\eq,$ the formula (\ref{equation: R_t and V}) gives
$|R_t|=|\mathcal{V}|.$ From this, we conclude that $R_t$ acts freely
on $\mathcal{V}$ because $R_t$ acts transitively on $\mathcal{V}.$
So, $R_{v^i}=R/R_t.$

In the case of $|P_R|=P_6^\eq,$ i.e. $R/R_t = \Z_3, \D_{3,2}$ with
$\Lambda=\Lambda^\eq,$ the formula (\ref{equation: R_t and V}) gives
$|R_t|= \frac 1 2 |\mathcal{V}|.$ Since $|\mathcal{V}/R_t|=2,$ $R_t$
acts freely and transitively on each $R_t$-orbit in $\mathcal{V}.$
Also, since $|\mathcal{V}/R|=2,$ $R_{v^i}=R/R_t.$

In the case of $R/R_t= \langle \id \rangle,$ $R_{v^i}=\langle \id
\rangle.$

In the case of $R/R_t=\D_1, \D_{1,4}$ with $A\bar{c}+\bar{c}=0,$ the
formula (\ref{equation: R_t and V}) gives $|R_t|=\frac 1 2
|\mathcal{V}|.$ From this, we conclude that $R_{v^i}=R/R_t$ because
$|\mathcal{V}/R|=|\mathcal{V}/R_t|=2.$

In the case of $R/R_t=\D_1, \D_{1,4}$ with
$A\bar{c}+\bar{c}=\bar{l}_0,$ the formula (\ref{equation: R_t and
V}) gives $|R_t|=\frac 1 2 |\mathcal{V}|.$ From this, we conclude
that $R_{v^i}$ is trivial because $R$ acts transitively on
$\mathcal{V}.$
\end{proof}

\begin{table}[!ht]
{\footnotesize
\begin{tabular}{c|c||c|l}
$R/R_t$                  & $\Lambda_t$     & $|\mathcal{E}/R|$   & $R_{b(e^i)}$      \\
\hhline{=|=#=|=}
 $\Z_2$                  & $\Lambda^\sq$   & 3                   & $R_{b(e^0)}=R_{b(e^2)}=\langle \id \rangle, R_{b(e^1)}=R_{b(e^3)}=R/R_t$         \\
 $\D_{2,2}$              &                 & 4                   & $R_{b(e^0)}=R_{b(e^2)}=\D_{1,-4}, R_{b(e^1)}=R_{b(e^3)}=\D_{1,4}$         \\
 $\D_2$                  &                 & 4                   & $R_{b(e^0)}=R_{b(e^2)}=\D_{1,2}, R_{b(e^1)}=R_{b(e^3)}=\D_1$     \\
 $\Z_4$                  &                 & 2                   & $\langle \id \rangle$         \\
 $\D_4$                  &                 & 2                   & $R_{b(e^0)}=R_{b(e^2)}=\D_{1,2}, R_{b(e^1)}=R_{b(e^3)}=\D_1$   \\
\hline
 $\D_2$                  & $\Lambda^\eq$   & 4                   & $R_{b(e^1)}=R_{b(e^3)}=\D_1,$ $R_{b(e^0)}=\D_{1,2},$ $R_{b(e^2)}=\Z_2$   \\
 $\D_{2,3}$              &                 & 4                   & $R_{b(e^0)}=R_{b(e^2)}=\D_{1,-3},$ $R_{b(e^1)}=\Z_2,$ $R_{b(e^3)}=\D_{1,6}$  \\
 $\D_3$                  &                 & 1                   & $R_{b(e^0)}=\D_{1,3},$ $R_{b(e^1)}=\D_{1,-3},$ $R_{b(e^2)}=\D_1$  \\
 $\Z_6$                  &                 & 1                   & $\Z_2$  \\
 $\D_6$                  &                 & 1                   & $R_{b(e^0)}=\D_{2,3/2},$ $R_{b(e^1)}=\D_{2,-3/2},$ $R_{b(e^2)}=\D_2$  \\
\hline
 $\Z_3$                  & $\Lambda^\eq$   & 1                   & $\langle \id \rangle$  \\
 $\D_{3,2}$              &                 & 1                   & $R_{b(e^0)}=\D_{1,2},$ $R_{b(e^1)}=\D_{1,6},$ $R_{b(e^2)}=\D_{1,-6},$  \\
                         &                 &                     & $R_{b(e^3)}=\D_{1,2},$ $R_{b(e^4)}=\D_{1,6},$ $R_{b(e^5)}=\D_{1,-6}$  \\
\hline
 $\langle \id \rangle$   & $\Lambda^\sq$   & 2                   & $\langle \id \rangle$  \\
\hline
 $\D_1,$                 & $\Lambda^\sq$   & 3                   & $R_{b(e^0)}=R_{b(e^2)}=\langle \id \rangle, R_{b(e^1)}=R_{b(e^3)}=R/R_t$  \\
 $A\bar{c}+\bar{c}=0$                &                 &                     &    \\
\hline
 $\D_{1,4},$             & $\Lambda^\sq$   & 3                   & $R_{b(e^0)}=R_{b(e^2)}=\langle \id \rangle, R_{b(e^1)}=R_{b(e^3)}=R/R_t$ \\
 $A\bar{c}+\bar{c}=0$                &                 &                     &    \\
\hline
 $\D_1,$                 & $\Lambda^\sq$   & 2                   & $\langle \id \rangle$  \\
 $A\bar{c}+\bar{c}=\bar{l}_0$              &                 &                     &    \\
\hline
 $\D_{1,4},$             & $\Lambda^\sq$   & 3                   & $R_{b(e^0)}=R_{b(e^2)}=\langle \id \rangle, R_{b(e^1)}=R_{b(e^3)}=R/R_t$ \\
 $A\bar{c}+\bar{c}=\bar{l}_0$              &                 &                     &    \\

\end{tabular}}
\caption{\label{table: isotropy of edge} $R_{b(e^i)}$}
\end{table}

Denote by $\bar{\mathcal{E}}$ the set of barycenters of edges in
$\lineK_R,$ and by $\mathcal{E}$ the set $\pi(\bar{\mathcal{E}}).$

\begin{proposition} \label{proposition: isotropy of edge}
For each $R$ in Table \ref{table: finite group of aff}, $R_{b(e^i)}$
is listed in Table \ref{table: isotropy of edge}. And, each
$R_{b(e^i)}$ satisfies $R_{b(e^i)} \cap R_{b(f^{-1})} = \langle \id
\rangle$ except $R/R_t = \D_6.$
\end{proposition}

\begin{proof}
As in proofs of Corollary \ref{corollary: finite group not D_1},
\ref{corollary: finite group D_1}, we can calculate
\begin{equation*}
|\mathcal{E}/R| = |\pi^{-1}(\mathcal{E}) / \pi_\aff^{-1}(R)| \quad
\text{ and } \quad |\mathcal{E}/R_t| = |\pi^{-1}(\mathcal{E}) /
\Lambda_t|
\end{equation*}
case by case because $\pi^{-1}(\mathcal{E})$ is the set of
barycenters of edges in $\complexK_{\R^2}.$ By this, we obtain
$|\mathcal{E}/R|$-entry of Table \ref{table: isotropy of edge}.

Since $R_{b(e^i)}$ preserves $|e^i|,$ $R_{b(e^i)} \subset \D_{2,l}$
in $\Iso (T_{b(e^i)} ~ \R^2/\Lambda)$ for some $l \in \Q^*.$ In
cases when $R/R_t \ne \D_1, \D_{1,4},$ we can show that a matrix $A
\in \D_{2,l}$ is in $R_{b(e^i)} \subset \Iso (T_{b(e^i)} ~
\R^2/\Lambda)$ if and only if $A b(\bar{e}^i) \in b(\bar{e}^i) +
\Lambda_t$ in $\R^2$ because $R = R_t \rtimes R/R_t,$ i.e. each
element in $R$ is of the form $[A\bar{x}+\bar{c}]$ with $\bar{c} \in
\Lambda_t.$ By using this, we can calculate $R_{b(e^i)}$ in Table
\ref{table: isotropy of edge}. Similarly, we obtain $R_{b(e^i)}$ in
Table \ref{table: isotropy of edge} in cases when $R/R_t = \D_1,
\D_{1,4}.$


The second statement is checked case by case by Table \ref{table:
isotropy of face}.
\end{proof}

Now, we explain for one-dimensional fundamental domain. A closed
subset $\bar{D}_R$ of $|\lineK_R^{(1)}|$ is called a
\textit{one-dimensional fundamental domain} if the orbit of
$\pi(\bar{D}_R)$ covers $\pi(|\lineK_R^{(1)}|)$ and $\bar{D}_R$ is a
minimal set satisfying the property.


\begin{proposition} \label{proposition: D_R}
Let $R$ be a group in Table \ref{table: finite group of aff}. For
each $i \in \Z_{i_R},$ $R_{b(e^i)}$ fixes $|e^i|$ if and only if
$\Z_2$ is not contained in $R_{b(e^i)}.$ And, $\bar{D}_R$ listed in
Table \ref{table: one-dimensional fundamental domain} is a
one-dimensional fundamental domain.
\end{proposition}

\begin{proof}
First, we prove the first statement. We can check that it is true
for the case of $R/R_t = \D_6.$ So, we may assume that $R/R_t
\ne\D_6.$ Sufficiency is easy. For necessity, assume that
$R_{b(e^i)}$ does not fix $|e^i|.$ So, $R_{b(e^i)}$ is nontrivial.
Since $R_{b(e^i)}$ preserves $|e^i|,$ $R_{b(e^i)} \subset \D_{2, l}$
for some $l \in \Q^*$ where we assume that $a_{2l} b$ fixes $|e^i|.$
So, $R_{b(e^i)} = \langle a_{2l} b \rangle$ or $\langle -a_{2l} b
\rangle$ because $\Z_2$ is not contained in $R_{b(e^i)}.$ Since
$R_{b(e^i)}$ does not fix $|e^i|,$ $R_{b(e^i)} = \langle -a_{2l} b
\rangle.$ But, $-a_{2l} b$ is orient reversing, and preserves
$|e^i|.$ So, it fixes $\pi([b(\bar{f}^{-1}), b(\bar{e}^i)]),$ i.e.
$-a_{2l} b \in R_{b(e^i)} \cap R_{b(f^{-1})}$ which is trivial by
Proposition \ref{proposition: isotropy of edge}. This is a
contradiction.

Since $R$ acts transitively on $B,$ each edge of $\lineK_R$ is in
the $R$-orbit of some $\bar{e}^i.$ So, we may assume that any
one-dimensional fundamental domain is contained in $|\mathcal{E}/R|$
edges of $|P_R|,$ say $\bar{e}^{\prime j}$ for $j=1, \cdots,
|\mathcal{E}/R|.$ We would select suitable $|\mathcal{E}/R|$ lines
in $|\lineK_R^{(1)}|$ whose union is a candidate for a
one-dimensional fundamental domain. For each $j,$ if $R_{b(e^{\prime
j})}$ fixes $|e^{\prime j}|,$ then pick the whole line
$|\bar{e}^{\prime j}|,$ otherwise pick a half of $|\bar{e}^{\prime
j}|$ because a half of $|e^{\prime j}|$ should be sent to the other
half by $R_{b(e^{\prime j})}.$ If we define $\bar{D}_R$ as the union
of picked lines, then $\bar{D}_R$ satisfies the definition of
one-dimensional fundamental domain. In Table \ref{table:
one-dimensional fundamental domain}, we have assigned such a
$\bar{D}_R$ to each $R.$

In the case of $R/R_t = \D_3,$ $\bar{D}_R$ is not chosen in this way
by a technical reason. But, we can check that $\bar{D}_R$ for the
case in Table \ref{table: one-dimensional fundamental domain}
becomes a one-dimensional fundamental domain.
\end{proof}

By using the above calculations, we can calculate $R_{|e^i|}$ and
$R_{C(\bar{x})}$ as promised in Introduction.

\begin{lemma} \label{lemma: isotropy at a point in an edge}
For each $\bar{x}$ in $|\bar{e}^i|$ such that $\bar{x} \ne
\bar{v}^i, \bar{v}^{i+1}, b(\bar{e}^i),$ we have $R_{\bar{x}} =
\langle \id \rangle.$ For $x=\pi(\bar{x}),$ $R_x = R_{|e^i|}.$ More
precisely,
\begin{enumerate}
  \item if $R_{b(e^i)} \subset \Z_2,$ then $R_x$ is trivial,
  \item if $R_{b(e^i)} = \D_{1,l}$ for some $l \in \Q^*,$
  then $R_x=\D_{1,l},$
  \item if $R/R_t = \D_6,$ then $R_x = \D_{1, 3}, \D_{1, -3}, \D_1$
  for $i=0, 1, 2,$ respectively.
\end{enumerate}
\end{lemma}

\begin{proof}
First, $R_{\bar{x}}$ fixes the whole $|\bar{f}^{-1}|$ so that
$R_{\bar{x}} = \langle \id \rangle.$

Next, $R_x$ fixes the whole $|e^i|,$ especially $R_x \subset
R_{b(e^i)}.$ From this, (1) is obtained. By Proposition
\ref{proposition: D_R}, we obtain (2). (3) is checked by direct
calculation.
\end{proof}

Next, we calculate $R_{C(\bar{x})}.$ If we put $\bar{C}(\bar{x}) =
[\bar{x}, b(P_\rho)]$ for $\bar{x} \in
\partial |P_\rho|,$ then $R_{C(\bar{x})}$ is equal to
\begin{equation*}
R_{\bar{C}(\bar{x})} = R_{b(\bar{f}^{-1})} \cap R_{\bar{x}} =
R_{\bar{x}} \subset R_{b(\bar{f}^{-1})}
\end{equation*}
for the $R$-action on $|\lineK_R|.$ Also, since $R_{b(\bar{f}^{-1})}
= R_{b(f^{-1})},$ we can easily obtain $R_{C(\bar{x})}$ by Table
\ref{table: isotropy of face}.

We give a remark on injectivity of $\pi|_{\bar{D}_R}.$

\begin{remark} \label{remark: in orbit}
\begin{enumerate}
  \item If $R/R_t \ne \langle \id \rangle, \D_1, \D_{1,4},$ then
it can be checked case by case that any two vertices $\bar{v}^i \ne
\bar{v}^{i^\prime}$ in $\bar{D}_R$ satisfy $v^i \ne v^{i^\prime}.$
In this check, it is helpful to note that $v^i = v^{i^\prime}$
implies $v^i \sim v^{i^\prime}.$ In the remaining cases, $\bar{v}^i
\ne \bar{v}^{i^\prime}$ need not imply $v^i \ne v^{i^\prime}$
according to $\Lambda.$
  \item For any two points $\bar{x} \ne \bar{x}^\prime$ in $\bar{D}_R,$
  if they are not vertices, then $\pi(\bar{x}) \ne
  \pi(\bar{x}^\prime)$ by the method by which we choose $\bar{D}_R.$
\end{enumerate}
\end{remark}

\section{equivariant clutching construction}
 \label{section: preclutching map}

Let a compact Lie group $G_\chi$ act affinely and not necessarily
effectively on $\R^2/\Lambda$ through a homomorphism $\rho: G_\chi
\rightarrow \Aff (\R^2/\Lambda).$ Assume that $\rho(G_\chi)=R$ for
some $R$ in Table \ref{table: finite group of aff}. From now on, we
simply denote $\lineK_{\rho(G_\chi)}, P_{\rho(G_\chi)},
i_{\rho(G_\chi)}, j_{\rho(G_\chi)}, \cdots$ by $\lineK_\rho, P_\rho,
i_\rho, j_\rho, \cdots,$ respectively. In this section, we would
consider an equivariant vector bundle over $\R^2/\Lambda$ as an
equivariant clutching construction of an equivariant vector bundle
over $|\lineK_\rho|.$ For this, we would define equivariant
clutching map and its generalization preclutching map. And, we state
an equivalent condition under which a preclutching map be an
equivariant clutching map. Before these, we need introduce notations
on some relevant simplicial complices. Since we should deal with
various cases at the same time, these notations are necessary.

Denote by $\lineL_\rho$ and $\hatL_\rho$ the 1-skeleton
$\lineK_\rho^{(1)}$ of $\lineK_\rho$ and the disjoint union
$\amalg_{\bar{e} \in \lineL_\rho} ~ \bar{e},$ respectively. Then,
$\lineL_\rho$ is a subcomplex of $\lineK_\rho,$ and can be regarded
as a quotient of $\hatL_\rho.$ These are expressed as two simplicial
maps
\begin{equation*}
\imath_\lineL : \lineL_\rho \rightarrow \lineK_\rho, \quad p_\lineL
: \hatL_\rho \rightarrow \lineL_\rho
\end{equation*}
where $\imath_\lineL$ is the inclusion, and $p_\lineL$ is the
quotient map whose preimage of each vertex and edge of $\lineL_\rho$
consists of two vertices and one edge of $\hatL_\rho,$ respectively.
Two maps on underlying spaces are denoted by
\begin{equation*}
\imath_{|\lineL|} : |\lineL_\rho| \rightarrow |\lineK_\rho|, \quad
p_{|\lineL|} : |\hatL_\rho| \rightarrow |\lineL_\rho|.
\end{equation*}
Here, we give an example of $\imath_{|\lineL|}$ and $p_{|\lineL|}$
when $\rho(G_\chi) = R$ with $R/R_t = \langle \id \rangle$ and $R_t
= \langle \id \rangle,$ i.e. $\Lambda_t = \Lambda = \Lambda^\sq.$ In
this case, $\lineK_\rho = P_\rho$ with $|P_\rho| = P_4^\sq$ by Table
\ref{table: two-dimensional fundamental domain}. Also, $\lineL_\rho$
is the 1-skeleton of $P_4^\sq,$ and $\hatL_\rho$ is disjoint union
of four edges of $\lineL_\rho.$ These are illustrated in Figure
\ref{figure: i and p}. In Figure \ref{figure: simple form}, one more
example of $\lineL_\rho$ and $\hatL_\rho$ is illustrated where its
notations are introduced in the below. The $G_\chi$-action on
$\R^2/\Lambda$ naturally induces actions on these relevant
simplicial complices $\lineK_\rho, \lineL_\rho, \hatL_\rho.$

\begin{figure}[ht]
\begin{center}
\begin{pspicture}(-4,-1)(7,1.5)\footnotesize

\psline[linearc=1]{->}(4.5, 0.5)(5.5, 0.5)

\uput[u](5, 0.5){$\pi$}

\uput[r](5.75, 0.5){$\R^2/\Lambda^\sq$}

\pspolygon[fillstyle=solid,fillcolor=lightgray](3,0)(3,1)(4,1)(4,0)
\psline(0,0)(0,1)(1,1)(1,0)(0,0)

\psline(-2,0)(-2,1) \psline(-3.5,0)(-3.5,1)
\psline(-2.25,1.25)(-3.25,1.25) \psline(-2.25,-0.25)(-3.25,-0.25)

\psline[linearc=1]{->}(1.5, 0.5)(2.5, 0.5)

\uput[u](2, 0.5){$\imath_{|\lineL|}$}

\psline[linearc=1]{->}(-1.5, 0.5)(-0.5, 0.5)

\uput[u](-1, 0.5){$p_{|\lineL|}$}

\uput[d](-2.75,-0.5){$|\hatL_\rho|$}

\uput[d](0.5,-0.5){$|\lineL_\rho|$}

\uput[d](3.5,-0.5){$|\lineK_\rho|$}
\end{pspicture}
\end{center}
\caption{\label{figure: i and p} An example of $\imath_{|\lineL|}$
and $p_{|\lineL|}$ when $R/R_t = \langle \id \rangle$ and $\Lambda_t
= \Lambda = \Lambda^\sq.$}
\end{figure}

Now, we introduce notations on simplices of relevant simplicial
complices. We use notations $\hat{v}$ and $\hat{e}$ to denote a
vertex and an edge of $\hatL_\rho,$ respectively. And, we use the
notation $\hat{x}$ to denote an arbitrary point in $|\hatL_\rho|.$
For simplicity, we often denote by $\bar{v}, \bar{e}, \bar{x}$
images $p_\lineL (\hat{v}), p_\lineL (\hat{e}), p_{|\lineL|}
(\hat{x}),$ respectively. Two edges $\hat{e}, \hat{e}^\prime$ of
$\hatL_\rho$ (and their images $\bar{e}, \bar{e}^\prime$ of
$\lineL_\rho$) are called \textit{adjacent} if $\hat{e} \ne
\hat{e}^\prime$ and $\pi( |p_\lineL (\hat{e})| ) = \pi( |p_\lineL
(\hat{e}^\prime)| ).$ In the example of Figure \ref{figure: i and
p}, two parallel edges are adjacent. And, two faces (not necessarily
different) $\bar{f}, \bar{f}^\prime$ of $\lineK_\rho$ are called
\textit{adjacent} if there are two adjacent edges $\bar{e} \in
\bar{f},$ $\bar{e}^\prime \in \bar{f}^\prime.$ In Introduction, we
have already defined $\bar{v}^i$'s and $\bar{e}^i$'s. For $i \in
\Z_{i_\rho},$ denote by $\bar{f}^i$ the unique face of $\lineK_\rho$
which contains the edge $\bar{e}^\prime$ adjacent to $\bar{e}^i.$
Faces $\bar{f}^i$'s need not be different. So far, we have finished
defining $\bar{f}^i, \bar{e}^i, \bar{v}^i$ for superscript $i \in
\Z_{i_\rho}.$

\begin{figure}[ht]
\begin{center}
\mbox{ \subfigure[$|\lineK_\rho|$]{
\begin{pspicture}(-5.5,-3.5)(5.5,3.5)\footnotesize

\pspolygon[fillstyle=solid,fillcolor=lightgray](-1,-1)(-1,1)(1,1)(1,-1)(-1,-1)

\pspolygon[fillstyle=solid,fillcolor=lightgray,linestyle=none](3,-1)(2,-1)(2,1)(3,1)
\rput(-5,0){\pspolygon[fillstyle=solid,fillcolor=lightgray,linestyle=none](3,-1)(2,-1)(2,1)(3,1)
}
\pspolygon[fillstyle=solid,fillcolor=lightgray,linestyle=none](-1,3)(-1,2)(1,2)(1,3)
\rput(0,-5){\pspolygon[fillstyle=solid,fillcolor=lightgray,linestyle=none](-1,3)(-1,2)(1,2)(1,3)
}
\pspolygon[fillstyle=solid,fillcolor=lightgray,linestyle=none](2,3)(2,2)(3,2)(3,3)
\rput(0,-5){\pspolygon[fillstyle=solid,fillcolor=lightgray,linestyle=none](2,3)(2,2)(3,2)(3,3)
}
\rput(-5,0){\pspolygon[fillstyle=solid,fillcolor=lightgray,linestyle=none](2,3)(2,2)(3,2)(3,3)
}
\rput(-5,-5){\pspolygon[fillstyle=solid,fillcolor=lightgray,linestyle=none](2,3)(2,2)(3,2)(3,3)
}

\psline(3,-1)(2,-1)(2,1)(3,1) \psline(-3,-1)(-2,-1)(-2,1)(-3,1)
\psline(-1,3)(-1,2)(1,2)(1,3) \psline(-1,-3)(-1,-2)(1,-2)(1,-3)

\psline(2,3)(2,2)(3,2) \psline(-2,3)(-2,2)(-3,2)
\psline(2,-3)(2,-2)(3,-2) \psline(-2,-3)(-2,-2)(-3,-2)

\psdots[dotsize=2pt](0,0) \psdots[dotsize=2pt](3,0)
\psdots[dotsize=2pt](-3,0) \psdots[dotsize=2pt](0,3)
\psdots[dotsize=2pt](0,-3) \psdots[dotsize=2pt](3,3)
\psdots[dotsize=2pt](-3,3) \psdots[dotsize=2pt](3,-3)
\psdots[dotsize=2pt](-3,-3)

\uput[d](0,0){$b(\bar{f}^{-1})$} \uput[dr](3,0){$b(\bar{f}^2)$}
\uput[dl](-3,0){$b(\bar{f}^0)$} \uput[u](0,3){$b(\bar{f}^1)$}
\uput[d](0,-3){$b(\bar{f}^3)$}

\uput[l](-0.9,0){$\bar{e}^0$} \uput[r](0.9,0){$\bar{e}^2$}
\uput[u](0,0.9){$\bar{e}^1$} \uput[d](0,-0.9){$\bar{e}^3$}

\uput[l](-1.9,0){$c(\bar{e}^0)$} \uput[r](1.9,0){$c(\bar{e}^2)$}
\uput[u](0,1.9){$c(\bar{e}^1)$} \uput[d](0,-1.9){$c(\bar{e}^3)$}

\uput[d](-1,-0.9){$\bar{v}^0$} \uput[u](-1,0.9){$\bar{v}^1$}
\uput[u](1,0.9){$\bar{v}^2$} \uput[d](1,-0.9){$\bar{v}^3$}

\uput[r](-2.1,-1){$\bar{v}_1^0$} \uput[u](-2,-2.1){$\bar{v}_2^0$}
\uput[l](-0.9,-2){$\bar{v}_3^0$} \uput[l](-0.9,-2){$\bar{v}_3^0$}

\uput[l](2.1,1){$\bar{v}_1^2$} \uput[d](2,2.1){$\bar{v}_2^2$}
\uput[r](0.9,2){$\bar{v}_3^2$}

\uput[r](-2.1,1){$\bar{v}_3^1$} \uput[d](-2,2.1){$\bar{v}_2^1$}
\uput[l](-0.9,2){$\bar{v}_1^1$}

\uput[l](2.1,-1){$\bar{v}_3^3$} \uput[u](2,-2.1){$\bar{v}_2^3$}
\uput[r](0.9,-2){$\bar{v}_1^3$}
\end{pspicture}
} }

\mbox{ \subfigure[$|\hatL_\rho|$]{
\begin{pspicture}(-5,-4)(5,4)\footnotesize

\psline(-1.5,-0.5)(-1.5,0.5) \psline(1.5,-0.5)(1.5,0.5)
\psline(-0.5,1)(0.5,1) \psline(-0.5,-1)(0.5,-1)

\psline(5,-1)(4,-1) \psline(3,-0.5)(3,0.5) \psline(4,1)(5,1)

\psline(-1.5,4)(-1.5,3) \psline(-0.5,2.5)(0.5,2.5)
\psline(1.5,3)(1.5,4)

\psline(-5,-1)(-4,-1) \psline(-3,-0.5)(-3,0.5) \psline(-4,1)(-5,1)

\psline(-1.5,4)(-1.5,3) \psline(-0.5,2.5)(0.5,2.5)
\psline(1.5,3)(1.5,4)

\psline(-1.5,-4)(-1.5,-3) \psline(-0.5,-2.5)(0.5,-2.5)
\psline(1.5,-3)(1.5,-4)

\psline(5,2.5)(4,2.5) \psline(3,3)(3,4)

\psline(-5,2.5)(-4,2.5) \psline(-3,3)(-3,4)

\psline(5,-2.5)(4,-2.5) \psline(3,-3)(3,-4)

\psline(-5,-2.5)(-4,-2.5) \psline(-3,-3)(-3,-4)

\psdots[dotsize=3pt](-1.5,-0.5)(-1.5,0.5)(1.5,-0.5)(1.5,0.5)
\psdots[dotsize=3pt](-0.5,1)(0.5,1)(-0.5,-1)(0.5,-1)
\psdots[dotsize=3pt](4,-1)(3,-0.5)(3,0.5)(4,1)
\psdots[dotsize=3pt](-1.5,3)(-0.5,2.5)(0.5,2.5) (1.5,3)
\psdots[dotsize=3pt](-4,-1)(-3,-0.5)(-3,0.5)(-4,1)
\psdots[dotsize=3pt](-1.5,3)(-0.5,2.5)(0.5,2.5) (1.5,3)
\psdots[dotsize=3pt](-1.5,-3)(-0.5,-2.5)(0.5,-2.5)(1.5,-3)
\psdots[dotsize=3pt](4,2.5)(3,3)(-4,2.5)(-3,3)
\psdots[dotsize=3pt](4,-2.5)(3,-3)(-4,-2.5)(-3,-3)

\uput[r](-1.6,0){$\hat{e}^0$} \uput[l](1.6,0){$\hat{e}^2$}
\uput[d](0,1.1){$\hat{e}^1$} \uput[u](0,-1.1){$\hat{e}^3$}

\uput[l](-2.9,0){$c(\hat{e}^0)$} \uput[r](2.9,0){$c(\hat{e}^2)$}
\uput[u](0,2.4){$c(\hat{e}^1)$} \uput[d](0,-2.4){$c(\hat{e}^3)$}

\uput[d](-0.5,-0.9){$\hat{v}_{0,-}^0$}
\uput[d](-1.5,-0.4){$\hat{v}_{0,+}^0$}
\uput[d](-3,-0.4){$\hat{v}_{1,-}^0$}
\uput[d](-4,-0.9){$\hat{v}_{1,+}^0$}
\uput[u](-4,-2.6){$\hat{v}_{2,-}^0$}
\uput[u](-3,-3.1){$\hat{v}_{2,+}^0$}
\uput[u](-1.5,-3.1){$\hat{v}_{3,-}^0$}
\uput[u](-0.5,-2.6){$\hat{v}_{3,+}^0$}

\uput[u](-0.5,0.9){$\hat{v}_{0,+}^1$}
\uput[u](-1.5,0.4){$\hat{v}_{0,-}^1$}
\uput[u](-3,0.4){$\hat{v}_{3,+}^1$}
\uput[u](-4,0.9){$\hat{v}_{3,-}^1$}
\uput[d](-4,2.6){$\hat{v}_{2,+}^1$}
\uput[d](-3,3.1){$\hat{v}_{2,-}^1$}
\uput[d](-1.5,3.1){$\hat{v}_{1,+}^1$}
\uput[d](-0.5,2.6){$\hat{v}_{1,-}^1$}

\uput[u](0.5,0.9){$\hat{v}_{0,-}^2$}
\uput[u](1.5,0.4){$\hat{v}_{0,+}^2$}
\uput[u](3,0.4){$\hat{v}_{1,-}^2$}
\uput[u](4,0.9){$\hat{v}_{1,+}^2$}
\uput[d](4,2.6){$\hat{v}_{2,-}^2$}
\uput[d](3,3.1){$\hat{v}_{2,+}^2$}
\uput[d](1.5,3.1){$\hat{v}_{3,-}^2$}
\uput[d](0.5,2.6){$\hat{v}_{3,+}^2$}

\uput[d](0.5,-0.9){$\hat{v}_{0,+}^3$}
\uput[d](1.5,-0.4){$\hat{v}_{0,-}^3$}
\uput[d](3,-0.4){$\hat{v}_{3,+}^3$}
\uput[d](4,-0.9){$\hat{v}_{3,-}^3$}
\uput[u](4,-2.6){$\hat{v}_{2,+}^3$}
\uput[u](3,-3.1){$\hat{v}_{2,-}^3$}
\uput[u](1.5,-3.1){$\hat{v}_{1,+}^3$}
\uput[u](0.5,-2.6){$\hat{v}_{1,-}^3$}
\end{pspicture}
}

}
\end{center}
\caption{\label{figure: simple form} Relation between $\lineK_\rho$
and $\hatL_\rho$ near $\bar{f}^{-1}$ when $R/R_t = \D_2, \Z_4, \D_4$
with $\Lambda_t = \Lambda^\sq$}
\end{figure}

Next, we define $\bar{x}_j$ with subscript $j$ for any point
$\bar{x}$ of $|\lineL_\rho|.$
\begin{notation}
~
\begin{enumerate}
  \item For a vertex $\bar{v}$ in $\lineL_\rho$ and its image $v =
\pi(\bar{v}),$ we label vertices in $\pi^{-1} (v)$ with $\bar{v}_j$
to satisfy
\begin{enumerate}
  \item[i)] $\pi^{-1} (v) = \{ \bar{v}_j ~|~ j \in \Z_{j_\rho} \},$
  \item[ii)] $\bar{v}_0 = \bar{v},$
  \item[iii)] in each face $|\bar{f}_j|$ containing $\bar{v}_j$
  for $j \in \Z_{j_\rho},$ we can take a small neighborhood
  $U_j$ of $\bar{v}_j$ so that
  $\pi(U_j)$'s are arranged in the counterclockwise way around
  $v.$
\end{enumerate}
  \item For a non-vertex $\bar{x}$ in $|\lineL_\rho|$ and its image $x =
\pi(\bar{x}),$ we label two points in $\pi^{-1} (x)$ with $\{
\bar{x}_j | j \in \Z_2 \}$ to satisfy $\bar{x}_0 = \bar{x}.$
\end{enumerate}
For simplicity, we denote $(\bar{v}^i)_j,$ $(\bar{d}^i)_j$ by
$\bar{v}_j^i,$ $\bar{d}_j^i,$ respectively.
\end{notation}

Now, we introduce superscript $i$ and subscripts $+, -$ for vertices
and edges in $\hatL_\rho.$ Before it, we introduce a simplicial map.
Let $c: \hatL_\rho \rightarrow \hatL_\rho$ be the simplicial map
whose underlying space map $|c| : |\hatL_\rho| \rightarrow
|\hatL_\rho|$ is defined as
\begin{enumerate}
  \item[] for each adjacent $\hat{e}, \hat{e}^\prime \in \hatL_\rho,$
        each point $\hat{x}$ in $|\hat{e}|$ is sent to the point $|c|(\hat{x})$
        in $|\hat{e}^\prime|$ to satisfy $\pi(p_{|\lineL|}(\hat{x}))
        = \pi(p_{|\lineL|}(c(\hat{x}))).$
\end{enumerate}
For example, $\hat{e}$ and $c(\hat{e})$ are adjacent for any edge
$\hat{e}$ in $\hatL_\rho.$ Easily, $|c|$ is equivariant. For
notational simplicity, we define $c$ also on edges of $\lineL_\rho$
to satisfy $c( p_{\lineL}(\hat{e}) ) = p_{\lineL}( c(\hat{e}) )$ for
each edge $\hat{e}.$

\begin{notation}
~
\begin{enumerate}
  \item For a vertex $\bar{v}$ in $\lineL_\rho,$ we label two vertices in
$p_{\lineL}^{-1} (\bar{v})$ with $\hat{v}_\pm$ to satisfy
\begin{equation*}
p_{\lineL} (c(\hat{v}_+)) = \bar{v}_1 ~ \text{ and } ~ p_{\lineL}
(c(\hat{v}_-)) = \bar{v}_{-1}.
\end{equation*}
  \item For a non-vertex $\bar{x}$ in $|\lineL_\rho|,$
  we denote the point in $p_{|\lineL|}^{-1} (\bar{x})$ as
  $\hat{x}_+$ or $\hat{x}_-,$ i.e. $\hat{x}_+ = \hat{x}_-.$
\end{enumerate}
\end{notation}
For simplicity, denote $\hat{x}_\pm$ for $\bar{x}= \bar{v}^i,
\bar{v}_j^i, \bar{d}^i, \bar{d}_j^i$ by $\hat{v}_\pm^i,
\hat{v}_{j,\pm}^i, \hat{d}_\pm^i, \hat{d}_{j,\pm}^i,$ respectively.
So, if $\bar{d}^i$ is a barycenter of an edge, then $\hat{d}_+^i =
\hat{d}_-^i.$ And, denote by $\hat{e}^i$ the edge in $\hatL_\rho$
whose image by $p_\lineL$ is $\bar{e}^i$ for $i \in \Z_{i_\rho}.$
Then, we have finished introducing notations in Figure \ref{figure:
simple form}. By using these notations, we introduce one-dimensional
fundamental domain of $|\hatL_\rho|.$ For any two points $\hat{a},
\hat{b}$ in an edge of $|\hatL_\rho|,$ denote by $[\hat{a},
\hat{b}]$ the shortest line connecting $\hat{a}$ and $\hat{b}.$ For
$\bar{D}_\rho= \bigcup_{i \in I_\rho^-} [\bar{d}^i, \bar{d}^{i+1}]
\subset |\lineK_\rho|,$ we define $\hat{D}_\rho$ in $|\hatL_\rho|$
as the disjoint union $\bigcup_{i \in I_\rho^-} [\hat{d}_+^i,
\hat{d}_-^{i+1}]$ so that $p_{|\lineL|}( \hat{D}_\rho
)=\bar{D}_\rho.$ And, denote by $\hat{\mathbf{D}}_\rho$ the set $(
\pi \circ p_{|\lineL|})^{-1} ( D_\rho )$ in $|\hatL_\rho|$ which is
equal to
\begin{equation*}
\Big( \bigcup_{i \in I_\rho^-} [\hat{d}_+^i, \hat{d}_-^{i+1}] \Big)
~ \bigcup ~ \Big( \bigcup_{i \in I_\rho^-} |c|([\hat{d}_+^i,
\hat{d}_-^{i+1}]) \Big) ~ \bigcup ~ \Big( \bigcup_{\bar{v} \in
\bar{D}_\rho} ( \pi \circ p_{\lineL} )^{-1} (v) \Big)
\end{equation*}
where $v=\pi(\bar{v}).$

For convenience in calculation, we parameterize each edge
$\hat{e}^i$ linearly by $s \in [0,1]$ so that $\hat{v}_+^i = 0,$
$b(\hat{e}^i) = 1/2,$ $\hat{v}_-^{i+1} = 1,$ and parameterize each
edge $c(\hat{e}^i)$ linearly by $s \in [0,1]$ so that $|c|(s) = 1-s$
for $s \in |\hat{e}^i|.$ We repeatedly use this parametrization.

%


Now, we describe an equivariant vector bundle over $\R^2/\Lambda$ as
an equivariant clutching construction of an equivariant vector
bundle over $|\lineK_\rho|.$ Let $V_B$ be a $G_\chi$ vector bundle
over $B$ such that $(\res_H^{G_\chi} V_B)|_{b(\bar{f})}$ is
$\chi$-isotypical at each $b(\bar{f})$ in $B.$ Especially,
$(\res_H^{G_\chi} V_B)|_{b(\bar{f})}$'s are all isomorphic. Also, if
we denote by $V_{\bar{f}}$ the isotropy representation of $V_B$ at
each $b(\bar{f})$ in $B,$ then $V_B \cong G_\chi
\times_{(G_\chi)_{b(\bar{f})}} V_{\bar{f}}$ because $G_\chi$ acts
transitively on $B.$ And, we define $\Vect_{G_\chi} (\R^2/\Lambda,
\chi)_{V_B}$ as the set
\begin{equation*}
\{ E \in \Vect_{G_\chi} (\R^2/\Lambda, \chi) ~ \big| ~ E|_{B} \cong
V_B \}.
\end{equation*}
Similarly, $\Vect_{G_\chi} (|\lineK_\rho|, \chi)_{V_B}$ is defined.
Observe that $\Vect_{G_\chi} (|\lineK_\rho|, \chi)_{V_B}$ has the
unique element $[F_{V_B}]$ for the bundle $F_{V_B} = G_\chi
\times_{(G_\chi)_{b(\bar{f}^{-1})}} ( |\bar{f}^{-1}| \times
V_{\bar{f}^{-1}} )$ because $|\lineK_\rho| \cong G_\chi
\times_{(G_\chi)_{b(\bar{f}^{-1})}} |\bar{f}^{-1}|$ is equivariant
homotopically equivalent to $B.$ Henceforward, we use
trivializations
\begin{equation}
\label{equation: trivialization}
\begin{array}{ll} |\bar{f}| \times
V_{\bar{f}} & \quad \text{ for } \big(
\res_{(G_\chi)_{b(\bar{f})}}^{G_\chi} F_{V_B} \big)
\big|_{|\bar{f}|}, \\
|\bar{e}| \times V_{\bar{f}} & \quad \text{ for } \big(
\res_{(G_\chi)_{b(\bar{e})}}^{G_\chi} F_{V_B} \big)
\big|_{|\bar{e}|}
\end{array}
\end{equation}
for each face $\bar{f}$ and edge $\bar{e} \in \bar{f}.$ Then, each
$E \in \Vect_{G_\chi} (\R^2/\Lambda, \chi)_{V_B}$ can be constructed
by gluing $F_{V_B} \cong \pi^* E$ along edges through
\begin{equation*}
|\bar{e}| \times V_{\bar{f}} \longrightarrow |c(\bar{e})| \times
V_{\bar{f}^\prime}, \quad ( ~ p_{|\lineL|}(\hat{x}), u ~ ) \mapsto
\big( ~ p_{|\lineL|}(|c|(\hat{x})), \varphi_{|\hat{e}|}(\hat{x}) ~ u
~ \big)
\end{equation*}
via some continuous maps
\begin{equation*}
\varphi_{|\hat{e}|} : |\hat{e}| \rightarrow \Iso ( V_{\bar{f}},
V_{\bar{f}^\prime} )
\end{equation*}
for each edge $\hat{e},$ $\hat{x} \in |\hat{e}|,$ $u \in
V_{\bar{f}}$ where $\bar{e} = p_{\lineL}(\hat{e})$ and $\bar{e} \in
\bar{f},$ $c(\bar{e}) \in \bar{f}^\prime.$ The union $\Phi = \bigcup
\varphi_{|\hat{e}|}$ is called an \textit{equivariant clutching map}
of $E$ with respect to $V_B.$ When we use the phrase `with respect
to $V_B$', it is assumed that we use the bundle $F_{V_B}$ and its
trivialization (\ref{equation: trivialization}) in gluing. This
construction of the bundle is called an equivariant clutching
construction, and $E$ is denoted by $E^\Phi$ to stress that it is
constructed through $\Phi.$ Here, note that $\Phi$ is defined on
$p_{|\lineL|}^* F_{V_B}$ over $|\hatL_\rho|.$ That is why we define
$\hatL_\rho.$ Sometimes, we regard $\Phi$ as a map
\begin{equation*}
p_{|\lineL|}^* F_{V_B} \rightarrow p_{|\lineL|}^* F_{V_B},
\quad(\hat{x}, u) \mapsto \big( |c|(\hat{x}), \Phi (\hat{x}) u \big)
\end{equation*}
by using trivialization (\ref{equation: trivialization}) for each
$(\hat{x}, u) \in |\hat{e}| \times V_{\bar{f}}$ where $\bar{e} \in
\bar{f}.$ Also, note that $\Phi$ should be equivariant, i.e.
\begin{equation*}
g \Phi(g^{-1} \hat{x}) g^{-1} = \Phi(\hat{x})
\end{equation*}
for all $g \in G_\chi, \hat{x} \in |\hatL_\rho|$ because $\Phi$
gives an equivariant vector bundle. An equivariant clutching map of
some bundle in $\Vect_{G_\chi} (\R^2/\Lambda, \chi)_{V_B}$ with
respect to $V_B$ is called simply an \textit{equivariant clutching
map} with respect to $V_B,$ and let $\Omega_{V_B}$ be the set of all
equivariant clutching maps with respect to $V_B.$ We also need to
restrict an equivariant clutching map in $\Omega_{V_B}$ to
$\hat{D}_\rho.$ We explain for this. Let $\Omega_{\hat{D}_\rho,
V_B}$ be the set
\begin{equation*}
\{ \Phi |_{\hat{D}_\rho} ~ | ~ \Phi \in \Omega_{V_B} \}.
\end{equation*}
If two equivariant clutching maps coincide on $\hat{D}_\rho,$ then
they are identical by equivariance and definition of one-dimensional
fundamental domain. So, the restriction map $\Omega_{V_B}
\rightarrow \Omega_{\hat{D}_\rho, V_B}, ~ \Phi \mapsto
\Phi|_{\hat{D}_\rho}$ is bijective, and we obtain isomorphism $\pi_0
(\Omega_{V_B}) \cong \pi_0 (\Omega_{\hat{D}_\rho, V_B})$ between two
homotopies. This is why we restrict an equivariant clutching map to
$\hat{D}_\rho.$ We call a map $\Phi$ in $\Omega_{V_B}$ the
\textit{extension} of $\Phi|_{\hat{D}_\rho}$ in
$\Omega_{\hat{D}_\rho, V_B}.$ And, denote by
$E^{\Phi_{\hat{D}_\rho}}$ the bundle $E^\Phi$ if $\Phi$ is the
extension of $\Phi_{\hat{D}_\rho}.$

Next, we define preclutching map, a generalization of equivariant
clutching map. Let $C^0 (|\hatL_\rho|, V_B)$ be the set of functions
$\Phi$ on $|\hatL_\rho|$ satisfying $\Phi|_{|\hat{e}|}(\hat{x}) \in
\Iso(V_{\bar{f}}, V_{\bar{f}^\prime})$ for each $\hat{e}$ and
$\hat{x} \in |\hat{e}|$ where $\hat{e}, \hat{e}^\prime$ are adjacent
and $\bar{e} \in \bar{f}, \bar{e}^\prime \in \bar{f}^\prime.$ And,
let $C^0 (\hat{D}_\rho, V_B)$ be the set of functions
$\Phi_{\hat{D}_\rho}$ on $\hat{D}_\rho$ defined by
\begin{equation*}
\{ \Phi |_{\hat{D}_\rho} ~ | ~ \Phi \in C^0 ( |\hatL_\rho|, V_B )
\}.
\end{equation*}
A function $\Phi$ in $C^0 (|\hatL_\rho|, V_B)$ or a function
$\Phi_{\hat{D}_\rho}$ in $C^0 (\hat{D}_\rho, V_B)$ is called a
\textit{preclutching map} with respect to $V_B.$ Denote by
$\varphi^i$ the restriction of $\Phi_{\hat{D}_\rho}$ in $C^0
(\hat{D}_\rho, V_B)$ to $[ \hat{d}_+^i, \hat{d}_-^{i+1} ]$ for $i
\in I_\rho^-,$ i.e. $\Phi_{\hat{D}_\rho}$ is the disjoint union of
$\varphi^i$'s. Then, it is a natural question under which conditions
a preclutching map becomes an equivariant clutching map. We can
answer this question for a preclutching map in $C^0 (|\hatL_\rho|,
V_B).$ A preclutching map $\Phi$ in $C^0 (|\hatL_\rho|, V_B)$ is an
equivariant clutching map with respect to $V_B$ if and only if
\begin{itemize}
  \item[N1.] $\Phi(|c|(\hat{x})) = \Phi(\hat{x})^{-1}$ for each $\hat{x} \in |\hatL_\rho|,$
  \item[N2.] For each vertex $\bar{v} \in \lineK_\rho,$
  \begin{equation*}
  \Phi(\hat{v}_{j_\rho-1,+}) \cdots \Phi(\hat{v}_{j,+}) \cdots
  \Phi(\hat{v}_{0,+}) = \id
  \end{equation*}
  for $j \in \Z_{j_\rho},$
  \item[E1.] $\Phi(g\hat{x}) = g \Phi(\hat{x}) g^{-1}$ for each $\hat{x} \in |\hatL_\rho|, g \in G_\chi.$
\end{itemize}
We explain for this more precisely. As we have seen in \cite{Ki}, if
$\Phi$ satisfies Condition N1., N2., then the bundle $E^\Phi$ is
well-defined and becomes an inequivariant vector bundle though it
need not be an equivariant vector bundle. And, if $\Phi$ also
satisfies Condition E1., then $E^\Phi$ becomes an equivariant vector
bundle so that $\Phi$ is an equivariant clutching map. We will
answer the same question for a preclutching map in $C^0
(\hat{D}_\rho, V_B)$ in Section \ref{section: equivariant clutching
maps}.

\section{Relations between $\pi_0 ( \Omega_{ V_B } ),$ $\Vect_{G_\chi} (\R^2/\Lambda,
\chi),$ $A_{G_\chi} (\R^2/\Lambda, \chi).$}
 \label{section: relation}

Now, we investigate relations between
\begin{equation*}
\pi_0 ( \Omega_{ V_B } ), \quad \Vect_{G_\chi} (\R^2/\Lambda,
\chi)_{V_B}, \quad A_{G_\chi} (\R^2/\Lambda, \chi).
\end{equation*}
Our classification of the paper is based on these relations. Before
it, we state basic facts on equivariant vector bundles from the
previous paper.

\begin{proposition}
 \label{proposition: homotopy gives isomorphism}
For two maps $\Phi$ and $\Phi^\prime$ in $\Omega_{V_B},$ if $\Phi$
and $\Phi^\prime$ are equivariantly homotopic, then $E^\Phi$ and
$E^{\Phi^\prime}$ are isomorphic equivariant vector bundles.
\end{proposition}

When we consider a map in $\Omega_{V_B}$ as a map defined on
$p_{|\lineL|}^* F_{V_B},$ we have the following result:

\begin{proposition}
 \label{proposition: equivalent condition for isomorphism}
For two maps $\Phi$ and $\Phi^\prime$ in $\Omega_{V_B},$ $[E^\Phi] =
[E^{\Phi^\prime}]$ in $\Vect_{G_\chi} (\R^2/\Lambda, \chi)_{V_B}$ if
and only if there is a $G_\chi$-isomorphism $\Theta : F_{V_B}
\rightarrow F_{V_B}$ such that $( p_{|\lineL|}^* \Theta ) \Phi =
\Phi^\prime ( p_{|\lineL|}^* \Theta )$ where $p_{|\lineL|}^* \Theta
: p_{|\lineL|}^* F_{V_B} \rightarrow p_{|\lineL|}^* F_{V_B} $ is the
pull-back of $\Theta.$
\end{proposition}

\[
\begin{CD}
p_{|\lineL|}^* F_{V_B}   @>{p_{|\lineL|}^* \Theta}>> p_{|\lineL|}^* F_{V_B} \\
@VV{\Phi}V    @VV{\Phi^\prime}V \\
p_{|\lineL|}^* F_{V_B}   @>{p_{|\lineL|}^* \Theta}>> p_{|\lineL|}^* F_{V_B} \\
\end{CD}
\]

Consider the map $\imath_\Omega : \pi_0 ( \Omega_{V_B} ) \rightarrow
\Vect_{G_\chi} (\R^2/\Lambda, \chi)_{V_B}$ mapping $[\Phi]$ to
$[E^\Phi].$ This is well defined by Proposition \ref{proposition:
homotopy gives isomorphism}. And, the map $p_\Omega : \pi_0 (
\Omega_{V_B} ) \rightarrow A_{G_\chi} (\R^2/\Lambda, \chi)$ defined
as $p_\Omega = p_\vect \circ \imath_\Omega$ satisfies the following
diagram:
\begin{equation*}
\SelectTips{cm}{} \xymatrix{ \pi_0 ( \Omega_{V_B} )
\ar[r]^-{\imath_\Omega}
\ar[dr]_-{p_\Omega} & \Vect_{G_\chi} (\R^2/\Lambda, \chi)_{V_B} \\
& A_{G_\chi} (\R^2/\Lambda, \chi) \ar@{<-}[u]_-{p_\vect} }.
\end{equation*}
We might consider the map $p_\Omega$ as defined also on $\pi_0 (
\Omega_{ \hat{D}_\rho, V_B} )$ because $\pi_0 (\Omega_{V_B}) \cong
\pi_0 (\Omega_{\hat{D}_\rho, V_B}).$ By using $p_\Omega,$ we would
decompose $\Omega_{V_B}.$ Let $p_{\pi_0} : \Omega_{\hat{D}_\rho,
V_B} \rightarrow \pi_0 ( \Omega_{\hat{D}_\rho, V_B} )$ be the
natural quotient map. For different elements in $A_{G_\chi}
(\R^2/\Lambda, \chi),$ their preimages through $(p_\Omega \circ
p_{\pi_0})^{-1}$ do not intersect so that we obtain a decomposition
of $\Omega_{ \hat{D}_\rho, V_B}.$ We describe this decomposition
more precisely. For each $(W_{d^i})_{i \in I_\rho^+} \in A_{G_\chi}
(\R^2/\Lambda, \chi),$ put
\begin{equation*}
V_B = G_\chi \times_{(G_\chi)_{d^{-1}}} W_{d^{-1}}, \quad F_{V_B} =
G_\chi \times_{(G_\chi)_{d^{-1}}} ( |\bar{f}^{-1}| \times W_{d^{-1}}
),
\end{equation*}
and by using these define $\Omega_{\hat{D}_\rho, (W_{d^i})_{i \in
I_\rho^+}}$ as $(p_\Omega \circ p_{\pi_0})^{-1} \big( (W_{d^i})_{i
\in I_\rho^+} \big ).$ Henceforward, we will use these $V_B$ and
$F_{V_B}$ when we deal with $\Omega_{\hat{D}_\rho, (W_{d^i})_{i \in
I_\rho^+}}.$ Then, given a bundle $V_B,$ the set
$\Omega_{\hat{D}_\rho, V_B}$ is equal to the disjoint union
\begin{equation*}
\bigcup_{(W_{d^i})_{i \in I_\rho^+} \in A_{G_\chi} (\R^2/\Lambda,
\chi) \text{ with } W_{d^{-1}} = V_{\bar{f}^{-1}} } \quad
\Omega_{\hat{D}_\rho, (W_{d^i})_{i \in I_\rho^+}}.
\end{equation*}
From this decomposition, it suffices to focus on
$\Omega_{\hat{D}_\rho, (W_{d^i})_{i \in I_\rho^+}}$ to understand
$\Omega_{\hat{D}_\rho, V_B}.$

In general, calculation of $\pi_0 ( \Omega_{ \hat{D}_\rho, V_B} )$
do not give classification of equivariant vector bundles even for
$S^2.$ However, in many cases it gives classification of equivariant
vector bundles as we have seen in \cite{Ki}. Let $c_1 : \pi_0 (
\Omega_{V_B} ) \rightarrow H^2 (\R^2/\Lambda), ~ \Phi \mapsto c_1
(E^\Phi)$ be the Chern class.

\begin{lemma}
 \label{lemma: lemma for isomorphism}
\begin{enumerate}
  \item Assume that $\pi_0 ( \Omega_{\hat{D}_\rho,
(W_{d^i})_{i \in I_\rho^+}} )$ is nonempty for each $(W_{d^i})_{i
\in I_\rho^+}$ in $A_{G_\chi} ( \R^2/\Lambda, \chi ),$ and that $c_1
: \pi_0 ( \Omega_{\hat{D}_\rho, (W_{d^i})_{i \in I_\rho^+}} )
\rightarrow H^2 (\R^2/\Lambda)$ is injective for each $(W_{d^i})_{i
\in I_\rho^+}$ in $A_{G_\chi} ( \R^2/\Lambda ).$ Then, $p_\vect$ is
surjective, and
\begin{equation*}
p_{\vect} \times c_1 : \Vect_{G_\chi} (\R^2/\Lambda, \chi)
\rightarrow A_{G_\chi} ( \R^2/\Lambda, \chi ) \times H^2
(\R^2/\Lambda)
\end{equation*}
is injective.

  \item Assume that $\pi_0 ( \Omega_{\hat{D}_\rho,
(W_{d^i})_{i \in I_\rho^+}} )$ consists of exactly one element for
each $(W_{d^i})_{i \in I_\rho^+}$ in $A_{G_\chi} ( \R^2/\Lambda,
\chi ).$ Then
\begin{equation*}
p_{\vect} : \Vect_{G_\chi} (\R^2/\Lambda, \chi) \rightarrow
A_{G_\chi} ( \R^2/\Lambda, \chi )
\end{equation*}
is isomorphic.
\end{enumerate}
\end{lemma}

By this lemma, we only have to calculate $\pi_0 (
\Omega_{\hat{D}_\rho, (W_{d^i})_{i \in I_\rho^+}} )$ to classify
equivariant vector bundles in many cases. When we can not apply this
lemma, we should apply Proposition \ref{proposition: equivalent
condition for isomorphism} directly. In fact, we can apply this
lemma except the case when $\rho(G_\chi)$ is equal to $\D_1$ with
$A\bar{c}+\bar{c}=\bar{l}_0$ as we shall see in the below.

\section{equivariant pointwise clutching map}
 \label{section: pointwise clutching map}

Let $\Phi$ be an equivariant clutching map of an equivariant vector
bundle $E$ in $\Vect_{G_\chi} ( \R^2/\Lambda ),$ i.e. the map $\Phi$
glues $F_{V_B}$ along $| \lineL_\rho |$ to give $E.$ Let us
investigate this gluing process pointwisely. For each $\bar{x} \in |
\lineL_\rho |$ and $x = \pi(\bar{x}),$ let $\bar{\mathbf{x}} =
\pi^{-1} ( x ) = \{ \bar{x}_j | j \in \Z_m \}$ for some $m.$ Then,
the map $\Phi$ glues the $(G_\chi)_x$-bundle
$(\res_{(G_\chi)_x}^{G_\chi} F_{V_B})|_{\bar{\mathbf{x}}}$ to give
the $(G_\chi)_x$-representation $E_x,$ and call this process
\textit{equivariant pointwise gluing}. Here, we can observe that
$(G_\chi)_{\bar{x}_j} < (G_\chi)_x$ for each $j \in \Z_m$ and
\begin{equation}
\label{equation: extension} \res_{(G_\chi)_{\bar{x}_j}}^{(G_\chi)_x}
E_x \cong (F_{V_B})_{\bar{x}_j}
\end{equation}
by equivariance of $\Phi.$ In dealing with equivariant clutching
maps, technical difficulties occur in equivariant pointwise gluings
because $\Phi$ is just a continuous collection of equivariant
pointwise gluings at points in $| \lineL_\rho |.$ In this section,
we review the concept and results of equivariant pointwise clutching
map from the previous paper \cite{Ki}, and supplement these with two
more cases. To deal with equivariant pointwise gluing, we need the
concept of representation extension. For compact Lie groups $N_1 <
N_2,$ let $V$ be a $N_2$-representation and $W$ be an
$N_1$-representation. Then, $V$ is called an \textit{representation
extension} or $N_2$-\textit{extension} of $W$ if $\res^{N_2}_{N_1} V
\cong W.$ For example, $E_x$ is an $(G_\chi)_x$-extension of
$(F_{V_B})_{\bar{x}_j}$ for each $j \in \Z_m$ by (\ref{equation:
extension}). And, let $\ext_{N_1}^{N_2} W$ be the set
\begin{equation*}
\{ V \in \Rep(N_2) ~ | ~ \res_{N_1}^{N_2} V \cong W \}.
\end{equation*}

Let a compact Lie group $N_2$ act on a finite set $\bar{\mathbf{x}}
= \{ \bar{x}_j | ~ j \in \Z_m \}$ for $m \ge 2,$ and let $N_0$ and
$N_1$ be the kernel of the action and the isotropy subgroup
$(N_2)_{\bar{x}_0},$ respectively. Let $F$ be an $N_2$-vector bundle
over $\bar{\mathbf{x}}.$ Assume that $(\res_{N_0}^{N_2} F)
|_{\bar{x}_j}$'s are all isomorphic (not necessarily
$N_0$-isotypical) for $j \in \Z_m.$ Consider an arbitrary map
\begin{equation*}
\psi : \bar{\mathbf{x}} \rightarrow \amalg_{j \in \Z_m} \Iso(
F_{\bar{x}_j}, F_{\bar{x}_{j+1}} )
\end{equation*}
such that $\psi( \bar{x}_j ) \in \Iso( F_{\bar{x}_j},
F_{\bar{x}_{j+1}} ).$ Call such a map \textit{pointwise preclutching
map} with respect to $F.$ By using $\psi,$ we glue
$F_{\bar{x}_j}$'s, i.e. a vector $u$ in $F_{\bar{x}_j}$ is
identified with $\psi(\bar{x}_j) u$ in $F_{\bar{x}_{j+1}}.$ Let
$F/\psi$ be the quotient of $F$ through this identification, and let
$p_\psi : F \rightarrow F/\psi$ be the quotient map. Let
$\imath_\psi : F_{\bar{x}_0} \rightarrow F/\psi$ be the composition
of the natural injection $\imath_{\bar{x}_0} : F_{\bar{x}_0}
\rightarrow F$ and the quotient map $p_\psi.$
\begin{equation}
\label{figure: diagram of equivariant pointwise gluing}
\SelectTips{cm}{} \xymatrix{ F_{\bar{x}_0}
\ar[r]^-{\imath_{\bar{x}_0}}
\ar[d]_-{\imath_\psi} & F \\
F/\psi \ar@{<-}[ur]_-{p_\psi} }
\end{equation}
We would find conditions on $\psi$ under which the quotient $F/\psi$
inherits an $N_2$-representation structure from $F$ and the map
$\imath_\psi$ becomes an $N_1$-isomorphism from $F_{\bar{x}_0}$ to
$\res_{N_1}^{N_2} (F/\psi).$ For notational simplicity, denote
\begin{equation*}
\psi(\bar{x}_{j^\prime}) \cdots \psi(\bar{x}_{j+1}) \psi(\bar{x}_j)
u
\end{equation*}
by $\psi^{j^\prime-j+1} u$ for $u \in F_{\bar{x}_j}$ and $j \le
j^\prime$ in $\Z.$

\begin{lemma}
 \label{lemma: equivalent condition for psi}
For a pointwise preclutching map $\psi$ with respect to $F,$ the
quotient $F/\psi$ carries an $N_2$-representation structure so that
$p_\psi$ is $N_2$-equivariant and the map $\imath_\psi$ is an
$N_1$-isomorphism if and only if the following conditions hold :
\begin{enumerate}
  \item $\psi^m = \id$ in $\Iso(F_{\bar{x}_j})$ for each $j \in
  \Z_m.$ So, $\psi^j$ is well defined for all $j \in \Z_m.$
  \item $\psi^{j_3 - j_1} = g \psi( \bar{x}_{j_2} ) g^{-1}$ in $F_{\bar{x}_{j_1}}$ for
  each $j_1 \in \Z_m,$ $g \in N_2$ when $g^{-1} \bar{x}_{j_1} = \bar{x}_{j_2}$
        and $g \bar{x}_{j_2 +1} = \bar{x}_{j_3}$ for some $j_2, j_3 \in \Z_m.$
\end{enumerate}
\end{lemma}

A pointwise preclutching map satisfying conditions (1), (2) of Lemma
\ref{lemma: equivalent condition for psi} is called an
\textit{equivariant pointwise clutching map} with respect to $F.$
Let $\mathcal{A}$ be the set of all equivariant pointwise clutching
maps with respect to $F,$ and an $N_2$-representation $W$ is called
\textit{determined} by $\psi \in \mathcal{A}$ with respect to $F$ if
$W \cong F/\psi.$

\begin{remark}
Here, we give an example of equivariant pointwise clutching map by
using the bundle $F_{V_B}$ over $|\lineK_\rho|$ of Section
\ref{section: preclutching map}. For each $\bar{x} \in | \lineL_\rho
|$ and $x = \pi(\bar{x}),$ let $\bar{\mathbf{x}} = \pi^{-1} ( x ) =
\{ \bar{x}_j | j \in \Z_m \}$ for some $m.$ Put $F =
(\res_{(G_\chi)_x}^{G_\chi} F_{V_B})|_{\bar{\mathbf{x}}}.$ Since
$(\res_H^{G_\chi} V)|_{b(\bar{f})}$'s are all isomorphic,
$(\res_{N_0}^{N_2} F) |_{\bar{x}_j}$'s are all isomorphic. Given a
$\Phi$ in $\Omega_{V_B},$ we can define a map $\psi$ in
$\mathcal{A}$ by using $\Phi$ as follows:
\begin{enumerate}
\item if $\bar{x}$ is not a vertex and
$\bar{x} = p_{|\lineL|}(\hat{x}),$ then
\begin{equation*}
\psi (\bar{x}_0) = \Phi (\hat{x}) \quad \text{and} \quad \psi
(\bar{x}_1) = \Phi (|c|(\hat{x})),
\end{equation*}
\item if $\bar{x}$ is a vertex, then
\begin{equation*}
\psi_{\bar{x}} (\bar{x}_j) = \Phi (\hat{x}_{j,+}) \quad \text{and}
\quad \psi_{\bar{x}}^{-1} (\bar{x}_j) = \Phi (\hat{x}_{j,-})
\end{equation*}
for each $j \in \Z_{j_\rho}.$
\end{enumerate}
Then, $(E^\Phi)_x$ is determined by $\psi.$
\end{remark}

To calculate $\pi_0 ( \Omega_{ V_B } )$ later, we need to understand
topology of $\mathcal{A}$ because an equivariant clutching map is a
continuous collection of equivariant pointwise gluings. First, the
zeroth homotopy of $\mathcal{A}$ is related to $\ext_{N_1}^{N_2}
F_{\bar{x}_0}$ as follows:

\begin{proposition}
 \label{proposition: bijectivity with extensions}
Assume that $N_2$ acts transitively $\bar{\mathbf{x}}.$ Then, the
map
\begin{equation*}
\pi_0 (\mathcal{A}) \longrightarrow \ext_{N_1}^{N_2} F_{\bar{x}_0},
\quad [\psi] \mapsto F/\psi
\end{equation*}
is bijective.
\end{proposition}

\begin{remark}
In Lemma \ref{lemma: pointwise clutching for m=2 nontransitive} and
Proposition \ref{proposition: psi for Z_2}, \ref{proposition: psi
for trivial group}, we show that Proposition \ref{proposition:
bijectivity with extensions} also holds for the following cases:
\begin{enumerate}
  \item $N_2 = N_1 = N_0$ with $m=2, 4,$
  \item $N_2/N_0 \cong \Z_2$ and $N_2/N_0$ acts freely on
$\bar{\mathbf{x}}$ with $m=4.$
\end{enumerate}
\end{remark}

Let $\mathcal{A}^j$ be the set
\begin{equation*}
\{ \psi(\bar{x}_j) ~|~ \psi \in \mathcal{A} \},
\end{equation*}
and let $\mathcal{A}^{j,j^\prime}$ be the set
\begin{equation*}
\{ ( \psi(\bar{x}_j), \psi(\bar{x}_{j^\prime}) ) ~|~ \psi \in
\mathcal{A} \}.
\end{equation*}
In the below, it will be witnessed that $\mathcal{A}$ is
homeomorphic to $\mathcal{A}^j$ or $\mathcal{A}^{j,j^\prime}$ in
many cases in the below. For example, the evaluation map
\begin{equation*}
\mathcal{A} \rightarrow \mathcal{A}^j, ~ \psi \mapsto
\psi(\bar{x}_j)
\end{equation*}
is homeomorphic by Lemma \ref{lemma: equivalent condition for
psi}.(1) when $m=2.$ Meanwhile, we also need to restrict our
arguments on $\mathcal{A}$ to $\{ \bar{x}_j, \bar{x}_{j^\prime} \}$
as follows: let $\mathcal{A}_{j,j^\prime}$ with $j \ne j^\prime$ be
the set of equivariant pointwise clutching maps with respect to the
$(N_2)_{\{ \bar{x}_j, \bar{x}_{j^\prime} \}}$-bundle $\big(
\res_{(N_2)_{\{ \bar{x}_j, \bar{x}_{j^\prime} \}}}^{N_2} F \big) |_{
\{ \bar{x}_j, \bar{x}_{j^\prime} \} }$ where $(N_2)_{\{ \bar{x}_j,
\bar{x}_{j^\prime} \}}$ is the subgroup preserving $\{ \bar{x}_j,
\bar{x}_{j^\prime} \}.$ And, let $\mathcal{A}_{j,j^\prime}^k$ for
$k=j,j^\prime$ be the set
\begin{equation*}
\{ \psi (\bar{x}_k) | \psi \in \mathcal{A}_{j,j^\prime} \}.
\end{equation*}
Then, we obtain a useful lemma.

\begin{lemma}
 \label{lemma: restricted pointwise clutching}
The map $\res_{j,j^\prime}: \mathcal{A} \rightarrow
\mathcal{A}_{j,j^\prime}, ~ \psi \mapsto \res_{j,j^\prime}(\psi)$
with
\begin{equation*}
\res_{j,j^\prime}(\psi)(\bar{x}_j) =
 \psi^{j^\prime-j}(\bar{x}_j), \quad \res_{j,j^\prime}(\psi)(\bar{x}_{j^\prime}) =
\psi^{j-j^\prime}(\bar{x}_{j^\prime})
\end{equation*}
is well defined. And,
\begin{equation*}
\res_{(N_2)_{\{ \bar{x}_j, \bar{x}_{j^\prime} \}}}^{N_2} (F/\psi)
\cong \big( \res_{(N_2)_{\{ \bar{x}_j, \bar{x}_{j^\prime} \}}}^{N_2}
F \big) |_{ \{ \bar{x}_j, \bar{x}_{j^\prime} \} } /
\res_{j,j^\prime}(\psi)
\end{equation*}
for each $\psi \in \mathcal{A}.$
\end{lemma}

\begin{proof}
The injection from $\big( \res_{(N_2)_{\{ \bar{x}_j,
\bar{x}_{j^\prime} \}}}^{N_2} F \big) |_{ \{ \bar{x}_j,
\bar{x}_{j^\prime} \} }$ to $F$ induces the isomorphism through
$p_{\res_{j,{j^\prime}}(\psi)}$ and $p_\psi.$
\end{proof}

For $\psi \in \mathcal{A},$ denote by $\mathcal{A}_\psi$ and
$\mathcal{A}_\psi^0$ the path component of $\mathcal{A}$ containing
$\psi$ and the set $\{ \psi^\prime (\bar{x}_0) | \psi^\prime \in
\mathcal{A}_\psi \},$ respectively. Then, each path component is
expressed by isomorphism groups.

\begin{lemma}
 \label{lemma: topology of mathcal A}
For each $\psi \in \mathcal{A},$ $\mathcal{A}_\psi$ is homeomorphic
to
\begin{equation*}
\Iso_{N_1} ( F/\psi ) / \Iso_{N_2} ( F/\psi ).
\end{equation*}
\end{lemma}

In some special cases, we can understand $\mathcal{A}$ and
$\mathcal{A}_\psi$ more precisely.

\begin{lemma}
 \label{lemma: pointwise clutching for m=2 nontransitive}
If $m=2,$ then $\mathcal{A}$ is homeomorphic to $\mathcal{A}^0.$
Moreover, if $N_2 = N_1 = N_0,$ then $\mathcal{A}$ is homeomorphic
to $\mathcal{A}^0 = \Iso_{N_2} ( F_{\bar{x}_0}, F_{\bar{x}_1} ).$
\end{lemma}

\begin{proposition}
 \label{proposition: psi for cyclic}
Assume that $N_2 = \langle N_0, a_0 \rangle$ with some $a_0 \in N_2$
such that $a_0 \bar{x}_j = \bar{x}_{j+1}$ for each $j \in \Z_m$ so
that $N_2 / N_0 \cong \Z_m$ and $N_1 = N_0.$ Then,
\begin{enumerate}
  \item A pointwise preclutching map $\psi$ with respect to $F$ is
  in $\mathcal{A}$ if and only if $\psi^m = \id,$ $\psi(
\bar{x}_0 ) \in \Iso_{N_0} (F_{\bar{x}_0}, F_{\bar{x}_1}),$ and
$\psi( \bar{x}_j ) = a_0^j \psi(\bar{x}_0) a_0^{-j}$ for each $j \in
\Z_m.$
  \item $\mathcal{A}, \mathcal{A}_\psi$ are homeomorphic to
  $\mathcal{A}^0, \mathcal{A}_\psi^0,$ respectively.
  \item If $F_{\bar{x}_0}$ is $N_0$-isotypical, then
  $\mathcal{A}_\psi$ is simply connected
  for each $\psi$ in $\mathcal{A}.$
\end{enumerate}
\end{proposition}

\begin{proposition}
 \label{proposition: psi for dihedral}
Assume that $N_2 = \langle {N_0}, a_0 , b_0 \rangle$ with some $a_0,
b_0 \in N_2$ such that $a_0 \bar{x}_j = \bar{x}_{j+1}$ and $b_0
\bar{x}_j = \bar{x}_{-j+1}$ for each $j \in \Z_m$ so that $N_2 /
{N_0} \cong \D_m$ and $N_1 = \langle N_0, b_0 a_0 \rangle.$ Then,
\begin{enumerate}
  \item A pointwise preclutching map $\psi$ with respect to $F$ is
  in $\mathcal{A}$ if and only if $\psi^m = \id,$ $\psi(
\bar{x}_0 ) \in \Iso_{{N_0}} (F_{\bar{x}_0}, F_{\bar{x}_1}),$
$\psi^{-1} (\bar{x}_1) = b_0 \psi(\bar{x}_0) b_0^{-1},$ and $\psi(
\bar{x}_j ) = a_0^j \psi(\bar{x}_0) a_0^{-j}$ for each $j \in \Z_m.$
  \item $\mathcal{A}, \mathcal{A}_\psi$ are homeomorphic to
  $\mathcal{A}^0, \mathcal{A}_\psi^0,$ respectively.
\end{enumerate}
\end{proposition}

\begin{proposition}
 \label{proposition: psi for Z_2xZ_2}
Assume that $N_2 = \langle N_0, \alpha_1, \alpha_2, \alpha_3
\rangle$ with some $\alpha_j$'s in $N_2$ such that $\alpha_1
\bar{x}_j = \bar{x}_{-j+1},$ $\alpha_2 \bar{x}_j = \bar{x}_{j+2},$
$\alpha_3 \bar{x}_j = \bar{x}_{-j+3}$ for each $j \in \Z_m$ with
$m=4$ so that $N_2 / N_0 \cong \Z_2 \times \Z_2$ and $N_1 = N_0.$
Then, $\mathcal{A}$ is homeomorphic to $\mathcal{A}^{0,3}.$
\end{proposition}

For two points $\bar{x}_j, \bar{x}_{j^\prime}$ in
$\bar{\mathbf{x}},$ $\bar{x}_j \sim \bar{x}_{j^\prime}$ means that
$\bar{x}_j$ and $\bar{x}_{j^\prime}$ are in an $N_2$-orbit.

\begin{proposition} \label{proposition: psi for Z_2}
Assume that $N_2/N_0 \cong \Z_2$ and $N_2/N_0$ acts freely on
$\bar{\mathbf{x}}$ with $m=4.$
\begin{enumerate}
  \item If $\bar{x}_j \sim \bar{x}_{j^\prime}$ and $\bar{x}_{j^\prime} \nsim
  \bar{x}_{j^{\prime \prime}}$ for some
  $j \ne j^\prime \ne j^{\prime \prime},$ then
  \begin{equation*}
  \res_{j,j^\prime} \times \res_{j^\prime,j^{\prime \prime}}
  : \mathcal{A} \rightarrow
  \mathcal{A}_{j,j^\prime} \times \mathcal{A}_{j^\prime,j^{\prime \prime}}
  \end{equation*}
  is homeomorphic.
  \item $\mathcal{A}_{j^\prime, j^{\prime \prime}}^{j^\prime}$
  $= \Iso_{N_0} (F_{\bar{x}_{j^\prime}}, F_{\bar{x}_{j^{\prime \prime}}}).$
  If $F_{\bar{x}_0}$ is $N_0$-isotypical,
  then path components of $\mathcal{A}_{j,j^\prime}$ are simply
  connected and $\pi(\mathcal{A}_\psi) \cong \Z$ for each
  $\psi \in \mathcal{A}.$
  \item $F/\psi \cong \big( \res_{(N_2)_{\{
\bar{x}_j, \bar{x}_{j^\prime} \}}}^{N_2} F \big) |_{ \{ \bar{x}_j,
\bar{x}_{j^\prime} \} } /
  \res_{j,j^\prime}(\psi)$ for any $\psi \in \mathcal{A}.$
\end{enumerate}
\end{proposition}

\begin{proof}
For simplicity, assume that $N_2=\langle N_0, a_0 \rangle$ with some
$a_0 \in N_2$ such that $a_0 \bar{x}_j = \bar{x}_{j+2}$ for each $j
\in \Z_4.$ And, we prove this proposition only for $j=3, j^\prime=1,
j^{\prime \prime}=0.$ Then, it suffices to show that
\begin{equation*}
p: \mathcal{A} \rightarrow \mathcal{A}_{3,1}^3 \times
\mathcal{A}_{1,0}^0 , \psi \mapsto \Big( \psi(\bar{x}_0)
\psi(\bar{x}_3), \psi(\bar{x}_0) \Big)
\end{equation*}
is homeomorphic to show (1) by Lemma \ref{lemma: pointwise clutching
for m=2 nontransitive}. First, we show injectivity. Given
$\psi(\bar{x}_0) \psi(\bar{x}_3)$ and $\psi(\bar{x}_3),$ i.e.
$\psi(\bar{x}_0)$ and $\psi(\bar{x}_3),$ equivariance of $\psi$
gives us
\begin{equation*}
\psi(\bar{x}_1) = \psi(a_0 \bar{x}_3) = a_0 \psi(\bar{x}_3)
a_0^{-1}.
\end{equation*}
By $\psi^4 = \id,$ $\psi(\bar{x}_2)$ is also obtained. So, we obtain
injectivity. Next, we show surjectivity. Given two arbitrary
elements $A \in \mathcal{A}_{0,1}^0$ and $A^\prime \in
\mathcal{A}_{1,3}^3,$ $A^{-1} A^\prime$ is $N_0$-equivariant. We
would construct $\psi \in \mathcal{A}$ s.t. $\psi(\bar{x}_0)=A$ and
$\psi(\bar{x}_3)=A^{-1} A^\prime.$ Define pointwise preclutching map
$\psi$ as
\begin{align*}
\psi(\bar{x}_0) &= A, \\
\psi(\bar{x}_1) &= a_0 \psi(\bar{x}_3) a_0^{-1}, \\
\psi(\bar{x}_2) &= a_0 \psi(\bar{x}_0) a_0^{-1}, \\
\psi(\bar{x}_3) &=A^{-1} A^\prime.
\end{align*}
To show $\psi \in \mathcal{A},$ we only have to show that
\begin{equation*}
\psi(\bar{x}_3) \psi(\bar{x}_2) \psi(\bar{x}_1) \psi(\bar{x}_0) =
\id
\end{equation*}
by Lemma \ref{lemma: equivalent condition for psi}. This is
equivalent to
\begin{equation}
\tag{*} \psi(\bar{x}_0) \psi(\bar{x}_3) \psi(\bar{x}_2)
\psi(\bar{x}_1) = \id.
\end{equation}
Here, $\psi(\bar{x}_0) \psi(\bar{x}_3) = A^\prime$ and
\begin{align*}
\psi(\bar{x}_2) \psi(\bar{x}_1) &= a_0 \psi(\bar{x}_0) a_0^{-1}
\cdot a_0 \psi(\bar{x}_3) a_0^{-1} \\
&= a_0 A^\prime a_0^{-1}.
\end{align*}
Since $A^\prime \in \mathcal{A}_{1,3}^3,$ $a_0 A^\prime a_0^{-1} =
A^{\prime -1}$ by Proposition \ref{proposition: psi for cyclic}.(1)
so that we obtain (*). Therefore, we obtain a proof of (1).

Next, we prove (2). Since $(N_2)_{\{ \bar{x}_1, \bar{x}_0 \}} =
N_0,$ we obtain $\mathcal{A}_{1, 0}^0 = \Iso_{N_2} ( F_{\bar{x}_0},
F_{\bar{x}_1} )$ by Lemma \ref{lemma: pointwise clutching for m=2
nontransitive}. By Schur's Lemma, $\mathcal{A}_{1, 0}$ is path
connected and $\pi_1 (\mathcal{A}_{1, 0}) \cong \Z.$ And, since
$(N_2)_{\{ \bar{x}_3, \bar{x}_1 \}} = N_2,$ simply connectedness of
each path component of $\mathcal{A}_{3, 1}$ is obtained by
Proposition \ref{proposition: psi for cyclic}. Therefore, we obtain
a proof of $\pi(\mathcal{A}_\psi) \cong \Z$ by (1).

Last, (3) follows from Lemma \ref{lemma: restricted pointwise
clutching}.
\end{proof}

\begin{proposition} \label{proposition: psi for trivial group}
Assume that $g \bar{x}_j = \bar{x}_j$ for each $g \in N_2, j \in
\Z_4,$ i.e. $N_2= N_0.$ Then,
\begin{enumerate}
  \item $\mathcal{A}$ is equal to
\begin{equation*}
\{ \psi ~ | ~ \psi(\bar{x}_j) \in \Iso_{N_0} (F_{\bar{x}_j},
F_{\bar{x}_{j+1}}) \text{ for each } j \in \Z_4, \text{ and } \psi^4
= \id \}.
\end{equation*}
        And, $p: \mathcal{A} \rightarrow \mathcal{A}_{j,j+1}^j \times
\mathcal{A}_{j^\prime,j^\prime+1}^{j^\prime} \times
\mathcal{A}_{j^{\prime \prime},j^{\prime \prime}+1}^{j^{\prime
\prime}}, \psi \mapsto \Big( \psi(\bar{x}_j),
\psi(\bar{x}_{j^\prime}), \psi(\bar{x}_{j^{\prime \prime}}) \Big)$
is bijective for $j \ne j^\prime \ne j^{\prime \prime}.$ Here,
$\mathcal{A}_{j, j+1}^j = \Iso_{N_0} (F_{\bar{x}_j},
F_{\bar{x}_{j+1}}).$
  \item If $F$ is $N_0$-isotypical, then $\pi_1 (\mathcal{A}) \cong
  \Z^3.$
\end{enumerate}
\end{proposition}

\begin{proof}
Since $N_2$ fixes all $\bar{x}_j$'s, we have $N_2=N_0$ so that Lemma
\ref{lemma: pointwise clutching for m=2 nontransitive} gives
$\mathcal{A}_{j, j+1}^j = \Iso_{N_0} (F_{\bar{x}_j},
F_{\bar{x}_{j+1}}).$ This and Lemma \ref{lemma: equivalent condition
for psi} proves (1). That is, $\psi(\bar{x}_j)$ in
$\mathcal{A}_{j,j+1}^j = \Iso_{N_2} (F_{\bar{x}_j},
F_{\bar{x}_{j+1}})$ for three $j$'s are free and the remaining one
is restricted by $\psi^4 = \id.$

Schur's Lemma and (1) give a proof of (2).
\end{proof}

Since any $\psi$ in $\mathcal{A}$ glues all fibers of $F$ to obtain
a single vector space $F/\psi,$ $\psi$ might be considered to glue
each pair of fibers of $F.$ That is, $\psi$ determines the function
$\bar{\psi}$ defined on $\bar{\mathbf{x}} \times \bar{\mathbf{x}} -
\Delta$ sending a pair $(\bar{x}, \bar{x}^\prime)$ to the element
$\bar{\psi}(\bar{x}, \bar{x}^\prime)$ in $\Iso(F_{\bar{x}},
F_{\bar{x}^\prime})$ such that each $u$ in $F_{\bar{x}}$ is
identified with $\bar{\psi}(\bar{x}, \bar{x}^\prime) u$ in
$F_{\bar{x}^\prime}$ by $\psi,$ i.e. $\bar{\psi}(\bar{x},
\bar{x}^\prime)$ satisfies $p_\psi (u) = p_\psi (
\bar{\psi}(\bar{x}, \bar{x}^\prime) u )$ where $\Delta$ is the
diagonal. Call $\bar{\psi}$ the \textit{saturation} of $\psi.$ Since
the index $j$ is not used in defining $\bar{\psi},$ it is often
convenient to use $\bar{\psi}$ instead of $\psi.$ Denote by
$\bar{\mathcal{A}}$ the set $\{ \bar{\psi} |~ \psi \in \mathcal{A}
\},$ and call it the \textit{saturation} of $\mathcal{A}.$ And,
denote $F / \psi, p_\psi$ also by $F / \bar{\psi}, p_{\bar{\psi}},$
respectively.

\section{a lemma on fundamental groups}
\label{section: lemmas on fundamental group}

In this section, we recall two lemmas needed to calculate homotopy
of equivariant clutching maps from \cite{Ki}. One is just a
rewriting of Schur's Lemma. Another is on relative homotopy.

\begin{lemma} \label{lemma: reduce to smaller matrix}
For $\chi \in \Irr(H),$ let $W$ be a $\chi$-isotypical
$H$-representation. For the natural inclusion $\imath: \Iso_H (W)
\rightarrow \Iso (W),$ the map
\begin{equation*}
\imath_* : \pi_1 ( \Iso_H (W) ) \rightarrow \pi_1 ( \Iso (W) )
\end{equation*}
is equal to the multiplication by $\chi (id)$ up to sign between two
$\Z$'s.
\end{lemma}

Here, we recall a notation.
\begin{notation}
Let $X$ be a topological space. For two points $y_0$ and $y_1$ in
$X$ and a path $\gamma_0 : [0,1] \rightarrow X$ such that $\gamma_0
(0) = y_0$ and $\gamma_0 (1) = y_1,$ denote by $\overline{\gamma}$
the function defined as
\begin{equation*}
\overline{\gamma}_0 : \pi_1 (X, y_0) \longrightarrow \pi_1 (X, y_1),
\quad \sigma \mapsto \gamma_0^{-1}.\sigma.\gamma_0
\end{equation*}
for $\sigma \in \pi_1 (X, y_0).$
\end{notation}

%

\begin{lemma} \label{lemma: relative homotopy}
Let $X$ be a path connected topological space with an abelian $\pi_1
(X).$ Let $A$ and $B$ be path connected subspaces of $X.$ Also, let
$\imath_0$ and $\imath_1$ denote inclusions from $A$ and $B$ to $X,$
respectively. Pick two points $y_0 \in A$ and $y_1 \in B,$ and a
path $\gamma_0 : [0,1] \rightarrow X$ such that $\gamma_0 (0) = y_0$
and $\gamma_0 (1) = y_1.$ Then, we have an isomorphism
\begin{align*}
\Pi : \pi_1 (X, y_1) / \left\{ \overline{\gamma}_0
\left({\imath_0}_* \pi_1 (A, y_0) \right) +{\imath_1}_* \pi_1 (B,
y_1) \right\} & \longrightarrow [[0,1], 0, 1 ;
X, A, B], \\
[\sigma] & \longmapsto [\gamma_0.\sigma].
\end{align*}
\end{lemma}

\section{Equivariant clutching maps on one-dimensional
fundamental domain}
 \label{section: equivariant clutching maps}

In this section, we find conditions on a preclutching map
$\Phi_{\hat{D}_\rho}$ in $C^0 (\hat{D}_\rho, V_B)$ to guarantee that
$\Phi_{\hat{D}_\rho}$ be the restriction of an equivariant clutching
map. By using this, we show that $\Omega_{\hat{D}_\rho, (W_{d^i})_{i
\in I_\rho^+}}$ is nonempty for each $(W_{d^i})_{i \in I_\rho^+} \in
A_{G_\chi} (\R^2/\Lambda, \chi).$ For these, we define notations on
equivariant pointwise clutching maps with respect to the
$(G_\chi)_x$-bundle $\big( \res_{(G_\chi)_x}^{G_\chi} F_{V_B}
\big)|_{\pi^{-1}(x)}$ for each $\bar{x} \in | \lineL_\rho |$ and $x
= \pi(\bar{x}),$ and prove some lemmas on them. Then, we can apply
results of Section \ref{section: pointwise clutching map} in dealing
with $\Omega_{\hat{D}_\rho, (W_{d^i})_{i \in I_\rho^+}}.$ Concrete
calculation of homotopy of equivariant clutching maps for each case
is done from the next section.

Henceforth, assume that $R = \rho(G_\chi)$ for some $R$ in Table
\ref{table: finite group of aff}. First, we define the set
$\mathcal{A}_{\bar{x}}$ of equivariant pointwise clutching maps for
each $\bar{x}$ in $|\lineL_\rho|.$ For each $\bar{x}$ in
$\lineL_\rho$ and $x = \pi(\bar{x}),$ put $\bar{\mathbf{x}} =
\pi^{-1} (x) = \{ \bar{x}_j | j \in \Z_m \}$ for $m=j_\rho$ or 2
according to whether $\bar{x}$ is a vertex or not. Then, let
$\mathcal{A}_{\bar{x}}$ be the set of equivariant pointwise
clutching maps with respect to the $(G_\chi)_x$-bundle $\big(
\res_{(G_\chi)_x}^{G_\chi} F_{V_B} \big) |_{\bar{\mathbf{x}}}.$
Here, we need to explain for codomain of maps in
$\mathcal{A}_{\bar{x}}.$ For each $\bar{x} \in |\lineL_\rho|$ and
$\psi_{\bar{x}} \in \mathcal{A}_{\bar{x}},$
$\psi_{\bar{x}}(\bar{x}_j)$ is in
\begin{equation*}
\Iso \Big( (F_{V_B})_{\bar{x}_j}, (F_{V_B})_{\bar{x}_{j+1}} \Big)
\end{equation*}
for $j \in \Z_{j_\rho}$ or $\Z_2.$ If $\bar{x}_j \in |\bar{f}|$ and
$\bar{x}_{j+1} \in |\bar{f}^\prime|$ for some $\bar{f}$ and
$\bar{f}^\prime,$ then $\psi_{\bar{x}}(\bar{x}_j)$ is henceforth
regarded as in $\Iso \big( V_{\bar{f}}, V_{\bar{f}^\prime} \big).$
This is justified because we have fixed the trivialization $\big(
\res_{(G_\chi)_{b(\bar{f})}}^{G_\chi} F_{V_B} \big)
\big|_{|\bar{f}|} = |\bar{f}| \times V_{\bar{f}}$ for each face
$\bar{f} \in \lineK_\rho.$ We define one more set
$\mathcal{A}_{\bar{e}}$ of equivariant pointwise clutching maps for
each edge $\hat{e}$ and its images $\bar{e} = p_{\lineL}(\hat{e}),$
$|e|=\pi(|\bar{e}|).$ For each $\hat{x}$ in $|\hat{e}|$ and $\bar{x}
= p_{|\lineL|} (\hat{x}),$ put
\begin{equation*}
\bar{x}_0^\prime = \bar{x}, \quad \bar{x}_1^\prime = p_{|\lineL|}
(|c|(\hat{x})), \quad \text{and} \quad \bar{\mathbf{x}}^\prime=\{
\bar{x}_j^\prime | j \in \Z_2 \}.
\end{equation*}
Consider the set $\mathcal{A}_{\hat{x}}$ of equivariant pointwise
clutching maps with respect to the $(G_\chi)_{|e|}$-bundle $\big(
\res_{(G_\chi)_{|e|}}^{G_\chi} F_{V_B} \big)
|_{\bar{\mathbf{x}}^\prime}$ where $(G_\chi)_{|e|}$ is the subgroup
of $G_\chi$ fixing $|e|.$ As for $\mathcal{A}_{\bar{x}},$ each map
$\psi_{\hat{x}}$ in $\mathcal{A}_{\hat{x}}$ is considered to satisfy
\begin{equation*}
\psi_{\hat{x}} (\bar{x}_0^\prime) \in \Iso(V_{\bar{f}},
V_{\bar{f}^\prime}) \quad \text{and} \quad \psi_{\hat{x}}
(\bar{x}_1^\prime) \in \Iso(V_{\bar{f}^\prime}, V_{\bar{f}})
\end{equation*}
when $\bar{x}_0^\prime \in |\bar{f}|$ and $\bar{x}_1^\prime \in
|\bar{f}^\prime|$ for some $\bar{f}$ and $\bar{f}^\prime.$ Here,
observe that $\mathcal{A}_{\hat{x}}$ is in one-to-one correspondence
with $\mathcal{A}_{\hat{y}}$ for any two $\hat{x}, \hat{y}$ in
$|\hat{e}|,$ i.e. an element $\psi_{\hat{x}}$ in
$\mathcal{A}_{\hat{x}}$ and an element $\psi_{\hat{y}}$ in
$\mathcal{A}_{\hat{y}}$ are corresponded to each other when
$\psi_{\hat{x}} (\bar{x}_j^\prime) = \psi_{\hat{y}}
(\bar{y}_j^\prime)$ for $j \in \Z_2.$ This is because the
$(G_\chi)_{|e|}$-bundle $\big( \res_{(G_\chi)_{|e|}}^{G_\chi}
F_{V_B} \big) |_{\bar{\mathbf{x}}^\prime}$ is all isomorphic
regardless of $\hat{x} \in |\hat{e}|.$ It is very useful to identify
all $\mathcal{A}_{\hat{x}}$'s for $\hat{x} \in |\hat{e}|$ in this
way, so we denote the identified set by $\mathcal{A}_{\bar{e}}.$
That is, each element $\psi_{\bar{e}}$ in $\mathcal{A}_{\bar{e}}$ is
considered as contained to $\mathcal{A}_{\hat{x}}$ for any $\hat{x}
\in \hat{e}$ according to the context.

Next, we would define a $G_\chi$-action on saturations. First, we
define notations on saturations. For each $x=\pi(\bar{x})$ and
$\psi_{\bar{x}} \in \mathcal{A}_{\bar{x}},$ denote saturations of
$\mathcal{A}_{\bar{x}}$ and $\psi_{\bar{x}}$ by
$\bar{\mathcal{A}}_x$ and $\bar{\psi}_x,$ respectively. Since index
set is irrelevant in defining $\bar{\mathcal{A}}_x,$ the saturation
depends not on $\bar{x}$ but on $x.$ This is why we use the
subscript $x$ instead of $\bar{x}.$ For any $g \in G_\chi,$ the
function $g \cdot \bar{\psi}_x$ is contained in
$\bar{\mathcal{A}}_{g x}$ where $g \cdot \bar{\psi}_x$ is defined as
\begin{equation*}
(g \cdot \bar{\psi}_x) ( \bar{y}, \bar{y}^\prime ) = g \bar{\psi}_x
(g^{-1} \bar{y}, g^{-1} \bar{y}^\prime) g^{-1}
\end{equation*}
for any $\bar{y} \ne \bar{y}^\prime$ in $\pi^{-1}(g x).$ That is, we
obtain $g \bar{\mathcal{A}}_x = \bar{\mathcal{A}}_{gx}.$ Especially,
it is easily shown that $g \cdot \bar{\psi}_x = \bar{\psi}_x$ for
each $g \in (G_\chi)_x$ by equivariance of $\psi_{\bar{x}}$ so that
$\bar{\mathcal{A}}_{gx} = \bar{\mathcal{A}}_x.$ From this, it is
noted that if $g^\prime x = gx$ for some $g^\prime, g \in G_\chi,$
then $g^\prime \cdot \bar{\psi}_x = g \cdot \bar{\psi}_x$ because
$g^\prime = g (g^{-1} g^\prime)$ with $g^{-1} g^\prime \in
(G_\chi)_x,$ i.e. $g \cdot \bar{\psi}_x$ is determined by not $g$
but $gx.$ We have defined a $G_\chi$-action on saturations. Since
$\mathcal{A}_{\bar{x}}$ and $\bar{\mathcal{A}}_x$ are in one-to-one
correspondence, $\mathcal{A}_{\bar{x}}$'s also deliver the
$G_\chi$-action induced from $G_\chi$-actions on
$\bar{\mathcal{A}}_x$'s, i.e. $g \cdot \psi_{\bar{x}} \in
\mathcal{A}_{g \bar{x}}$ for each $\psi_{\bar{x}} \in \mathcal{A}_{
\bar{x}}$ is defined and $\mathcal{A}_{g \bar{x}} = g
\mathcal{A}_{\bar{x}}.$ Here, we prove a useful lemma on this
action. Before it, we state an elementary fact.

\begin{lemma}
 \label{lemma: elementary lemma on isotropy representation}
Let $G$ be a compact Lie group acting on a topological space $X$.
And, let $E$ be an equivariant vector bundle over $X.$ Then, $^g E_x
\cong E_{g x}$ for each $g \in G$ and $x \in X.$ Also,
\begin{equation*}
\res_{G_x \cap G_{x^\prime}}^{G_x} E_x \cong \res_{G_x \cap
G_{x^\prime}}^{G_{x^\prime}} E_{x^\prime}
\end{equation*}
for any two points $x, x^\prime$ in $X$ which are in the same
component of the fixed set $X^{G_x \cap G_{x^\prime}}.$
\end{lemma}

Here, we explain for the subscript $^g.$

\begin{definition}
Let $K$ be a closed subgroup of a compact Lie group $G.$  For a
given element $g \in G$ and $W \in \Rep(K),$ the $g K
g^{-1}$-representation $^g W$ is defined to be $V$ with the new $g K
g^{-1}$-action
\begin{equation*}
g K g^{-1} \times W \rightarrow W, \quad (k, u) \mapsto g^{-1} k g u
\end{equation*}
for $k \in g K g^{-1},$ $u \in W.$
\end{definition}

\begin{lemma} \label{lemma: conjugate pointwise clutching}
For each $\bar{x} \in |\lineL_\rho|,$ $x=\pi(\bar{x}),$ $g \in
G_\chi,$ $\psi_{\bar{x}} \in \mathcal{A}_{\bar{x}},$ we have an
$(G_\chi)_{gx}$-isomorphism
\begin{equation*}
\big( \res_{(G_\chi)_{gx}}^{G_\chi} F_{V_B} \big) |_{g
\bar{\mathbf{x}}}
 \Big/ g \cdot \bar{\psi}_x \cong ~ ^g \Big( ~
\big(\res_{(G_\chi)_x}^{G_\chi} F_{V_B} \big) |_{\bar{\mathbf{x}}}
\big/ \bar{\psi}_x ~ \Big).
\end{equation*}
\end{lemma}

\begin{proof}
Put $L= \langle g, (G_\chi)_x \rangle,$ and let $k_0 \in \N$ be the
smallest natural number satisfying $g^{k_0} \in (G_\chi)_x$ where
such a number exists because $R = \rho(G_\chi)$ is finite. To prove
this lemma by Lemma \ref{lemma: elementary lemma on isotropy
representation}, we would construct an $L$-bundle $\mathcal{F}$ over
the orbit $L x = \{ x, gx, \cdots, g^{k_0-1}x \}.$ Put
$\bar{\mathcal{F}} = \big( \res_L^{G_\chi} F_{V_B} \big) |_{L
\bar{\mathbf{x}}}$ over $L \bar{\mathbf{x}},$ and consider the
inequivariant $\mathcal{F}$ over the orbit $L x$ whose fiber
$\mathcal{F}_{g^k x}$ at $g^k x$ is equal to $\bar{\mathcal{F}}
|_{g^k \bar{\mathbf{x}}} \big/ ( g^k \cdot \bar{\psi}_x )$ for $k =
0, \cdots, k_0-1.$ And, let $P : \bar{\mathcal{F}} \rightarrow
\mathcal{F}$ be the map such that the restriction of $P$ to
$\bar{\mathcal{F}} |_{g^k \bar{\mathbf{x}}}$ is equal to $p_{g^k
\cdot \bar{\psi}_x}$ where $p_{g^k \cdot \bar{\psi}_x} :
\bar{\mathcal{F}} |_{g^k \bar{\mathbf{x}}} \rightarrow
\bar{\mathcal{F}} |_{g^k \bar{\mathbf{x}}} \big/ ( g^k \cdot
\bar{\psi}_x )$ is the quotient map of (\ref{figure: diagram of
equivariant pointwise gluing}). Then, we would define an $L$-action
on $\mathcal{F}$ as $l u = P( l \bar{u} )$ for each $l \in L, u \in
\mathcal{F},$ and any $\bar{u} \in P^{-1}(u)$ so that $P$ becomes
$L$-equivariant. As long as this is well defined, it is easily shown
that it becomes an action because it is defined by the $L$-action on
$\bar{\mathcal{F}}$ through $P.$ So, we would prove well definedness
of the action. For this, it suffices to show $P(l \bar{u}) = P(l
\bar{u}^\prime)$ for each $l \in L$ and each $\bar{u},
\bar{u}^\prime$ in $\bar{\mathcal{F}}$ satisfying $P(\bar{u}) =
P(\bar{u}^\prime).$ If $\bar{u} \in (F_{V_B})_{g^k \bar{x}_j}$ and
$\bar{u}^\prime \in (F_{V_B})_{g^k \bar{x}_{j^\prime}}$ for some $j,
j^\prime,$ then $P(\bar{u}) = P(\bar{u}^\prime)$ is written as
\begin{equation}
\tag{*} (g^k \cdot \bar{\psi}_x) ( g^k \bar{x}_j, g^k
\bar{x}_{j^\prime} ) \bar{u} = \bar{u}^\prime.
\end{equation}
Note that $l \bar{u} \in (F_{V_B})_{l g^k \bar{x}_j}$ and $l
\bar{u}^\prime \in (F_{V_B})_{l g^k \bar{x}_{j^\prime}}.$ And, put
$l=g^{k^\prime} l^\prime$ with $l^\prime \in L_{g^k x}$ and some
integer $k^\prime$ so that $l g^k \bar{\mathbf{x}} = g^{k+k^\prime}
\bar{\mathbf{x}}$ because $l^\prime$ fixes $g^k \bar{\mathbf{x}}.$
Remembering that the restriction of $P$ to $\bar{\mathcal{F}} |_{l
g^k \bar{\mathbf{x}}} = \bar{\mathcal{F}} |_{g^{k+k^\prime}
\bar{\mathbf{x}}}$ is equal to $p_{g^{k+k^\prime} \cdot
\bar{\psi}_x},$ $P(l \bar{u}) = P(l \bar{u}^\prime)$ is shown as
\begin{align*}
& (g^{k+k^\prime} \cdot \bar{\psi}_x)
( lg^k \bar{x}_j, lg^k \bar{x}_{j^\prime} ) l \bar{u}   \\
=& g^{k^\prime} (g^k \cdot \bar{\psi}_x)
( g^{-k^\prime} lg^k \bar{x}_j, g^{-k^\prime} lg^k \bar{x}_{j^\prime} )
g^{-k^\prime} l \bar{u}   \\
=& g^{k^\prime} l^\prime l^{\prime -1} (g^k \cdot \bar{\psi}_x) (
l^\prime g^k \bar{x}_j, l^\prime g^k \bar{x}_{j^\prime} )
l^\prime \bar{u}   \\
=& g^{k^\prime} l^\prime (l^{\prime -1} g^k \cdot \bar{\psi}_x) (
g^k \bar{x}_j, g^k \bar{x}_{j^\prime} )
\bar{u}   \\
=& l (g^k \cdot \bar{\psi}_x) ( g^k \bar{x}_j, g^k
\bar{x}_{j^\prime} )
\bar{u}   \\
=& l \bar{u}^\prime
\end{align*}
where we use (*) in the last line. So, we obtain well definedness of
$L$-action on $\mathcal{F}.$ By definition, isotropy representations
$\mathcal{F}_x$ and $\mathcal{F}_{gx}$ are equal to representations
\begin{equation*}
\big(\res_{(G_\chi)_x}^{G_\chi} F_{V_B} \big) |_{\bar{\mathbf{x}}}
\big/ \bar{\psi}_x \quad \text{and} \quad \big(
\res_{(G_\chi)_{gx}}^{G_\chi} F_{V_B} \big) |_{g \bar{\mathbf{x}}}
 \Big/ g \cdot \bar{\psi}_x,
\end{equation*}
respectively. Then, the lemma follows from Lemma \ref{lemma:
elementary lemma on isotropy representation}.
\end{proof}

To investigate $\Omega_{\hat{D}_\rho, (W_{d^i})_{i \in I_\rho^+}},$
we need to understand $\mathcal{A}_{\bar{x}}$ precisely for each
$\bar{x} \in |\lineL_\rho|.$ For this, we need prove a basic lemma
on relations between $\mathcal{A}_{\bar{e}^i}$ and
$\mathcal{A}_{\bar{x}}$ with $\bar{x} = \bar{v}^i$ or
$\bar{v}^{i+1}.$ Also, we prove lemmas on evaluation of equivariant
pointwise clutching maps.

\begin{lemma}
 \label{lemma: inclusions of A_x's}
Put $\bar{\mathbf{v}}^k = \pi^{-1}(v^k) = \{ \bar{x}_j ~ | ~
\bar{x}_j = \bar{v}_j^k \text{ for } j \in \Z_{j_\rho} \}$ for $k=i,
i+1.$ And, put $\bar{\mathbf{v}}^{\prime i} = \{ \bar{v}_0^i,
\bar{v}_1^i \}$ and $\bar{\mathbf{v}}^{\prime i+1} = \{
\bar{v}_0^{i+1}, \bar{v}_{-1}^{i+1} \}.$
\begin{enumerate}
  \item $\mathcal{A}_{\bar{x}} \subset \mathcal{A}_{\bar{e}^i}$
  for each interior $\bar{x}$ in $|\bar{e}^i|.$
  \item $\mathcal{A}_{\bar{e}^i} = \mathcal{A}_{\bar{x}}$
  for each interior $\bar{x} \ne b(\bar{e}^i)$ in $|\bar{e}^i|.$
  Moreover, $\mathcal{A}_{\bar{e}^i} = \mathcal{A}_{b(\bar{e}^i)}$
  if $(G_\chi)_{b(e^i)} = (G_\chi)_x$ for $x=\pi(\bar{x}).$
  \item $\mathcal{A}_{\bar{v}^i}^0 \subset
  \mathcal{A}_{\bar{e}^i}^0$ and $\mathcal{A}_{\bar{v}^{i+1}}^{j_\rho-1} \subset
  \mathcal{A}_{\bar{e}^i}^1.$
  \item For each $\psi_{\bar{v}^i}$ in $\mathcal{A}_{\bar{v}^i},$
  we have $\psi_{\bar{v}^i} (\bar{v}_0^i) = \psi_{\bar{e}^i} (\bar{v}_0^i)$
  for the unique $\psi_{\bar{e}^i} \in \mathcal{A}_{\bar{e}^i},$ and
  \begin{equation*}
  \res_{(G_\chi)_{|e^i|}}^{(G_\chi)_{v^i}}
  \big\{ \big( \res_{(G_\chi)_{v^i}}^{G_\chi}
  F_{V_B} \big) |_{\bar{\mathbf{v}}^i}
  / \psi_{\bar{v}^i} \big\}
  \cong \big( \res_{(G_\chi)_{|e^i|}}^{G_\chi} F_{V_B} \big) |_{\bar{\mathbf{v}}^{\prime i}} /
  \psi_{\bar{e}^i}.
  \end{equation*}
  \item For each $\psi_{\bar{v}^{i+1}}$ in $\mathcal{A}_{\bar{v}^{i+1}},$
  we have $\psi_{\bar{v}^{i+1}} (\bar{v}_{-1}^{i+1}) =
  \psi_{\bar{e}^i} (\bar{v}_{-1}^{i+1})$
  for the unique $\psi_{\bar{e}^i} \in \mathcal{A}_{\bar{e}^i},$ and
  \begin{equation*}
  \res_{(G_\chi)_{|e^i|}}^{(G_\chi)_{v^{i+1}}}
  \big\{ \big( \res_{(G_\chi)_{v^{i+1}}}^{G_\chi} F_{V_B} \big) |_{\bar{\mathbf{v}}^{i+1}}
  / \psi_{\bar{v}^{i+1}} \big\}
  \cong \big( \res_{(G_\chi)_{|e^i|}}^{G_\chi} F_{V_B} \big) |_{\bar{\mathbf{v}}^{\prime i+1}} /
  \psi_{\bar{e}^i}.
  \end{equation*}
\end{enumerate}
\end{lemma}

\begin{proof}
(1) follows from $(G_\chi)_{|e^i|} \subset (G_\chi)_x$ when
$x=\pi(\bar{x}).$ (2) follows from $(G_\chi)_{|e^i|} = (G_\chi)_x.$
Similarly, (3) holds by Lemma \ref{lemma: restricted pointwise
clutching} and $(G_\chi)_{|e^i|} \subset
(G_\chi)_{\bar{\mathbf{v}}^{\prime i}},$ $(G_\chi)_{|e^i|} \subset
(G_\chi)_{\bar{\mathbf{v}}^{\prime i+1}}$ where
$(G_\chi)_{\bar{\mathbf{v}}^{\prime i}}$ and
$(G_\chi)_{\bar{\mathbf{v}}^{\prime i+1}}$ are subgroups of $G_\chi$
preserving $\bar{\mathbf{v}}^{\prime i}$ and
$\bar{\mathbf{v}}^{\prime i+1},$ respectively. The first statement
of (4) follows by (3) and Lemma \ref{lemma: pointwise clutching for
m=2 nontransitive}. By Lemma \ref{lemma: restricted pointwise
clutching}, we have the second statement of (4). Similarly, (5) is
also obtained.
\end{proof}

\begin{lemma} \label{lemma: injective for A_x 1}
For a vertex $\bar{v}$ in $\lineK_\rho$ and $v=\pi(\bar{v}),$ if
$R_v \cong \Z_m$ or $\D_m$ with $m=|\pi^{-1}(v)|,$ then
$\mathcal{A}_{\bar{v}} \rightarrow \mathcal{A}_{\bar{v}}^0$ is
injective.
\end{lemma}
\begin{proof}
Proof is done by Proposition \ref{proposition: psi for cyclic}.(2)
and \ref{proposition: psi for dihedral}.(2).
\end{proof}

\begin{lemma} \label{lemma: injective for A_x 2}
If $R/R_t \ne \langle \id \rangle, \D_1$ with
$A\bar{c}+\bar{c}=\bar{l}_0,$ then the map
\begin{equation*}
\mathcal{A}_{\bar{v}^i} \rightarrow \mathcal{A}_{\bar{v}^i}^0 \times
\mathcal{A}_{\bar{v}^i}^{-1}, ~ \psi_{\bar{v}^i} \mapsto \big( ~
\psi_{\bar{v}^i}(\bar{v}_0^i), ~ \psi_{\bar{v}^i}(\bar{v}_{-1}^i) ~
\big)
\end{equation*}
is injective.
\end{lemma}

\begin{proof}
In these cases, $R_{v^i}$ is not trivial by Table \ref{table:
isotropy of vertex}, and it is checked case by case that $\big( ~
\psi_{\bar{v}^i}(\bar{v}_0^i), ~ \psi_{\bar{v}^i}(\bar{v}_{-1}^i) ~
\big)$ determines $\psi_{\bar{v}^i}$ by using equivariance of
equivariant pointwise clutching maps. In this check, Lemma
\ref{lemma: injective for A_x 1} is helpful.
\end{proof}

Now, we state conditions on a preclutching map $\Phi_{\hat{D}_\rho}$
in $C^0 (\hat{D}_\rho, V_B)$ to guarantee that $\Phi_{\hat{D}_\rho}$
be the restriction of an equivariant clutching map. When $\Lambda_t
= \Lambda^\sq,$ pick two elements $g_0, g_1 \in G_\chi$ such that
\begin{equation*}
\rho(g_0)[\bar{x}] = [ \bar{x}+(1,0)], ~ \rho(g_1)[\bar{x}] = [
\bar{x}+(0,1)].
\end{equation*}
When $R/R_t = \D_1$ with $A\bar{c}+\bar{c}=\bar{l}_0,$ pick an
element $g_2 \in G_\chi$ such that $\rho(g_2)$ is $[
A\bar{x}+\bar{c} ]$ where $A\bar{x}+\bar{c}$ is the glide through
$\frac 1 2 \bar{c} + L$ and $\frac 1 2 (A\bar{c}+\bar{c}).$ Also, we
define a terminology.
\begin{definition}
For $\bar{x} \in |\lineL_\rho|,$ $x=\pi(\bar{x}),$ $\psi_{\bar{x}}
\in \mathcal{A}_{\bar{x}},$ $\Phi \in C^0(|\hatL_\rho|, V_B),$ the
map $\psi_{\bar{x}}$ or its saturation $\bar{\psi}_x$ is called
\textit{determined} by $\Phi$ if $\Phi$ satisfies the following
condition:
\begin{equation*}
\bar{\psi}_x \Big( p_{|\lineL|}(\hat{x}), p_{|\lineL|}(|c|(\hat{x}))
\Big) = \Phi(\hat{x})
\end{equation*}
for each $\hat{x} \in (\pi \circ p_{|\lineL|})^{-1}(x).$ The
condition is concretely written as
\begin{enumerate}
\item if $\bar{x}$ is not a vertex, then
\begin{equation*}
\psi_{\bar{x}} (\bar{x}_0) = \Phi (\hat{x}_{0, +}) \quad \text{and}
\quad \psi_{\bar{x}} (\bar{x}_1) = \Phi (|c|(\hat{x}_{0, +})),
\end{equation*}
\item if $\bar{x}$ is a vertex, then
\begin{equation*}
\psi_{\bar{x}} (\bar{x}_j) = \Phi (\hat{x}_{j,+}) \quad \text{and}
\quad \psi_{\bar{x}}^{-1} (\bar{x}_j) = \Phi (\hat{x}_{j,-})
\end{equation*}
for each $j \in \Z_{j_\rho}.$
\end{enumerate}

\end{definition}

\begin{theorem} \label{theorem: clutching condition}
When $R=\rho(G_\chi)$ for some $R$ in Table \ref{table: finite group
of aff}, a preclutching map $\Phi_{\hat{D}_\rho}$ in $C^0
(\hat{D}_\rho, V_B)$ is in $\Omega_{\hat{D}_\rho, V_B}$ if and only
if there exists the unique $\psi_{\bar{x}} \in
\mathcal{A}_{\bar{x}}$ for each $\bar{x} \in \bar{D}_\rho$ and $x =
\pi(\bar{x})$ satisfying
\begin{enumerate}
  \item[E2.] $\bar{\psi}_x \Big( p_{|\lineL|}(\hat{x}),
  p_{|\lineL|}(|c|(\hat{x})) \Big) = \Phi_{\hat{D}_\rho} (\hat{x})$
  for each $\hat{x} \in \hat{D}_\rho,$
  \item[E3.] for each $\bar{x}, \bar{x}^\prime \in \bar{D}_\rho$
  and their images $x=\pi(\bar{x}), x^\prime=\pi(\bar{x}^\prime),$
  if $x^\prime = gx$ for some $g \in G_\chi,$ then
  $\bar{\psi}_{x^\prime} = g \cdot \bar{\psi}_x.$
\end{enumerate}
\end{theorem}
A set $( \psi_{\bar{x}} )_{\bar{x} \in \bar{D}_\rho}$ with
$\psi_{\bar{x}} \in \mathcal{A}_{\bar{x}}$ is called
\textit{determined} by $\Phi_{\hat{D}_\rho} \in C^0 (\hat{D}_\rho,
V_B)$ if $\Phi_{\hat{D}_\rho}$ is in $\Omega_{\hat{D}_\rho, V_B}$
and each $\psi_{\bar{x}}$ is the unique element of this theorem.


\begin{proof}
For necessity, it suffices to construct a map $\Phi$ in
$\Omega_{V_B}$ such that $\Phi|_{\hat{D}_\rho} =
\Phi_{\hat{D}_\rho}.$ For this, we would define $\Phi$ on
$\hat{\mathbf{D}}_\rho,$ and then extend the domain of definition of
$\Phi$ to the whole $|\hatL_\rho|.$ First, we define $\Phi$ on
$\hat{\mathbf{D}}_\rho$ so that each $\psi_{\bar{x}}$ is determined
by $\Phi.$ Then, Condition E2. says that $\Phi|_{\hat{D}_\rho} =
\Phi_{\hat{D}_\rho}.$ Next, we define $\Phi(\hat{x}) = g^{-1} \Phi(g
\hat{x}) g$ for each $\hat{x} \in |\hatL_\rho|$ and some $g \in
G_\chi$ such that $g \hat{x}$ is in $\hat{\mathbf{D}}_\rho.$ We need
prove well definedness of this. Assume that $\hat{y} = g \hat{x}$
and $\hat{y}^\prime = g^\prime \hat{x}$ are in
$\hat{\mathbf{D}}_\rho$ for two elements $g, g^\prime$ in $G_\chi$
so that $\hat{y}^\prime = g^\prime g^{-1} \hat{y}.$ And, let $y$ and
$y^\prime$ be images of $\hat{y}$ and $\hat{y}^\prime$ through $\pi
\circ p_{|\lineL|},$ respectively. Then, $y^\prime = g^\prime g^{-1}
y.$ These give us $\bar{\psi}_{y^\prime} = (g^\prime g^{-1}) \cdot
\bar{\psi}_y$ by Condition E3. From this, we obtain
\begin{align*}
g^{\prime -1} \Phi(\hat{y}^\prime) g^\prime &= g^{\prime -1}
\bar{\psi}_{y^\prime} \Big(
p_{|\lineL|}(\hat{y}^\prime), p_{|\lineL|}(|c|(\hat{y}^\prime)) \Big) g^\prime \\
&= g^{-1} \bar{\psi}_y \Big( p_{|\lineL|}(\hat{y}),
p_{|\lineL|}(|c|(\hat{y})) \Big) g \\
&= g^{-1} \Phi(\hat{y}) g
\end{align*}
where we use equivariance of $|c|.$ So, well definedness is proved.
It is easily checked that $\Phi$ satisfies Condition N1., N2., E1.
Therefore, $\Phi$ is the wanted equivariant clutching map.

For sufficiency, assume that $\Phi_{\hat{D}_\rho} =
\Phi|_{\hat{D}_\rho}$ for some $\Phi \in \Omega_{V_B}.$ Then, we
should choose the unique $\psi_{\bar{x}} \in \mathcal{A}_{\bar{x}}$
for each $\bar{x} \in \bar{D}_\rho$ satisfying Condition E2. and E3.
When we show that $\psi_{\bar{x}}$'s satisfy Condition E3., it
suffices to consider only vertices $\bar{x}$'s by definition of
$\bar{D}_\rho.$ In $\sim$ entry of Table \ref{table: one-dimensional
fundamental domain}, vertices in the same orbit are listed. Proof
for sufficiency is different according to $R,$ especially choices of
$\psi_{\bar{x}}$'s.

First, assume that $R/R_t$ is not equal to one of following groups:
\begin{equation*}
\D_2 \text{ with } \Lambda_t = \Lambda^\eq, ~ \D_{2,3} \text{ with }
\Lambda_t = \Lambda^\eq, ~ \langle \id \rangle, ~ \D_1, ~ \D_{1,4}.
\end{equation*}
At each $\bar{x} \in \bar{D}_\rho,$ the unique $\psi_{\bar{x}}$ in
$\mathcal{A}_{\bar{x}}$ is determined by $\Phi$ because $\Phi$
satisfies Condition N1., N2., E1. Moreover, it can be checked that
$\psi_{\bar{x}}$ is the unique element satisfying Condition E2. for
each $\bar{x}$ by Lemma \ref{lemma: injective for A_x 2}. So, it
suffices to show that $\psi_{\bar{x}}$'s satisfy Condition E3.
However, since $\Phi$ is equivariant, $\psi_{\bar{x}}$'s satisfy
Condition E3.

Second, assume that $R/R_t = \D_2$ or $\D_{2,3}$ with $\Lambda_t =
\Lambda^\eq.$ Observe that there exists the unique $\psi_{\bar{x}}$
satisfying Condition E2. for each $\bar{x} \in \bar{D}_\rho-\{
\bar{v}^2 \}$ or $\bar{D}_\rho-\{ \bar{v}^1 \}$ according to $R/R_t
= \D_2$ or $\D_{2,3}$ by Lemma \ref{lemma: injective for A_x 2},
respectively. For the remaining one point, pick the unique element
as
\begin{equation*}
\bar{\psi}_{v^2} = g \cdot \bar{\psi}_{v^3}
\end{equation*}
in $\bar{\mathcal{A}}_{v^2}$ by Condition E3. when $R/R_t = \D_2$
and $g v_3 = v_2$ for some $g \in G_\chi,$ or
\begin{equation*}
\bar{\psi}_{v^1} = g^\prime \cdot \bar{\psi}_{v^2}
\end{equation*}
in $\bar{\mathcal{A}}_{v^1}$ by Condition E3. when $R/R_t =
\D_{2,3}$ and $g^\prime v_2 = v_1$ for some $g^\prime \in G_\chi.$
By definition, these two are determined by $\Phi$ because $\Phi$ is
equivariant. So, these two satisfy Condition E2. From this, all
$\psi_{\bar{x}}$'s satisfy Condition E2. By definition, these two
also satisfy Condition E3., and this implies that all
$\psi_{\bar{x}}$'s satisfy Condition E3. by $\sim$ entry of Table
\ref{table: one-dimensional fundamental domain}. Uniqueness of
$\bar{\psi}_{v^2}$ and $\bar{\psi}_{v^1}$ is guaranteed by
definition because $\bar{\psi}_{v^3}$ and $\bar{\psi}_{v^2}$ should
satisfy Condition E3.

Third, assume that $R/R_t = \D_1$ with $A\bar{c}+\bar{c}=0$ or
$\D_{1,4}.$ Observe that there exists the unique $\psi_{\bar{x}}$'
satisfying Condition E2. for each $\bar{x} \in \bar{D}_\rho-\{
\bar{v}^0, \bar{v}^1 \}$ by Lemma \ref{lemma: injective for A_x 2}.
For the remaining two points, pick unique elements
\begin{equation*}
\bar{\psi}_{v^0} = g_0^{-1} \cdot \bar{\psi}_{v^3}, ~
\bar{\psi}_{v^1} = g_0^{-1} \cdot \bar{\psi}_{v^2} \text{ in }
\bar{\mathcal{A}}_{v^0}, ~ \bar{\mathcal{A}}_{v^1}
\end{equation*}
when $R/R_t = \D_1$ with $A\bar{c}+\bar{c}=0,$ and
\begin{equation*}
\bar{\psi}_{v^0} = g_1^{-1} \cdot \bar{\psi}_{v^2}, ~
\bar{\psi}_{v^1} = g_0^{-1} \cdot \bar{\psi}_{v^3} \text{ in }
\bar{\mathcal{A}}_{v^0}, ~ \bar{\mathcal{A}}_{v^1}
\end{equation*}
when $R/R_t = \D_{1,4},$ respectively. Then, the remaining proof is
the same with the second case.

Last, assume that $R/R_t = \langle \id \rangle$ or $\D_1$ with
$A\bar{c}+\bar{c}=\bar{l}_0.$ Observe that there exists the unique
$\psi_{\bar{x}}$ satisfying Condition E2. for each $\bar{x} \in
\bar{D}_\rho-\{ \bar{v}^2, \bar{v}^3, \bar{v}^0 \}$ by Lemma
\ref{lemma: injective for A_x 2}. Since $\bar{D}_\rho-\{ \bar{v}^2,
\bar{v}^3, \bar{v}^0 \}$ has no vertex, these $\psi_{\bar{x}}$'s
satisfy Condition E3. Then, we should show unique existence of
$\psi_{\bar{x}}$ for $\bar{x} \in \{ \bar{v}^2, \bar{v}^3, \bar{v}^0
\}$ to satisfy Condition E2., E3. To show uniqueness, first assume
existence so that all $\psi_{\bar{x}}$'s for $\bar{x} \in \{
\bar{v}^2, \bar{v}^3, \bar{v}^0 \}$ satisfy Condition E2., E3. Then,
Condition E2. says that
\begin{align}
\tag{*} \psi_{\bar{v}^2} (\bar{v}^2) = \Phi_{\hat{D}_\rho}
(\hat{v}_+^2), \quad & \psi_{\bar{v}^3}^{-1} (\bar{v}^3) =
\Phi_{\hat{D}_\rho}
(\hat{v}_-^3), \\
\notag \psi_{\bar{v}^3} (\bar{v}^3) = \Phi_{\hat{D}_\rho}
(\hat{v}_+^3), \quad & \psi_{\bar{v}^0} (\bar{v}^0) =
\Phi_{\hat{D}_\rho} (\hat{v}_-^0).
\end{align}
And, Condition E3. says that that
\begin{align}
\tag{**} \bar{\psi}_{v^2} &= g_1 \cdot \bar{\psi}_{v^3}, \\
\notag \bar{\psi}_{v^0} &= \left\{
  \begin{array}{ll}
    g_0^{-1} \bar{\psi}_{v^3}
    & \text{when } R/R_t = \langle \id \rangle, \\
    (g_2 g_1)^{-1} \bar{\psi}_{v^3}
    & \text{when } R/R_t = \D_1 \text{ with }
    A\bar{c}+\bar{c}=\bar{l}_0.
  \end{array}
\right.
\end{align}
By $\psi_{\bar{v}^2} (\bar{v}^2)$ of (*) and $\bar{\psi}_{v^2}$ of
(**), we have
\begin{equation}
\tag{***} \psi_{\bar{v}^3} (\bar{v}_1^3) = g_1^{-1}
\Phi_{\hat{D}_\rho}(\hat{v}_+^2) g_1
\end{equation}
because $g_1 \bar{v}_1^3 = \bar{v}^2.$ Similarly, by
$\psi_{\bar{v}^0} (\bar{v}^0)$ of (*) and $\bar{\psi}_{v^0}$ of
(**), $\psi_{\bar{v}_3} (\bar{v}_2^3)$ is equal to
\begin{equation}
\tag{****} \left\{
  \begin{array}{ll}
    g_0 \Phi_{\hat{D}_\rho}(\hat{v}_-^0)^{-1} g_0^{-1}
    & \text{when } R/R_t = \langle \id \rangle, \\
    g_2 g_1 \Phi_{\hat{D}_\rho}(\hat{v}_-^0) (g_2 g_1)^{-1}
    & \text{when } R/R_t = \D_1, A\bar{c}+\bar{c}=\bar{l}_0.
  \end{array}
\right.
\end{equation}
By (*), (***), (****), $\psi_{\bar{v}_3}$ is uniquely determined by
values of $\Phi_{\hat{D}_\rho}$ so that $\psi_{\bar{v}_2},
\psi_{\bar{v}_0}$ are also uniquely determined by Condition E3. So,
we obtain uniqueness. Proof of existence is easy. Define
$\psi_{\bar{v}_3}$ as in the proof of uniqueness. By equivariance of
$\Phi,$ we have $\psi_{\bar{v}_3} \in \mathcal{A}_{\bar{v}^3}.$ And,
define $\psi_{\bar{v}_2}$ and $\psi_{\bar{v}_0}$ as in (**) so that
$\psi_{\bar{x}}$'s satisfy Condition E3. Since $\Phi$ is
equivariant, $\psi_{\bar{v}_3}$ is determined by $\Phi$ by
definition of $\psi_{\bar{v}_3}.$ So, $\psi_{\bar{v}_3}$ satisfies
Condition E2. Again by equivariance of $\Phi,$ $\psi_{\bar{v}_2}$
and $\psi_{\bar{v}_0}$ satisfy Condition E2. Therefore, we obtain a
proof.
\end{proof}

\begin{remark}
This theorem holds even though we might omit the word `unique' in
the statement of the theorem because uniqueness is not used in the
proof of necessity.
\end{remark}

By using Theorem \ref{theorem: clutching condition}, we would
describe $\Omega_{\hat{D}_\rho, V_B}$ through
$\mathcal{A}_{\bar{x}}$'s. Define the following set of equivariant
pointwise clutching maps on $\bar{d}^i$'s.

\begin{definition} \label{definition: barA_G}
Denote by $\bar{A}_{G_\chi} (\R^2/\Lambda, V_B)$ the set
\begin{equation*}
\{ (\psi_{\bar{d}^i})_{i \in I_\rho} ~ | ~ \psi_{\bar{d}^i} \in
\mathcal{A}_{\bar{d}^i} \text{ and } \bar{\psi}_{d^{i^\prime}} = g
\cdot \bar{\psi}_{d^i} \text{ if } d^{i^\prime} = g d^i \text{ for
some } g \in G_\chi \}.
\end{equation*}
For $(\psi_{\bar{d}^i})_{i \in I_\rho} \in \bar{A}_{G_\chi}
(\R^2/\Lambda, V_B)$ and $\Phi_{\hat{D}_\rho} \in
\Omega_{\hat{D}_\rho, V_B},$ the element $(\psi_{\bar{d}^i})_{i \in
I_\rho}$ is \textit{determined} by $\Phi_{\hat{D}_\rho}$ if
$\psi_{\bar{d}^i}$'s are unique elements determined by
$\Phi_{\hat{D}_\rho}$ in Theorem \ref{theorem: clutching condition}.
Also, for $(W_{d^i})_{i \in I_\rho^+} \in A_{G_\chi} (\R^2/\Lambda,
\chi)$ and $(\psi_{\bar{d}^i})_{i \in I_\rho} \in \bar{A}_{G_\chi}
(\R^2/\Lambda, V_B),$ the element $(W_{d^i})_{i \in I_\rho^+}$ is
\textit{determined} by $(\psi_{\bar{d}^i})_{i \in I_\rho}$ if
$W_{d^i}$ is determined by $\psi_{\bar{d}^i}$ for each $i \in
I_\rho.$
\end{definition}

\begin{corollary} \label{corollary: Omega}
The set $\Omega_{\hat{D}_\rho, V_B}$ is equal to the set
\begin{align*}
\Big\{ ~ \Phi_{\hat{D}_\rho} & = \bigcup_{i \in I_\rho^-} \varphi^i
\in
C^0 (\hat{D}_\rho, V_B) ~ \Big| \\
& \varphi^i(\hat{d}_+^i) = \psi_{\bar{d}^i} (\bar{d}^i), \quad
\varphi^i(\hat{d}_-^{i+1}) =
\psi_{\bar{d}^{i+1}}^{-1} (\bar{d}^{i+1}), \quad \varphi^i (\hat{x}) \in \mathcal{A}_{\bar{e}^i}^0  \\
& \text{for each } i \in I_\rho^-, ~ \hat{x} \in [\hat{d}_+^i,
\hat{d}_-^{i+1}], \text{ and some } (\psi_{\bar{d}^i})_{i \in
I_\rho} \in \bar{A}_{G_\chi} (\R^2/\Lambda, V_B) ~ \Big\}.
\end{align*}
\end{corollary}

\begin{proof}
To prove this corollary, we would rewrite Theorem \ref{theorem:
clutching condition} by using $\mathcal{A}_{\bar{e}^i}$ and
$\bar{A}_{G_\chi} (\R^2/\Lambda, V_B).$ By Theorem \ref{theorem:
clutching condition}, a preclutching map $\Phi_{\hat{D}_\rho}$ is in
$\Omega_{\hat{D}_\rho, V_B}$ if and only if a set $\Psi = (
\psi_{\bar{x}} )_{\bar{x} \in \bar{D}_\rho}$ is determined by
$\Phi_{\hat{D}_\rho}.$ As we have seen in the proof of the theorem,
$gx=x^\prime$ with $x \ne x^\prime$ in Condition E3. is possible
only when $x$ and $x^\prime$ are $d^i$'s (of course, more precisely
when they are images of vertices). So, $\Psi$ is determined by
$\Phi_{\hat{D}_\rho}$ if and only if $(\psi_{\bar{d}^i})_{i \in
I_\rho}$ satisfies Condition E2. and E3. and $\Psi -
(\psi_{\bar{d}^i})_{i \in I_\rho}$ satisfies Condition E2. Here,
$(\psi_{\bar{d}^i})_{i \in I_\rho}$ satisfies Condition E2. if and
only if
\begin{equation}
\tag{*} \varphi^i(\hat{d}_+^i) = \psi_{\bar{d}^i} (\bar{d}^i) \quad
\text{ and } \quad \varphi^i(\hat{d}_-^{i+1}) =
\psi_{\bar{d}^{i+1}}^{-1} (\bar{d}^{i+1})
\end{equation}
for each $i \in I_\rho^-.$ Lemma \ref{lemma: inclusions of
A_x's}.(3) says that (*) implies $\varphi^i (\hat{x}) \in
\mathcal{A}_{\bar{e}^i}^0$ for $\hat{x} = \hat{d}_+^i,
\hat{d}_-^{i+1}.$ So, (*) could be redundantly rewritten as
\begin{equation}
\tag{**} \varphi^i(\hat{d}_+^i) = \psi_{\bar{d}^i} (\bar{d}^i),
\quad \varphi^i(\hat{d}_-^{i+1}) = \psi_{\bar{d}^{i+1}}^{-1}
(\bar{d}^{i+1}), \quad \text{and } \varphi^i (\hat{x}) \in
\mathcal{A}_{\bar{e}^i}^0
\end{equation}
for $\hat{x} = \hat{d}_+^i, \hat{d}_-^{i+1}.$ And,
$(\psi_{\bar{d}^i})_{i \in I_\rho}$ satisfies Condition E3. if and
only if
\begin{equation}
\tag{***} (\psi_{\bar{d}^i})_{i \in I_\rho} ~ \in ~ \bar{A}_{G_\chi}
(\R^2/\Lambda, V_B).
\end{equation}
Next, we deal with $\psi_{\bar{x}}$'s in $\Psi -
(\psi_{\bar{d}^i})_{i \in I_\rho},$ i.e. $\bar{x}$'s in $(
\bar{d}^i, \bar{d}^{i+1} )$ for some $i \in I_\rho^-.$ They satisfy
Condition E2. if and only if $\psi_{\bar{x}}(\bar{x}) = \varphi^i
(\hat{x}_+)$ for each $\bar{x}$ and $i$ such that $\bar{x} \in
(\bar{d}^i, \bar{d}^{i+1}).$ And, this is satisfied if and only if
$\varphi^i (\hat{x}_+) \in \mathcal{A}_{\bar{x}}^0 =
\mathcal{A}_{\bar{e}^i}^0$ and we have chosen $\psi_{\bar{x}}$'s
such that $\psi_{\bar{x}}(\bar{x}) = \varphi^i (\hat{x}_+)$ for each
$i,$ $\bar{x}.$ In summary, three conditions of this, (**), (***)
are equivalent conditions for $\Psi$ to be determined by
$\Phi_{\hat{D}_\rho}.$ Therefore, we obtain a proof.
\end{proof}

By using this corollary, we would show nonemptiness of
$\Omega_{\hat{D}_\rho, (W_{d^i})_{i \in I_\rho^+}}.$ For this, we
need a lemma.

\begin{lemma}
 \label{lemma: existence of barA}
For each $(W_{d^i})_{i \in I_\rho^+} \in A_{G_\chi} (\R^2/\Lambda,
\chi),$ if we put
\begin{equation*}
V_B = G_\chi \times_{(G_\chi)_{d^{-1}}} W_{d^{-1}}, \quad F_{V_B} =
G_\chi \times_{(G_\chi)_{d^{-1}}} ( |\bar{f}^{-1}| \times W_{d^{-1}}
),
\end{equation*}
then we can pick an element $(\psi_{\bar{d}^i})_{i \in I_\rho}$ in
$\bar{A}_{G_\chi} (\R^2/\Lambda, V_B)$ which determines
$(W_{d^i})_{i \in I_\rho^+}.$
\end{lemma}

\begin{proof}
For each $i \in I_\rho,$ put $\bar{\mathbf{x}} = \pi^{-1}(d^i) = \{
\bar{x}_j | \bar{x}_j = \bar{d}_j^i \text{ for } j \in \Z_m \}$ for
$m = j_\rho$ or 2. Put $F_i = \big( \res_{(G_\chi)_{d^i}}^{G_\chi}
F_{V_B} \big) |_{\bar{\mathbf{x}}},$ and $N_2 = (G_\chi)_{d^i},$
$N_1 = (G_\chi)_{C(\bar{d}^i)},$ $N_0 = H.$ Since $W_{d^{-1}}$ is
$\chi$-isotypical and $G_\chi$ fixes $\chi,$ $(\res_{N_0}^{N_2} F_i)
|_{\bar{x}_j}$'s are all isomorphic regardless of $j.$ So, the
$(G_\chi)_{d^i}$-bundle $F_i$ satisfies the assumption on $F$ in
Section \ref{section: pointwise clutching map}. Note that
\begin{equation*}
(F_i)_{\bar{x}_0} \cong
\res_{(G_\chi)_{C(\bar{d}^i)}}^{(G_\chi)_{d^{-1}}} W_{d^{-1}}
\end{equation*}
because $(G_\chi)_{C(\bar{d}^i)} = (G_\chi)_{\bar{x}_0}.$ This
implies
\begin{equation*}
(F_i)_{\bar{x}_0} \cong
\res_{(G_\chi)_{C(\bar{d}^i)}}^{(G_\chi)_{d^i}} W_{d^i}
\end{equation*}
by Definition \ref{definition: A_G}.(4), i.e. $W_{d^i}$ is a
$(G_\chi)_{d^i}$-extension of $(F_i)_{\bar{x}_0}.$ So, Proposition
\ref{proposition: bijectivity with extensions} says that there
exists an element $\psi_{\bar{d}^i}$ in $\mathcal{A}_{\bar{d}^i}$
which determines $W_{d^i}.$ If $d^{i^\prime} = g \cdot d^i$ for some
$g \in G_\chi,$ then the element $g \cdot \bar{\psi}_{d^i}$ in
$\bar{\mathcal{A}}_{\bar{d}_{i^\prime}^\rho}$ satisfies
\begin{equation*}
F_{i^\prime} \big/ g \cdot \bar{\psi}_{d^i} \cong ~ ^g W_{d^i} \cong
W_{d^{i^\prime}}
\end{equation*}
by Lemma \ref{lemma: conjugate pointwise clutching} and Definition
\ref{definition: A_G}.(3). So, we may assume that
$\bar{\psi}_{d_{i^\prime}^\rho} = g \cdot \bar{\psi}_{d^i}.$ Then,
$(\psi_{\bar{d}^i})_{i \in I_\rho}$ is in $\bar{A}_{G_\chi}
(\R^2/\Lambda, V_B)$ which determines $(W_{d^i})_{i \in I_\rho^+}.$
Therefore, we obtain a proof.
\end{proof}

\begin{proposition} \label{proposition: nonempty omega}
For each $(W_{d^i})_{i \in I_\rho^+} \in A_{G_\chi} (\R^2/\Lambda,
\chi),$ $\Omega_{\hat{D}_\rho, (W_{d^i})_{i \in I_\rho^+}}$ is
nonempty.
\end{proposition}

\begin{proof}
By Lemma \ref{lemma: existence of barA}, we can pick an element
$(\psi_{\bar{d}^i})_{i \in I_\rho}$ in $\bar{A}_{G_\chi}
(\R^2/\Lambda, V_B)$ which determines $(W_{d^i})_{i \in I_\rho^+}.$
For $[\bar{d}^i, \bar{d}^{i+1}] \subset \bar{D}_\rho,$ Lemma
\ref{lemma: inclusions of A_x's} says that both $\psi_{\bar{d}^i}
(\bar{d}^i)$ and $\psi_{\bar{d}^{i+1}}^{-1} (\bar{d}^{i+1})$ are in
$(\mathcal{A}_{\bar{e}^i})^0$ which determine
$\res_{(G_\chi)_{|e^i|}}^{(G_\chi)_{d^i}} W_{d^i}$ and
$\res_{(G_\chi)_{|e^i|}}^{(G_\chi)_{d^{i+1}}} W_{d^{i+1}},$
respectively. Since these two are isomorphic by definition of
$A_{G_\chi} (\R^2/\Lambda, \chi),$ Proposition \ref{proposition:
bijectivity with extensions} says that $\psi_{\bar{d}^i}
(\bar{d}^i)$ and $\psi_{\bar{d}^{i+1}}^{-1} (\bar{d}^{i+1})$ are in
the same component of $(\mathcal{A}_{\bar{e}^i})^0.$ So, we can pick
a function $\varphi^i : [\hat{d}_+^i, \hat{d}_-^{i+1}] \rightarrow
(\mathcal{A}_{\bar{e}^i})^0 \subset \Iso (V_{\bar{f}^{-1}},
V_{\bar{f}^i})$ such that $\varphi^i (\hat{d}_+^i) =
\psi_{\bar{d}^i} (\bar{d}^i)$ and $\varphi^i (\hat{d}_-^{i+1}) =
\psi_{\bar{d}^{i+1}}^{-1} (\bar{d}^{i+1})$ for each $i \in
I_\rho^-.$ From this, $\Phi_{\hat{D}_\rho} = \cup_{i \in I_\rho^-}
\varphi^i$ is in $\Omega_{\hat{D}_\rho, V_B}$ by Corollary
\ref{corollary: Omega}. Therefore, we obtain a proof.
\end{proof}

%
%

\section{The case of $R/R_t \ne \langle \id \rangle, \D_1$
with $A\bar{c}+\bar{c}=\bar{l}_0$}
 \label{section: not id and D_1 cases}

Now, we are ready to calculate homotopy of equivariant clutching
maps.

\begin{theorem} \label{theorem: one point case}
If $\rho(G_\chi)$ is one in Table \ref{table: groups with one
element}, then $\pi_0 ( \Omega_{\hat{D}_\rho, (W_{d^i})_{i \in
I_\rho^+}} )$ is one point set for each element $(W_{d^i})_{i \in
I_\rho^+}$ $\in A_{G_\chi} (\R^2/\Lambda, \chi).$
\begin{table}[!ht]
{\footnotesize
\begin{tabular}{l|c}
$R/R_t$                       & $\Lambda_t$        \\
\hhline{=|=}
 $\D_{2,2}, \D_2, \D_4$                   & $\Lambda^\sq$     \\
\hline
 $\D_2, \D_{2,3}, \D_3, \D_6, \D_{3,2}$   & $\Lambda^\eq$     \\
\hline
 $\D_1$ with $A\bar{c}+\bar{c}=0,$                    & $\Lambda^\sq$     \\
 $\D_{1,4}$                               &
\end{tabular}}
\caption{\label{table: groups with one element} $\rho(G_\chi)$ such
that $\pi_0 ( \Omega_{\hat{D}_\rho, (W_{d^i})_{i \in I_\rho^+}} )$
consists of one element}
\end{table}
\end{theorem}

\begin{proof}
We divide these groups in the table into two categories. The first
are the following six kinds of groups: $\D_{2,2}, \D_2, \D_4$ when
$\Lambda_t = \Lambda^\sq,$ and $\D_3, \D_6, \D_{3,2}$ when
$\Lambda_t = \Lambda^\eq.$ The second are the remaining. We deal
with two categories one after another. When $R=\rho(G_\chi)$ is
given, put
\begin{equation*}
V_B = G_\chi \times_{(G_\chi)_{d^{-1}}} W_{d^{-1}}, \quad F_{V_B} =
G_\chi \times_{(G_\chi)_{d^{-1}}} ( |\bar{f}^{-1}| \times W_{d^{-1}}
)
\end{equation*}
for each $(W_{d^i})_{i \in I_\rho^+}$ $\in A_{G_\chi} (\R^2/\Lambda,
\chi).$

Assume that $R=\rho(G_\chi)$ is in the first category. By Table
\ref{table: one-dimensional fundamental domain}, $\bar{d}^i$'s are
all vertices. For two arbitrary $\Phi_{\hat{D}_\rho} = \bigcup_{i
\in I_\rho^-} \varphi^i$ and $\Phi_{\hat{D}_\rho}^\prime =
\bigcup_{i \in I_\rho^-} \varphi^{\prime i}$ in
$\Omega_{\hat{D}_\rho, (W_{d^i})_{i \in I_\rho^+}},$ let
$(\psi_{\bar{d}^i})_{i \in I_\rho}$ and
$(\psi_{\bar{d}^i}^\prime)_{i \in I_\rho}$ in $\bar{A}_{G_\chi}
(\R^2/\Lambda, V_B)$ be two elements determined by
$\Phi_{\hat{D}_\rho}$ and $\Phi_{\hat{D}_\rho}^\prime,$
respectively. We would construct a homotopy connecting
$\Phi_{\hat{D}_\rho}$ and $\Phi_{\hat{D}_\rho}^\prime$ in
$\Omega_{\hat{D}_\rho, (W_{d^i})_{i \in I_\rho^+}}.$ First, we show
that we may assume that $(\psi_{\bar{d}^i}^\prime)_{i \in I_\rho} =
(\psi_{\bar{d}^i})_{i \in I_\rho}.$ Since $\psi_{\bar{d}^i}$ and
$\psi_{\bar{d}^i}^\prime$ for $i \in I_\rho$ determine the same
representation $W_{d^i},$ these two are in the same path component
of $\mathcal{A}_{\bar{d}^i}$ by Proposition \ref{proposition:
bijectivity with extensions}. Take paths $\gamma^i : [0,1]
\rightarrow \mathcal{A}_{\bar{d}^i}$ for $i \in I_\rho$ such that
$\gamma^i(0) = \psi_{\bar{d}^i}^\prime$ and $\gamma^i(1) =
\psi_{\bar{d}^i}.$ If $d^{i^\prime} = g \cdot d^i$ for some $g \in
G_\chi,$ then $g \cdot \gamma^i$ satisfies $(g \cdot \gamma^i)(0) =
\psi_{\bar{d}^{i^\prime}}^\prime$ and $(g \cdot \gamma^i)(1) =
\psi_{\bar{d}^{i^\prime}}$ by Definition \ref{definition: barA_G}.
So, we may assume that $\gamma^{i^\prime} = g \cdot \gamma^i$ if
$d^{i^\prime} = g \cdot d^i.$ Here, these $\gamma^i$'s are well
defined because elements of $\mathcal{A}_{\bar{d}^i}$ are fixed by
$(G_\chi)_{d^i}.$ Recall that the parametrization on $\hat{e}^i =
[\hat{v}_+^i, \hat{v}_-^{i+1}]$ by $s \in [0,1]$ satisfies
$\hat{v}_+^i = 0,$ $b(\hat{e}^i) = 1/2,$ $\hat{v}_-^{i+1} = 1$ for
$i \in I_\rho^-.$ Then, we construct a homotopy $L^i (s,t) :
[\hat{d}_+^i, \hat{d}_-^{i+1}] \times [0,1] \rightarrow
\mathcal{A}_{\bar{e}^i}^0$ for each $i \in I_\rho^-$ as
\begin{equation*}
  \begin{array}{ll}
L^i (s,t) = \gamma^i ( (1-3s) t ) (\bar{d}^i) &
\text{for } s \in [0, \frac 1 3], \\
L^i (s,t) = \varphi^{\prime i} (3s-1) & \text{for } s \in [ \frac 1
3, \frac 2 3 ], \\
L^i (s,t) = \gamma^{i+1} ((3s-2)t)^{-1} (\bar{d}^{i+1})
& \text{for } s \in [\frac 2 3, 1]. \\
  \end{array}
\end{equation*}
If we put $L = \cup_{i \in I_\rho^-} L^i,$ then $L_t$ for each $t
\in [0,1]$ is in $\Omega_{\hat{D}_\rho, (W_{d^i})_{i \in I_\rho^+}}$
by Corollary \ref{corollary: Omega}, and $L$ connects
$\Phi_{\hat{D}_\rho}^\prime$ with $L_1$ in $\Omega_{\hat{D}_\rho,
(W_{d^i})_{i \in I_\rho^+}}$ which determines $(\psi_{\bar{d}^i})_{i
\in I_\rho}.$ So, we may put $\Phi_{\hat{D}_\rho}^\prime = L_1,$ and
we may assume that $\Phi_{\hat{D}_\rho}$ and
$\Phi_{\hat{D}_\rho}^\prime$ determine the same element
$(\psi_{\bar{d}^i})_{i \in I_\rho}$ in $\bar{A}_{G_\chi}
(\R^2/\Lambda, V_B).$

As in the proof of Proposition \ref{proposition: nonempty omega},
$\psi_{\bar{d}^i} (\bar{d}^i)$ and $\psi_{\bar{d}^{i+1}}^{-1}
(\bar{d}^{i+1})$ are in the same component
$(\mathcal{A}_{\bar{e}^i})_{\psi_{\bar{e}^i}}^0$ for some element
$\psi_{\bar{e}^i}$ in $\mathcal{A}_{\bar{e}^i}.$ Since $\varphi^i$
and $\varphi^{\prime i}$ have values in $\mathcal{A}_{\bar{e}^i}^0$
by Corollary \ref{corollary: Omega} and they connect
$\psi_{\bar{d}^i} (\bar{d}^i)$ with $\psi_{\bar{d}^{i+1}}^{-1}
(\bar{d}^{i+1}),$ $\varphi^i$ and $\varphi^{\prime i}$ have values
in $(\mathcal{A}_{\bar{e}^i})_{\psi_{\bar{e}^i}}^0.$ Here, note that
$(\mathcal{A}_{\bar{e}^i})_{\psi_{\bar{e}^i}}$ is simply connected
by Table \ref{table: isotropy of edge}, Lemma \ref{lemma: isotropy
at a point in an edge}, and Proposition \ref{proposition: psi for
cyclic}.(3). By simply connectedness, we can obtain a homotopy
$L^{\prime i} (s,t) : [\hat{d}_+^i, \hat{d}_-^{i+1}] \times [0,1]
\rightarrow \mathcal{A}_{\bar{e}^i}^0$ for $i \in I_\rho^-$ as
\begin{equation*}
\begin{array}{lll}
L^{\prime i} (s,t) \in
(\mathcal{A}_{\bar{e}^i})_{\psi_{\bar{e}^i}}^0, \quad & L^{\prime i}
(0,t) = \psi_{\bar{d}^i} (\bar{d}^i), \quad & L^{\prime i} (1,t) =
\psi_{\bar{d}^{i+1}}^{-1} (\bar{d}^{i+1}), \\
L^{\prime i} (s,0) = \varphi^i (s), \quad & L^{\prime i} (s,1) =
\varphi^{\prime i} (s) \quad &
\end{array}
\end{equation*}
for $s,t \in [0,1].$ Then, $L^\prime = \cup_{i \in I_\rho^-}
L^{\prime i}$ connects $\Phi_{\hat{D}_\rho}$ and
$\Phi_{\hat{D}_\rho}^\prime$ in $\Omega_{\hat{D}_\rho, (W_{d^i})_{i
\in I_\rho^+}}$ by Corollary \ref{corollary: Omega}. Therefore, we
obtain a proof for the first category. Here, we remark that simply
connectedness is critical in obtaining $L^\prime.$

Next, we deal with $R=\rho(G_\chi)$'s in the second category. Assume
that $R/R_t = \D_2$ with $\Lambda_\eq.$ The calculation for this
case also applies for other groups in the second category. In this
case, we have
\begin{equation*}
\bar{D}_\rho = [b(\bar{e}_2), \bar{v}_3, \bar{v}_0, \bar{v}_1,
\bar{v}_2] \text{ with } v^2 \sim v^3, \quad I_\rho = \{ 2, 3, 4, 5,
6 \},
\end{equation*}
and
\begin{equation*}
\begin{array}{lll}
(G_\chi)_{b(e^0)}/H = \D_{1, 2}, & (G_\chi)_{b(e^1)}/H = \D_1, &
(G_\chi)_{b(e^2)}/H = \Z_2, \\
(G_\chi)_{b(e^3)}/H = \D_1, & (G_\chi)_{v^0}/H = \D_2, &
(G_\chi)_{v^1}/H = \D_2, \\ (G_\chi)_{v^2}/H = \D_1, &
(G_\chi)_{v^3}/H = \D_1 &
\end{array}
\end{equation*}
by Table \ref{table: one-dimensional fundamental domain} and tables
in Section \ref{section: equivariant simplicial complex}. From
these, we especially obtain that path components of
$\mathcal{A}_{\bar{e}^i}$ for $i=0,1,3$ are all simply connected by
Proposition \ref{proposition: psi for cyclic}.(3) so that we would
focus on $[b(\bar{e}_2), \bar{v}_3]$ to calculate
$\Omega_{\hat{D}_\rho, (W_{d^i})_{i \in I_\rho^+}}.$ As in the proof
for the first category, we may assume that two arbitrary
$\Phi_{\hat{D}_\rho} = \bigcup_{i \in I_\rho^-} \varphi^i$ and
$\Phi_{\hat{D}_\rho}^\prime = \bigcup_{i \in I_\rho^-}
\varphi^{\prime i}$ in $\Omega_{\hat{D}_\rho, (W_{d^i})_{i \in
I_\rho^+}}$ determine the same element $(\psi_{\bar{d}^i})_{i \in
I_\rho}$ in $\bar{A}_{G_\chi} (\R^2/\Lambda, V_B).$ We would
construct a homotopy $L^{\prime i} (s, t) : [\hat{d}_+^i,
\hat{d}_-^{i+1}] \times [0,1] \rightarrow \mathcal{A}_{\bar{e}^i}^0$
for $i \in I_\rho^-$ such that $L^\prime = \cup L^{\prime i}$
satisfies
\begin{enumerate}
  \item $L_t^\prime \in \Omega_{\hat{D}_\rho, (W_{d^i})_{i \in I_\rho^+}},$
  \item $L_0^\prime = \Phi_{\hat{D}_\rho}$
  and $L_1^\prime = \Phi_{\hat{D}_\rho}^\prime.$
\end{enumerate}
First, we define $L^{\prime 2} (s,t) :
\partial([\hat{d}_+^2, \hat{d}_-^{3}] \times [0,1])
\rightarrow \mathcal{A}_{\bar{e}^2}^0$ as
\begin{equation*}
\begin{array}{ll}
L^{\prime 2} (1 / 2, t) = \psi_{\bar{d}^2} (\bar{d}^2)  & \text{for
} t \in [0, 1],    \\
L^{\prime 2} (s, 0)     = \varphi^2 (s)                 & \text{for
} s \in [1/2, 1],  \\
L^{\prime 2} (s, 1)     = \varphi^{\prime 2} (s)        & \text{for
} s \in [1/2, 1],  \\
L^{\prime 2} (1, t)     = \varphi^2 (1-t)               & \text{for
} t \in [0, 1/2],
\\
L^{\prime 2} (1, t)     = \varphi^{\prime 2} (t)                 &
\text{for } t \in [1/2,1].
\end{array}
\end{equation*}
Then, we can extend $L^{\prime 2}$ to $[\hat{d}_+^2, \hat{d}_-^{3}]
\times [0,1]$ because $L^{\prime 2}$ on the boundary is trivial in
$\pi_1 (\mathcal{A}_{\bar{e}^2}^0).$ Here, note that $L^{\prime
2}(1,t)^{-1}$ is in $\mathcal{A}_{\bar{v}_3}^3 = \Iso_H
(V_{\bar{f}^2}, V_{\bar{f}^{-1}}).$ If we apply Proposition
\ref{proposition: psi for Z_2}.(1) with $F =
F_{V_B}|_{\pi^{-1}(v^3)}$ and $j=1, j^\prime=0, j^{\prime
\prime}=3,$ then $L^{\prime 2} (1, t)$ and $\res_{0,1}
\psi_{\bar{d}^3}$ give the unique element in
$\mathcal{A}_{\bar{v}_3}$ for each $t \in [0,1],$ say
$\psi_{\bar{d}^3}^t.$ If we put $\psi_{\bar{d}^6}^t = g \cdot
\psi_{\bar{d}^3}^t$ where $g v^3 = v^2$ for some $g \in G_\chi,$
then five maps $\psi_{\bar{d}^2},$ $\psi_{\bar{d}^3}^t,$
$\psi_{\bar{d}^4},$ $\psi_{\bar{d}^5},$ $\psi_{\bar{d}^6}^t$ consist
of an element of $\bar{A}_{G_\chi} (\R^2/\Lambda, V_B),$ call it
$\Psi^t,$ which determines $(W_{d^i})_{i \in I_\rho^+}$ regardless
of $t.$ With these, define other $L^{\prime i}$'s on
$\partial([\hat{d}_+^i, \hat{d}_-^{i+1}] \times [0,1])$ as
\begin{align*}
L^{\prime i} (s, 0)     &= \varphi^i (s),      \\
L^{\prime i} (s, 1)     &= \varphi^{\prime i} (s)      \\
\end{align*}
for $i = 3, 4, 5$ and
\begin{align*}
L^{\prime 3} (0, t)     &= \psi_{\bar{d}^3}^t (\bar{d}^3),     \\
L^{\prime 3} (1, t)     &= \psi_{\bar{d}^4}^{-1} (\bar{d}^4),     \\
L^{\prime 4} (0, t)     &= \psi_{\bar{d}^4} (\bar{d}^4),     \\
L^{\prime 4} (1, t)     &= \psi_{\bar{d}^5}^{-1} (\bar{d}^5),     \\
L^{\prime 5} (0, t)     &= \psi_{\bar{d}^5} (\bar{d}^5),     \\
L^{\prime 5} (1, t)     &= \psi_{\bar{d}^6}^{t ~ -1} (\bar{d}^6)
\end{align*}
for $s, t \in [0,1].$ We can extend these $L^{\prime i}$'s to
$[\hat{d}_+^i, \hat{d}_-^{i+1}] \times [0,1]$ because path
components of $\mathcal{A}_{\bar{e}^i}$ for $i=0,1,3$ are all simply
connected. Then, it can be checked that $L_t^\prime$ determines
$\Psi^t$ and $L_t^\prime$ is in $\Omega_{\hat{D}_\rho, (W_{d^i})_{i
\in I_\rho^+}}$ by Corollary \ref{corollary: Omega}. So, $L^\prime$
is a wanted homotopy and we obtain a proof.

\end{proof}

This theorem gives a proof of Theorem \ref{main: only by isotropy}.
\begin{proof}[Proof of Theorem \ref{main: only by isotropy}]
By Theorem \ref{theorem: one point case} and Lemma \ref{lemma: lemma
for isomorphism}.(2), we obtain a proof.
\end{proof}

\begin{theorem} \label{theorem: by isotropy and chern}
If $\rho(G_\chi)$ is one in Table \ref{table: orientable rho(G)},
then
\begin{equation*}
p_\vect \times c_1 : \Vect_{G_\chi} (\R^2/\Lambda, \chi) \rightarrow
A_{G_\chi} ( \R^2/\Lambda, \chi ) \times H^2 (\R^2/\Lambda, \chi)
\end{equation*}
is injective. For each element $(W_{d^i})_{i \in I_\rho^+}$ in
$A_{G_\chi} ( \R^2/\Lambda, \chi ),$ the set of the first Chern
classes of bundles in the preimage of the element under $p_{\vect}$
is equal to the set $\{ \chi( \id ) ( l_\rho k + k_0 ) ~ | ~ k \in
\Z \}$ where $k_0$ is dependent on $(W_{d^i})_{i \in I_\rho^+}.$

\begin{table}[!ht]
{\footnotesize
\begin{tabular}{l|c}
$R/R_t$                       & $\Lambda_t$        \\
\hhline{=|=}
 $\Z_2, \Z_4$        & $\Lambda^\sq$     \\
\hline
 $\Z_6, \Z_3$                             & $\Lambda^\eq$
\end{tabular}}
\caption{\label{table: orientable rho(G)} $R$'s with nontrivial
cyclic $R/R_t$ }
\end{table}
\end{theorem}

\begin{proof}

When $R=\rho(G_\chi)$ is given, put
\begin{equation*}
V_B = G_\chi \times_{(G_\chi)_{d^{-1}}} W_{d^{-1}}, \quad F_{V_B} =
G_\chi \times_{(G_\chi)_{d^{-1}}} ( |\bar{f}^{-1}| \times W_{d^{-1}}
)
\end{equation*}
for each $(W_{d^i})_{i \in I_\rho^+}$ $\in A_{G_\chi} (\R^2/\Lambda,
\chi).$ We only deal with the case of $R/R_t = \Z_2$ and other cases
are calculated in a similar way. In this case, we have
\begin{equation*}
\begin{array}{lll}
\bar{D}_\rho = [b(\bar{e}^1), \bar{v}^2, \bar{v}^3, b(\bar{e}^3)], &
(G_\chi)_{b(e^1)}/H = \Z_2, & (G_\chi)_{b(e^2)}/H = \langle \id \rangle, \\
(G_\chi)_{b(e^3)}/H = \Z_2, & (G_\chi)_{v^2}/H = \Z_2, &
(G_\chi)_{v^3}/H = \Z_2
\end{array}
\end{equation*}
by Table \ref{table: one-dimensional fundamental domain} and tables
in Section \ref{section: equivariant simplicial complex}. And, no
two vertices $v^i$ and $v^{i^\prime}$ are in the same orbit by Table
\ref{table: one-dimensional fundamental domain}. By using these,
path components of
\begin{equation*}
\mathcal{A}_{b(\bar{e}^1)}, ~ \mathcal{A}_{b(\bar{e}^3)}, ~
(\mathcal{A}_{\bar{v}^2})_{j, j+2}, ~ (\mathcal{A}_{\bar{v}^3})_{j,
j+2}
\end{equation*}
for $j \in \Z_4$ are all simply connected by Proposition
\ref{proposition: psi for cyclic} and \ref{proposition: psi for
Z_2}, and we have $\mathcal{A}_{\bar{e}^2}^0 = \Iso_H
(V_{\bar{f}^{-1}}, V_{\bar{f}^2})$ by Lemma \ref{lemma: pointwise
clutching for m=2 nontransitive}. We will use these without
reference in this proof.

Pick an element $(\psi_{\bar{d}^i})_{i \in I_\rho}$ in
$\bar{A}_{G_\chi} (\R^2/\Lambda, V_B)$ which determines
$(W_{d^i})_{i \in I_\rho^+}$ by Lemma \ref{lemma: existence of
barA}. Here, we may assume that $\psi_{\bar{v}^2}^{-1} (\bar{v}^2) =
\psi_{b(\bar{e}^1)}(b(\bar{e}^1))$ and $\psi_{\bar{v}^3} (\bar{v}^3)
= \psi_{b(\bar{e}^3)}(b(\bar{e}^3)).$ We explain for this. By
Proposition \ref{proposition: psi for Z_2}, the map
\begin{equation*}
\mathcal{A}_{\bar{v}^2} \rightarrow
(\mathcal{A}_{\bar{v}^2})_{0,3}^0 \times
(\mathcal{A}_{\bar{v}^2})_{1,3}^3 ~, \quad \psi \mapsto \Big(
\psi^{-1}(\bar{v}^2), \big( \psi(\bar{v}_2^2) \psi(\bar{v}_1^2)
\big)^{-1} \Big)
\end{equation*}
is homeomorphic, and $(\mathcal{A}_{\bar{v}^2})_{0,3}^0 = \Iso_H
(V_{\bar{f}^{-1}}, V_{\bar{f}^1}) = \mathcal{A}_{\bar{e}^1}^0$ is
path-connected so that elements of each component of
$\mathcal{A}_{\bar{v}^2}$ can take any values of
$(\mathcal{A}_{\bar{v}^2})_{0,3}^0.$ Also, since
$\mathcal{A}_{b(\bar{e}^1)} \subset \mathcal{A}_{\bar{e}^1}$ by
Lemma \ref{lemma: restricted pointwise clutching}.(1), we may take
an element $\psi_{\bar{v}^2}^\prime$ in the component
$(\mathcal{A}_{\bar{v}^2})_{\psi_{\bar{v}^2}}$ such that
$\psi_{\bar{v}^2}^{\prime -1} (\bar{v}^2) = \psi_{b(\bar{e}^1)}
(b(\bar{e}^1)).$ So, we may assume that
\begin{equation}
\tag{*} \psi_{\bar{v}^2}^{-1} (\bar{v}^2) = \psi_{b(\bar{e}^1)}
(b(\bar{e}^1)).
\end{equation}
In the same reason, we may assume that
\begin{equation}
\tag{**} \psi_{\bar{v}^3} (\bar{v}^3) = \psi_{b(\bar{e}^3)}
(b(\bar{e}^3)).
\end{equation}
As in the proof of Theorem \ref{theorem: one point case}, each
$\Phi_{\hat{D}_\rho} = \cup_{i \in I_\rho^-} \varphi^i$ in
$\Omega_{\hat{D}_\rho, (W_{d^i})_{i \in I_\rho^+}}$ is homotopic to
an element which determines $(\psi_{\bar{d}^i})_{i \in I_\rho}.$ So,
we may assume that each $\Phi_{\hat{D}_\rho}$ determines
$(\psi_{\bar{d}^i})_{i \in I_\rho}$ so that it satisfies
\begin{equation}
\tag{***} \varphi^1 (1/2) = \varphi^1(1) = \psi_{b(\bar{e}^1)}
(b(\bar{e}^1)) , \quad \varphi^3(0) = \varphi^3 (1/2) =
\psi_{b(\bar{e}^3)} (b(\bar{e}^3))
\end{equation}
by (*) and (**).

Let $\imath_2$ be the map from $\Omega_2$ to $\Omega_{\hat{D}_\rho,
(W_{d^i})_{i \in I_\rho^+}}$ which sends an arbitrary
$\varphi^{\prime 2}$ to the map $\Phi_{\hat{D}_\rho} = \cup_{i \in
I_\rho^-} \varphi^i$ such that
\begin{equation*}
\varphi^1 \equiv \psi_{b(\bar{e}^1)} (b(\bar{e}^1)), \quad \varphi^2
= \varphi^{\prime 2}, \quad \varphi^3 \equiv \psi_{b(\bar{e}^3)}
(b(\bar{e}^3))
\end{equation*}
where $\Omega_2$ is defined as
\begin{equation*}
\big\{ \varphi^{\prime 2} \in C^0 \big([\hat{d}_+^2, \hat{d}_-^3 ],
\Iso_H (V_{\bar{f}^{-1}}, V_{\bar{f}^2}) \big) ~ | ~
 \varphi^{\prime 2} ( \hat{d}_+^2 ) = \psi_{\bar{v}^2} (\bar{v}_0^2) \text{
and } \varphi^{\prime 2} ( \hat{d}_-^3 ) = \psi_{\bar{v}^3}^{-1}
(\bar{v}_0^3) \big\}.
\end{equation*}
To calculate $\pi_0 ( \Omega_{\hat{D}_\rho, (W_{d^i})_{i \in
I_\rho^+}} ),$ we would show that
\begin{equation*}
\pi_0 (\imath_2) : \pi_0 (\Omega_2) \rightarrow \pi_0
(\Omega_{\hat{D}_\rho, (W_{d^i})_{i \in I_\rho^+}})
\end{equation*}
is bijective because it is is easy to calculate $\pi_0 (\Omega_2).$

First, we would show that each $\Phi_{\hat{D}_\rho} = \cup_{i \in
I_\rho^-} \varphi^i$ satisfying (***) is homotopic to an element
$\Phi_{\hat{D}_\rho}^\prime = \cup_{i \in I_\rho^-} \varphi^{\prime
i}$ in $\Omega_{\hat{D}_\rho, (W_{d^i})_{i \in I_\rho^+}}$ such that
$\varphi^{\prime 1} \equiv \psi_{b(\bar{e}^1)} (b(\bar{e}^1)),$
$\varphi^{\prime 3} \equiv \psi_{b(\bar{e}^3)} (b(\bar{e}^3)),$ and
$\Phi_{\hat{D}_\rho}^\prime$ determines $(\psi_{\bar{d}^i})_{i \in
I_\rho},$ i.e. surjectiveness of $\pi_0( \imath_2 ).$ For this, we
construct a homotopy $L^i (s,t) : [\hat{d}_+^i, \hat{d}_-^{i+1}]
\times [0,1] \rightarrow \mathcal{A}_{\bar{e}^i}^0$ for $i \in
I_\rho^-$ such that the homotopy $L = \cup_{i \in I_\rho^-} L^i$
satisfies
\begin{enumerate}
  \item $L_t \in \Omega_{\hat{D}_\rho, (W_{d^i})_{i \in I_\rho^+}}$
  for each $t \in [0,1],$
  \item $L_0 = \Phi_{\hat{D}_\rho},$ and $L_1$ is
  constant except $[\hat{d}_+^2,
\hat{d}_-^3],$
  \item $L_1$ determines
$(\psi_{\bar{d}^i})_{i \in I_\rho}.$
\end{enumerate}
Define $L^i (s,t) :
\partial([\hat{d}_+^i, \hat{d}_-^{i+1}] \times [0,1])
\rightarrow \mathcal{A}_{\bar{e}^i}^0$ for $i =1, 3$ as
\begin{equation*}
\begin{array}{ll}
L^1 (\frac 1 2,t) = L^1 (s,1) = \varphi^1 (\frac 1 2)   & \text{ for } s \in [\frac 1 2, 1], \\
L^1 (s,0)         = \varphi^1 (s)                       & \text{ for } s \in [\frac 1 2, 1], \\
L^1 (1,t)         = \varphi^1 (1-\frac 1 2 t),          &                \\
L^3 (\frac 1 2,t) = L^3 (s,1) = \varphi^3 (\frac 1 2)   & \text{ for } s \in [0, \frac 1 2], \\
L^3 (s,0)         = \varphi^3 (s)                       & \text{ for } s \in [0, \frac 1 2], \\
L^3 (0,t)         = \varphi^3 (\frac 1 2 t)             &                \\
\end{array}
\end{equation*}
for $t \in [0,1],$ and define a homotopy $L^i (s,t) : [\hat{d}_+^i,
\hat{d}_-^{i+1}] \times [0,1] \rightarrow \mathcal{A}_{\bar{e}^i}^0$
for $i =2$ as
\begin{equation*}
  \begin{array}{ll}
L^2 (s,t) = \gamma_2 ( (1 - 3s) t ) &
\text{for } s \in [0, \frac 1 3], \\
L^2 (s,t) = \varphi^2 (3s-1) & \text{for } s \in [\frac 1 3, \frac 2
3], \\
L^2 (s,t) = \gamma_3 ( (3s-2) t ) & \text{for } s \in [\frac 2 3,
1] \\
  \end{array}
\end{equation*}
where
\begin{equation*}
\begin{array}{ll}
\gamma_2 : [0,1] \rightarrow \Iso_H (V_{\bar{f}^{-1}},
V_{\bar{f}^2}), \quad & t \mapsto \Big(
\psi_{\bar{v}^2}(\bar{v}_2^2) \psi_{\bar{v}^2}(\bar{v}_1^2)
\Big)^{-1} \varphi^1 (1-\frac 1 2 t), \\
\gamma_3 : [0,1] \rightarrow \Iso_H (V_{\bar{f}^{-1}},
V_{\bar{f}^2}), \quad & t \mapsto \psi_{\bar{v}^3}(\bar{v}_2^3)
\psi_{\bar{v}^3}(\bar{v}_1^3) \varphi^3 (\frac 1 2 t).
\end{array}
\end{equation*}
By (***), $L^i$'s are well defined, and we can easily extend $L^i$'s
to the whole $[\hat{d}_+^i, \hat{d}_-^{i+1}] \times [0,1].$ If we
define pointwise preclutching maps $\psi_{\bar{v}^2}^t,$
$\psi_{\bar{v}^3}^t$ for $t \in [0,1]$ as
\begin{equation*}
\begin{array}{ll}
\psi_{\bar{v}^2}^t (\bar{v}_0^2) = \gamma_2(t), \quad &
\psi_{\bar{v}^2}^t (\bar{v}_1^2) =
\psi_{\bar{v}^2}(\bar{v}_1^2), \\
\psi_{\bar{v}^2}^t (\bar{v}_2^2) = \psi_{\bar{v}^2}(\bar{v}_2^2),
\quad & \psi_{\bar{v}^2}^t
(\bar{v}_3^2) = \varphi^1 (1 - \frac 1 2 t)^{-1}, \\
\psi_{\bar{v}^3}^t (\bar{v}_0^3) = \varphi^3 (\frac 1 2 t), \quad &
\psi_{\bar{v}^3}^t (\bar{v}_1^3) =
\psi_{\bar{v}^3}(\bar{v}_1^3), \\
\psi_{\bar{v}^3}^t (\bar{v}_2^3) = \psi_{\bar{v}^3}(\bar{v}_2^3),
\quad & \psi_{\bar{v}^3}^t (\bar{v}_3^3) = \gamma_3(t)^{-1},
\end{array}
\end{equation*}
then they are contained in $\mathcal{A}_{\bar{v}^2},$
$\mathcal{A}_{\bar{v}^3}$ by Proposition \ref{proposition: psi for
Z_2}, respectively. And, $L_t$ determines $\{ \psi_{b(\bar{e}^1)},
\psi_{\bar{v}^2}^t, \psi_{\bar{v}^3}^t, \psi_{b(\bar{e}^3)} \}$ and
$L_t$ is contained in $\Omega_{\hat{D}_\rho, (W_{d^i})_{i \in
I_\rho^+}}$ by Corollary \ref{corollary: Omega}. So, it is easily
checked that $L$ is a wanted homotopy.

Next, for two $\Phi_{\hat{D}_\rho} = \cup_{i \in I_\rho^-}
\varphi^i,$ $\Phi_{\hat{D}_\rho}^\prime = \cup_{i \in I_\rho^-}
\varphi^{\prime i}$ in $\Omega_{\hat{D}_\rho, (W_{d^i})_{i \in
I_\rho^+}}$ such that
\begin{enumerate}
  \item $\Phi_{\hat{D}_\rho}$ and $\Phi_{\hat{D}_\rho}^\prime$ determine
$(\psi_{\bar{d}^i})_{i \in I_\rho},$
  \item $\varphi^i =
\varphi^{\prime i} \equiv \psi_{b(\bar{e}^i)} (b(\bar{e}^i))$ for
$i=1,3,$
\end{enumerate}
assume that there exists a homotopy $L^{\prime i} (s,t) :
[\hat{d}_+^i, \hat{d}_-^{i+1}] \times [0,1] \rightarrow
\mathcal{A}_{\bar{e}^i}^0$ for $i \in I_\rho^-$ such that $L^\prime
= \cup_{i \in I_\rho^-} L^{\prime i}$ satisfies
\begin{enumerate}
  \item $L_t^\prime \in \Omega_{\hat{D}_\rho, (W_{d^i})_{i \in I_\rho^+}}$
  for each $t \in [0,1],$
  \item $L_0^\prime = \Phi_{\hat{D}_\rho}$ and
  $L_1^\prime = \Phi_{\hat{D}_\rho}^\prime.$
\end{enumerate}
Then, we would show that there exists a homotopy $L^{\prime \prime
i} (s,t) : [\hat{d}_+^i, \hat{d}_-^{i+1}] \times [0,1] \rightarrow
\mathcal{A}_{\bar{e}^i}^0$ for $i \in I_\rho^-$ such that $L^{\prime
\prime} = \cup_{i \in I_\rho^-} L^{\prime \prime i}$ satisfies
\begin{enumerate}
  \item $L_t^{\prime \prime} \in \Omega_{\hat{D}_\rho,
(W_{d^i})_{i \in I_\rho^+}}$
  for each $t \in [0,1],$
  \item $L_0^{\prime \prime} = \Phi_{\hat{D}_\rho}$ and
  $L_1^{\prime \prime} = \Phi_{\hat{D}_\rho}^\prime.$
  \item $L_t^{\prime \prime i} \equiv
\psi_{b(\bar{e}^i)}(b(\bar{e}^i))$ for each $t \in [0,1]$ and
$i=1,3,$
  \item $L_t^{\prime \prime}$ determines
  $(\psi_{\bar{d}^i})_{i \in I_\rho}$ for each $t \in [0,1],$
\end{enumerate}
i.e. injectiveness of $\pi_0( \imath_2 ).$ To begin with, we would
show that $[L^{\prime 2}(0,t)]$ and $[L^{\prime 2}(1,t)]$ are
trivial in $\pi_1 \big( \Iso_H (V_{\bar{f}^{-1}}, V_{\bar{f}^2})
\big).$ Note that $[L^{\prime 1} (\frac 1 2,t)]$ is trivial in
$\pi_1 \big( \Iso_H (V_{\bar{f}^{-1}}, V_{\bar{f}^1}) \big) $
because $[L^{\prime 1} (\frac 1 2,t)]$ is in the trivial $\pi_1
\big( (\mathcal{A}_{b(\bar{e}^1)})_{\psi_{b(\bar{e}^1)}}^0 \big)$
where we regard $[L^{\prime 1} (\frac 1 2,t)]$ as a loop. From this,
we obtain that $[L^{\prime 1} (1,t)]$ is trivial in $\pi_1 \big(
\Iso_H (V_{\bar{f}^{-1}}, V_{\bar{f}^1}) \big) $ because $L^{\prime
1}(s, t)$ for $t=0,1$ is constant and $L^{\prime 1}$ is defined over
$[\hat{d}_+^1, \hat{d}_-^2] \times [0,1].$ Let $L_t^\prime$
determine $(\psi_{\bar{d}^i}^t)_{i \in I_\rho}$ in $\bar{A}_{G_\chi}
(\R^2/\Lambda, V_B)$ for $t \in [0,1]$ so that
\begin{align}
\tag{****} L^{\prime 2}(0,t) &= \big( \psi_{\bar{v}^2}^t
(\bar{v}_2^2)
\psi_{\bar{v}^2}^t (\bar{v}_1^2) \big)^{-1} L^{\prime 1}(1,t), \\
\notag L^{\prime 2}(1,t) &= \big( \psi_{\bar{v}^3}^t (\bar{v}_2^3)
\psi_{\bar{v}^3}^t (\bar{v}_1^3) \big) L^{\prime 3}(0,t).
\end{align}
Here, $[ \big( \psi_{\bar{v}^2}^t (\bar{v}_2^2) \psi_{\bar{v}^2}^t
(\bar{v}_1^2) \big)^{-1} ]$ is trivial in $\pi_1 \big( \Iso_H
(V_{\bar{f}^1}, V_{\bar{f}^2}) \big)$ because it is contained in
$\pi_1 \big( (\mathcal{A}_{\bar{v}^2})_{1,3}^3 \big)$ and each
component of $(\mathcal{A}_{\bar{v}^2})_{1,3}^3$ is simply connected
by Proposition \ref{proposition: psi for Z_2}. From this and (****),
triviality of $[L^{\prime 1} (1,t)]$ gives triviality of $[L^{\prime
2}(0,t)]$ in $\pi_1 \big( \Iso_H (V_{\bar{f}^{-1}}, V_{\bar{f}^2})
\big).$ In the same reason, $[L^{\prime 2}(1,t)]$ is trivial in
$\pi_1 \big( \Iso_H (V_{\bar{f}^{-1}}, V_{\bar{f}^2}) \big).$ By
triviality, there exist two homotopies $F^i (s,t) : [0,1] \times
[0,1] \rightarrow \Iso_H (V_{\bar{f}^{-1}}, V_{\bar{f}^2})$ for
$i=2,3$ such that
\begin{align*}
F^2 (s, 0) &= F^2 (s, 1) = F^2 (0, t) = \psi_{\bar{v}^2} (\bar{v}_0^2), \\
F^2 (1, t) &= L^{\prime 2} (0, t), \\
F^3 (s, 0) &= F^3 (s, 1) = F^3 (1, t) = \psi_{\bar{v}^3}^{-1} (\bar{v}_0^3). \\
F^3 (0, t) &= L^{\prime 2} (1, t), \\
\end{align*}
Then, we construct a wanted homotopy $L^{\prime \prime}$ as
\begin{equation*}
  \begin{array}{ll}
L^{\prime \prime 2} (s,t) = F^2 (3s,t) &
\text{for } s \in [0, \frac 1 3], \\
L^{\prime \prime 2} (s,t) = L^{\prime 2} (3s-1,t) & \text{for }
s \in [\frac 1 3, \frac 2 3], \\
L^{\prime \prime 2} (s,t) = F^3 (3s-2,t)
& \text{for } s \in [\frac 2 3, 1], \\
L^{\prime \prime 1} (s,t) \equiv \psi_{b(\bar{e}^1)} (b(\bar{e}^1)), \\
L^{\prime \prime 3} (s,t) \equiv \psi_{b(\bar{e}^3)} (b(\bar{e}^3)). \\
  \end{array}
\end{equation*}

Since $\pi_0( \imath_2 )$ is bijective and $\pi_0( \Omega_2 )$ is in
one-to-one correspondence with $\pi_1 (\Iso_H ( V_{\bar{f}^{-1}},
V_{\bar{f}^2} ))$ by Lemma \ref{lemma: relative homotopy}, we have
$\pi_0 \big( \Omega_{\hat{D}_\rho, (W_{d^i})_{i \in I_\rho^+}} \big)
\cong \pi_1 \big(\Iso_H ( V_{\bar{f}^{-1}}, V_{\bar{f}^2} ) \big)$
which is isomorphic to $\Z$ by Lemma \ref{lemma: reduce to smaller
matrix}. As in the proof of \cite[Theorem A.]{Ki}, it can be shown
that Chern classes of equivariant vector bundles determined by
different classes in $\pi_0 \big( \Omega_{\hat{D}_\rho, (W_{d^i})_{i
\in I_\rho^+}} \big)$ are all different, and the statement on the
set of possible Chern classes is obtained. Here, we use the fact the
$\rho(G_\chi)$-action on $\R^2/\Lambda$ is orient-preserving and the
cardinality of the $G_\chi$-orbit of $b(e^2)$ is equal to $l_\rho.$
Then, we obtain a proof by Lemma \ref{lemma: lemma for isomorphism}.
\end{proof}

\section{The case of $R/R_t = \langle \id \rangle, \D_1$
with $A\bar{c}+\bar{c}=\bar{l}_0$}
 \label{section: id and D_1 cases}

In this section, we calculate homotopy of equivariant clutching maps
in cases of $R/R_t = \langle \id \rangle, \D_1$ with
$A\bar{c}+\bar{c}=\bar{l}_0.$ First, we show that we only have to
consider equivariant clutching maps with a simple form. Then, we
give a discrimination on whether two equivariant clutching maps with
the simple form are equivariantly homotopic or not. With these, we
can calculate homotopy of equivariant clutching maps in the case of
$R/R_t = \langle \id \rangle.$ But, in the case of $R/R_t = \D_1$
with $A\bar{c}+\bar{c}=\bar{l}_0,$ Lemma \ref{lemma: lemma for
isomorphism} is not applied and we should also consider
$G_\chi$-isomorphisms in addition to equivariant clutching maps as
in cases of $\rho(G_\chi) = \Z_n \times Z$ with odd $n$ or $\langle
-a_n \rangle$ with even $n/2$ of \cite{Ki}.

In these two cases, $\bar{D}_\rho = [\bar{v}^2, \bar{v}^3,
\bar{v}^0]$ with $v^i \sim v^{i^\prime}$ and $(G_\chi)_x = H$ for
each $x \in \R^2/\Lambda$ by Table \ref{table: one-dimensional
fundamental domain} and tables in Section \ref{section: equivariant
simplicial complex}. For each $(W_{d^i})_{i \in I_\rho^+}$ in
$A_{G_\chi} (\R^2/\Lambda, \chi),$ put
\begin{equation*}
V_B = G_\chi \times_{(G_\chi)_{d^{-1}}} W_{d^{-1}}, \quad F_{V_B} =
G_\chi \times_{(G_\chi)_{d^{-1}}} ( |\bar{f}^{-1}| \times W_{d^{-1}}
).
\end{equation*}
Since $V_{\bar{f}} \cong W_{d^{-1}}$ for each $\bar{f} \in
\lineK_\rho$ in these cases, we may assume that all $V_{\bar{f}}$'s
are the same as $H$-representation and $\Iso_H (V_{\bar{f}^{-1}},
V_{\bar{f}^i}) = \Iso_H ( V_{\bar{f}^{-1}} ).$ By this,
$\mathcal{A}_{\bar{e}^i}^0 = \Iso_H ( V_{\bar{f}^{-1}} ),$ and
Corollary \ref{corollary: Omega} is rewritten as follows:

\begin{lemma} \label{lemma: cluching condition for id, D_1}
A preclutching map $\Phi_{\hat{D}_\rho} = \varphi^2 \cup \varphi^3$
in $C^0(\hat{D}_\rho, V_B)$ is contained in $\Omega_{\hat{D}_\rho,
(W_{d^i})_{i \in I_\rho^+}}$ if and only if $\varphi^i$'s are
$\Iso_H ( V_{\bar{f}^{-1}} )$-valued and
\begin{align}
\tag{*} \varphi^2 (1)^{-1} \cdot g_0 \varphi^3 (1)^{-1} g_0^{-1}
\cdot
g_1^{-1} \varphi^2 (0) g_1 \cdot \varphi^3 (0) &= \id \text{ or} \\
\notag \varphi^2 (1)^{-1} \cdot (g_2 g_1) \varphi^3 (1) (g_2
g_1)^{-1} \cdot g_1^{-1} \varphi^2 (0) g_1 \cdot \varphi^3 (0) &=
\id
\end{align}
according to $R/R_t = \langle \id \rangle$ or $\D_1$ with
$A\bar{c}+\bar{c}=\bar{l}_0,$ respectively.
\end{lemma}

\begin{proof}
We prove this lemma only for the case of $R/R_t = \id$ because proof
for another case is similar. We would prove that
$\Phi_{\hat{D}_\rho} = \varphi^2 \cup \varphi^3$ is in the set of
Corollary \ref{corollary: Omega} if and only if $\varphi^i$'s are
$\Iso_H ( V_{\bar{f}^{-1}} )$-valued and they satisfy (*). Since
$\mathcal{A}_{\bar{e}^i}^0 = \Iso_H ( V_{\bar{f}^{-1}} ),$ the
condition $\varphi^i (\hat{x}) \in \mathcal{A}_{\bar{e}^i}^0$ of
Corollary \ref{corollary: Omega} holds if and only if $\varphi^i$'s
are $\Iso_H ( V_{\bar{f}^{-1}} )$-valued. So, we do not need pay
attention to this condition. First, we prove sufficiency. Corollary
\ref{corollary: Omega} says that
\begin{equation*}
\begin{array}{ll}
\varphi^2 (0) = \psi_{\bar{v}^2} (\bar{v}^2), & \varphi^2 (1) =
\psi_{\bar{v}^3}^{-1} (\bar{v}^3), \\
\varphi^3 (0) = \psi_{\bar{v}^3} (\bar{v}^3), & \varphi^3 (1) =
\psi_{\bar{v}^0}^{-1} (\bar{v}^0)
\end{array}
\end{equation*}
for some $(\psi_{\bar{d}^i})_{i \in I_\rho}$ in $\bar{A}_{G_\chi}
(\R^2/\Lambda, V_B).$ Since $\psi_{\bar{v}^2} = g_1 \cdot
\psi_{\bar{v}^3}$ and $\psi_{\bar{v}^0} = g_0^{-1} \cdot
\psi_{\bar{v}^3},$ $\psi_{\bar{v}^3}^4 = \id$ is expressed as (*).
For necessity, define a pointwise preclutching map
$\psi_{\bar{v}^3}$ as
\begin{equation*}
\begin{array}{ll}
\psi_{\bar{v}^3} (\bar{v}^3) = \varphi^3 (0), \quad &
\psi_{\bar{v}^3} (\bar{v}_1^3) = g_1^{-1} \varphi^2 (0) g_1,   \\
\psi_{\bar{v}^3} (\bar{v}_2^3) = g_0 \varphi^3 (1)^{-1} g_0^{-1} ,
\quad & \psi_{\bar{v}^3} (\bar{v}_3^3) = \varphi^2 (1)^{-1}.
\end{array}
\end{equation*}
Then, (*) says that $\psi_{\bar{v}^3}$ is in
$\mathcal{A}_{\bar{v}^3}$ by Proposition \ref{proposition: psi for
trivial group}. If we put $\psi_{\bar{v}^2} = g_1 \cdot
\psi_{\bar{v}^3}$ and $\psi_{\bar{v}^0} = g_0^{-1} \cdot
\psi_{\bar{v}^3},$ then it is easily checked that $\varphi^i$'s are
in the set of Corollary \ref{corollary: Omega}. Therefore, we obtain
a proof.
\end{proof}

Given $(W_{d^i})_{i \in I_\rho^+},$ we would pick a simple
$(\psi_{\bar{d}^i})_{i \in I_\rho}$ in $\bar{A}_{G_\chi}
(\R^2/\Lambda, V_B)$ which determines $(W_{d^i})_{i \in I_\rho^+}$
as follows:

\begin{lemma} \label{lemma: simple psi}
For each $(W_{d^i})_{i \in I_\rho^+}$ in $A_{G_\chi} (\R^2/\Lambda,
\chi),$ there exists the unique $(\psi_{\bar{d}^i})_{i \in I_\rho}$
in $\bar{A}_{G_\chi} (\R^2/\Lambda, V_B)$ which determines
$(W_{d^i})_{i \in I_\rho^+}$ and satisfies
\begin{equation*}
\psi_{\bar{v}^2} (\bar{v}^2) = \psi_{\bar{v}^3}^{-1} (\bar{v}^3) =
\psi_{\bar{v}^3} (\bar{v}^3) = \id.
\end{equation*}
\end{lemma}

\begin{proof}
Since $(G_\chi)_x = H$ for each $x \in \R^2/\Lambda,$ we note that
$W_{d^i}$'s are all isomorphic $H$-representations and any element
of $\mathcal{A}_{\bar{x}}$ for any $\bar{x}$ determines the same
$W_{d^{-1}}.$

First, we prove existence. Since $v^i \sim v^{i^\prime},$ it
suffices to construct a suitable element in
$\mathcal{A}_{\bar{v}^3}.$ Define $\psi_{\bar{v}^3}$ as
\begin{align}
\notag \psi_{\bar{v}^3} (\bar{v}^3)   &= \id, \\
\notag \psi_{\bar{v}^3} (\bar{v}_1^3) &= g_1^{-1} \id g_1, \\
\notag \psi_{\bar{v}^3} (\bar{v}_3^3) &= \id, \\
\tag{*} \psi_{\bar{v}^3} (\bar{v}_2^3) &= \psi_{\bar{v}^3}
(\bar{v}_3^3)^{-1} \psi_{\bar{v}^3} (\bar{v}^3)^{-1}
\psi_{\bar{v}^3} (\bar{v}_1^3)^{-1}.
\end{align}
Here, the expression $g_1^{-1} \id g_1$ might seem awkward, but two
$g_1$'s act on different spaces because $\id$ is an element of $\Iso
( V_{\bar{f}^{-1}}, V_{\bar{f}^2} ).$ So, $g_1^{-1} \id g_1$ needs
not be equal to $\id.$ By Proposition \ref{proposition: psi for
trivial group}, $\psi_{\bar{v}^3}$ is contained in
$\mathcal{A}_{\bar{v}^3}.$ Other maps $\psi_{\bar{v}^2},
\psi_{\bar{v}^0}$ are obtained by group action from
$\psi_{\bar{v}^3}$ so that $(\psi_{\bar{d}^i})_{i \in I_\rho}$ is
contained in $\bar{A}_{G_\chi} (\R^2/\Lambda, V_B).$ Since
$\psi_{\bar{v}^3} (\bar{v}_1^3) = g_1^{-1} \id g_1$ and
$\psi_{\bar{v}^2} = g_1 \cdot \psi_{\bar{v}^3},$ we obtain
$\psi_{\bar{v}^2} (\bar{v}^2) = \id$ so that $(\psi_{\bar{d}^i})_{i
\in I_\rho}$ satisfies the lemma.

Uniqueness is easy. Since $\psi_{\bar{v}^2} = g_1 \cdot
\psi_{\bar{v}^3},$ the condition $\psi_{\bar{v}^2} (\bar{v}^2) =
\id$ gives $\psi_{\bar{v}^3} (\bar{v}_1^3) = g_1^{-1} \id g_1.$ And,
$\psi_{\bar{v}^3}^4 = \id$ gives (*). So, $\psi_{\bar{v}^3}$ in the
above is the unique map to satisfy the lemma. Other maps
$\psi_{\bar{v}^2}, \psi_{\bar{v}^0}$ are uniquely obtained by group
action from $\psi_{\bar{v}^3}.$
\end{proof}
For the unique element  $(\psi_{\bar{d}^i})_{i \in I_\rho}$ of Lemma
\ref{lemma: simple psi}, put $c_0 = \psi_{\bar{v}^0}^{-1}
(\bar{v}^0).$ So, if $\Phi_{\hat{D}_\rho} = \varphi^2 \cup
\varphi^3$ in $\Omega_{\hat{D}_\rho, (W_{d^i})_{i \in I_\rho^+}}$
determines $(\psi_{\bar{d}^i})_{i \in I_\rho},$ then we have
\begin{equation*}
\varphi^2 (0) = \varphi^2 (1) = \varphi^3 (0) = \id, \quad \varphi^3
(1) = c_0.
\end{equation*}
Pick a loop $\sigma : [0,1] \rightarrow \Iso_H( V_{\bar{f}^{-1}} )$
with $\sigma(0) = \sigma(1) = \id$ such that $[\sigma]$ is a
generator of $\pi_1 (\Iso_H ( V_{\bar{f}^{-1}} ))$ where $\pi_1
(\Iso_H ( V_{\bar{f}^{-1}} )) \cong \Z$ by Lemma \ref{lemma: reduce
to smaller matrix}, and pick a path $\gamma_0 : [0,1] \rightarrow
\Iso_H( V_{\bar{f}^{-1}} )$ with $\gamma_0 (0) = \id$ and $\gamma_0
(1) = c_0.$ To calculate $\pi_0 \big( \Omega_{\hat{D}_\rho,
(W_{d^i})_{i \in I_\rho^+}} \big),$ we might consider only
$\Phi_{\hat{D}_\rho}$'s with the following simple form:

\begin{lemma} \label{lemma: simple Phi D_1}
In the case of $R/R_t = \langle \id \rangle$ or $\D_1$ with
$A\bar{c}+\bar{c}=\bar{l}_0,$ any $\Phi_{\hat{D}_\rho} = \varphi^2
\cup \varphi^3$ in $\Omega_{\hat{D}_\rho, (W_{d^i})_{i \in
I_\rho^+}}$ is homotopic to some $\Phi_{\hat{D}_\rho}^\prime =
\varphi^{\prime 2} \cup \varphi^{\prime 3}$ which satisfies
$\varphi^{\prime 2} \equiv \id,$ $\varphi^{\prime 3} =
\sigma^k.\gamma_0$ for some $k \in \Z.$
\end{lemma}

\begin{proof}
We only prove the case of $R/R_t = \D_1$ with
$A\bar{c}+\bar{c}=\bar{l}_0$ because proofs are similar. First, we
would construct a homotopy $L^i (s,t) : [\hat{d}_+^i,
\hat{d}_-^{i+1}] \times [0,1] \rightarrow \Iso(V_{\bar{f}^{-1}})$
for each $i \in I_\rho^- = \{ 2, 3 \}$ such that the homotopy $L =
L^2 \cup L^3$ satisfies
\begin{enumerate}
  \item $L_t \in \Omega_{\hat{D}_\rho, (W_{d^i})_{i \in I_\rho^+}}$
  for each $t \in [0,1],$
  \item $L_0 = \Phi_{\hat{D}_\rho}$ and $L_1^2 \equiv \id,$
  \item $L^3(0,1) = \id$ and $L^3(1,1) = c_0.$
\end{enumerate}
Let $\gamma^i : [0,1] \rightarrow \Iso_H (V_{\bar{f}^{-1}})$ for
$i=2,3$ be paths satisfying $\gamma^2(0)=\varphi^2(0),$
$\gamma^3(0)=\varphi^3(0),$ $\gamma^i(1)=\id.$ Define $L^2$ on
$\partial ( [\hat{d}_+^2, \hat{d}_-^3] \times [0,1] )$ as
\begin{equation*}
  \begin{array}{ll}
L^2 (s,0)   = \varphi^2 (s)      & \text{ for } s \in [0, 1],   \\
L^2 (0,t)   = \gamma^2 (t)       & \text{ for } t \in [0, 1],  \\
L^2 (s,1)   = \id                & \text{ for } s \in [0, 1],   \\
L^2 (1,t)   = \varphi^2 (1-2t)   & \text{ for } t \in [0, \frac 1 2], \\
L^2 (1,t)   = \gamma^2 (2t-1)    & \text{ for } t \in [\frac 1 2, 1], \\
  \end{array}
\end{equation*}
and define $L^3$ on $\big( [\hat{d}_+^3, \hat{d}_-^0] \times 0 \big)
\cup \big(
\partial [\hat{d}_+^3, \hat{d}_-^0] \times [0,1] \big)$ as
\begin{align}
\notag  L^3 (s,0) &= \varphi^3(s), \\
\notag  L^3 (0,t) &= \gamma^3 (t), \\
(g_2 g_1) L^3 (1,t) (g_2 g_1)^{-1} &= L^2 (1,t) \cdot L^3 (0,t)^{-1}
\cdot g_1^{-1} L^2 (0,t)^{-1} g_1
\end{align}
for each $s,t.$ Then, we can extend these to the whole $\hat{D}_\rho
\times [0,1]$ to obtain a wanted homotopy $L$ by Lemma \ref{lemma:
cluching condition for id, D_1}. So, any $\Phi_{\hat{D}_\rho} =
\varphi^2 \cup \varphi^3$ in $\Omega_{\hat{D}_\rho, (W_{d^i})_{i \in
I_\rho^+}}$ is homotopic to some $\Phi_{\hat{D}_\rho}^\prime =
\varphi^{\prime 2} \cup \varphi^{\prime 3}$ which satisfies
$\varphi^{\prime 2} \equiv \id,$ $\varphi^{\prime 3}(0) = \id,$
$\varphi^{\prime 3}(1) = c_0.$

Next, we would construct a homotopy $L^{\prime i} (s,t) :
[\hat{d}_+^i, \hat{d}_-^{i+1}] \times [0,1] \rightarrow
\Iso(V_{\bar{f}^{-1}})$ for each $i \in I_\rho^-$ so that the
homotopy $L^\prime = L^{\prime 2} \cup L^{\prime 3}$ satisfies
\begin{enumerate}
  \item $L_t^\prime \in \Omega_{\hat{D}_\rho, (W_{d^i})_{i \in I_\rho^+}}$
  for each $t \in [0,1],$
  \item $L_0^\prime = \Phi_{\hat{D}_\rho}^\prime$ and $L^{\prime 2} \equiv \id,$
  \item $L^{\prime 3}(0,t) = \id,$ $L^{\prime 3}(1,t) = c_0,$
  $L_1^{\prime 3} = \sigma^k.\gamma_0$ for some $k \in \Z.$
\end{enumerate}
Since $\big[ ~ [0,1],0,1;\Iso_H (V_{\bar{f}^{-1}}),\id,c_0 ~ \big]
\cong \pi_1 (\Iso_H (V_{\bar{f}^{-1}}))$ by Lemma \ref{lemma:
relative homotopy} and each class is represented by
$\sigma^k.\gamma_0$ for some $k \in \Z,$ a wanted $L^\prime$ is
easily obtained. Therefore, we obtain a proof.
\end{proof}

We can determine whether two arbitrary $\Phi_{\hat{D}_\rho}$ and
$\Phi_{\hat{D}_\rho}^\prime$ with the simple form are equivariantly
homotopic or not.

\begin{lemma} \label{lemma: simple homotopy D_1}
In the case of $R/R_t = \langle \id \rangle$ or $\D_1$ with
$A\bar{c}+\bar{c}=\bar{l}_0,$ let $\Phi_{\hat{D}_\rho} = \varphi^2
\cup \varphi^3$ and $\Phi_{\hat{D}_\rho}^\prime = \varphi^{\prime 2}
\cup \varphi^{\prime 3}$ in $\Omega_{\hat{D}_\rho, (W_{d^i})_{i \in
I_\rho^+}}$ satisfy
\begin{equation*}
\varphi^2 = \varphi^{\prime 2} \equiv \id, \quad \varphi^3 =
\sigma^k.\gamma_0, \quad \varphi^{\prime 3} =
\sigma^{k^\prime}.\gamma_0
\end{equation*}
for some $k, k^\prime \in \Z.$ Then, two $\Phi_{\hat{D}_\rho}$ and
$\Phi_{\hat{D}_\rho}^\prime$ are homotopic in $\Omega_{\hat{D}_\rho,
(W_{d^i})_{i \in I_\rho^+}}$ if and only if $k=k^\prime$ or $k
\equiv k^\prime \mod 2$ according to $R/R_t = \langle \id \rangle$
or $\D_1$ with $A\bar{c}+\bar{c}=\bar{l}_0,$ respectively.
\end{lemma}

\begin{proof}
We only prove the case of $R/R_t = \D_1$ with
$A\bar{c}+\bar{c}=\bar{l}_0$ because proofs are similar.

First, we prove sufficiency. Let $L = L^2 \cup L^3$ be a homotopy
between $\Phi_{\hat{D}_\rho}$ and $\Phi_{\hat{D}_\rho}^\prime$ in
$\Omega_{\hat{D}_\rho, (W_{d^i})_{i \in I_\rho^+}}.$ Lemma
\ref{lemma: cluching condition for id, D_1} gives
\begin{equation}
\tag{*} g_1^{-1} L^2 (0,t) g_1 \cdot L^3 (0,t) \cdot L^2 (1,t)^{-1}
\cdot (g_2 g_1) L^3 (1,t) (g_2 g_1)^{-1} = \id
\end{equation}
for each $t$ because $L_t \in \Omega_{\hat{D}_\rho, (W_{d^i})_{i \in
I_\rho^+}}.$ Since $L^2$ is defined on $[\hat{d}_+^2, \hat{d}_-^3]
\times [0,1]$ and $L^2(s,t) \equiv \id$ for $t=0,1,$ we have
$[L^2(0,t)]=[L^2(1,t)]$ when they are regarded as elements in $\pi_1
(\Iso_H (V_{\bar{f}^{-1}})).$ Then, (*) says that
\begin{equation}
\tag{**} [L^3(1,t)] = [L^3(0,t)^{-1}]
\end{equation}
in $\pi_1 (\Iso_H (V_{\bar{f}^{-1}})).$ Since $L^3$ is defined over
$[\hat{d}_+^3, \hat{d}_-^0] \times [0,1],$ the homotopy $L^3$ on
$\partial ([\hat{d}_+^3, \hat{d}_-^0] \times [0,1])$ is trivial in
$\pi_1 (\Iso_H (V_{\bar{f}^{-1}})).$ From this and (**), we obtain a
proof of sufficiency.

Next, we prove necessity. Let $l = (k-k^\prime)/2.$ We would define
a homotopy $L = L^2 \cup L^3$ connecting $\Phi_{\hat{D}_\rho}$ and
$\Phi_{\hat{D}_\rho}^\prime$ in $\Omega_{\hat{D}_\rho, (W_{d^i})_{i
\in I_\rho^+}}.$ Define $L^2 (s,t) \equiv \id,$ and define $L^3
(s,t) :
\partial([\hat{d}_+^3, \hat{d}_-^0] \times [0,1])
\rightarrow \mathcal{A}_{\bar{e}^3}^0$ as
\begin{align}
\notag L^3 (s,0) &= \varphi^3 (s), \\
\notag L^3 (s,1) &= \varphi^{\prime 3} (s), \\
\notag L^3 (0,t) &= \sigma^l (t), \\
\tag{***} (g_2 g_1) L^3 (1,t) (g_2 g_1)^{-1} &= \Big( ~ g_1^{-1} L^2
(0,t) g_1 \cdot L^3 (0,t) \cdot L^2 (1,t)^{-1} ~ \Big)^{-1}
\end{align}
which satisfies (*). Since $L^2 \equiv \id,$ (***) gives $[L^3(1,t)]
= [L^3(0,t)^{-1}] = [\sigma^{-l}(t)]$ in $\pi_1 (\Iso_H
(V_{\bar{f}^{-1}})).$ Then, it is checked that the homotopy $L^3$ on
$\partial ([\hat{d}_+^3, \hat{d}_-^0] \times [0,1])$ is trivial in
$\pi_1 (\Iso_H (V_{\bar{f}^{-1}})).$ So, we can extend $L^3$ to the
whole $[\hat{d}_+^3, \hat{d}_-^0] \times [0,1].$ Therefore, we
obtain a proof.
\end{proof}

By using these lemmas, we can calculate $\pi_0 \big(
\Omega_{\hat{D}_\rho, (W_{d^i})_{i \in I_\rho^+}} \big).$

\begin{theorem} \label{theorem: by isotropy and chern for id}
In the case of $R/R_t = \langle \id \rangle,$
\begin{equation*}
p_\vect \times c_1 : \Vect_{G_\chi} (\R^2/\Lambda, \chi) \rightarrow
A_{G_\chi} ( \R^2/\Lambda, \chi ) \times H^2 (\R^2/\Lambda, \chi)
\end{equation*}
is injective. For each element $(W_{d^i})_{i \in I_\rho^+}$ in
$A_{G_\chi} ( \R^2/\Lambda, \chi ),$ the set of the first Chern
classes of bundles in the preimage of the element under $p_{\vect}$
is equal to the set $\{ \chi( \id ) ( l_\rho k + k_0 ) ~ | ~ k \in
\Z \}$ where $k_0$ is dependent on $(W_{d^i})_{i \in I_\rho^+}.$
\end{theorem}

\begin{proof}
First, we prove that $\pi_0( \Omega_{\hat{D}_\rho, (W_{d^i})_{i \in
I_\rho^+}} )$ is in one-to-one correspondence with $\Z.$ By Lemma
\ref{lemma: simple Phi D_1}, we may assume that arbitrary two
$\Phi_{\hat{D}_\rho} = \varphi^2 \cup \varphi^3$ and
$\Phi_{\hat{D}_\rho}^\prime = \varphi^{\prime 2} \cup
\varphi^{\prime 3}$ in $\Omega_{\hat{D}_\rho, (W_{d^i})_{i \in
I_\rho^+}}$ satisfy \begin{equation*} \varphi^2 = \varphi^{\prime 2}
\equiv \id, \quad \varphi^3 = \sigma^k.\gamma_0, \quad
\varphi^{\prime 3} = \sigma^{k^\prime}.\gamma_0
\end{equation*}
for some $k, k^\prime \in \Z.$ By Lemma \ref{lemma: simple homotopy
D_1}, $\Phi_{\hat{D}_\rho}$ is homotopic to
$\Phi_{\hat{D}_\rho}^\prime$ if and only if $k=k^\prime.$ So, we
obtain
\begin{equation*}
\pi_0( \Omega_{\hat{D}_\rho, (W_{d^i})_{i \in I_\rho^+}} ) \cong \Z.
\end{equation*}
As in Theorem \ref{theorem: by isotropy and chern}, the statement on
Chern classes is obtained. Therefore, we obtain a proof by Lemma
\ref{lemma: lemma for isomorphism}.
\end{proof}

Now, we can prove Theorem \ref{main: by isotropy and chern}.
\begin{proof}
Theorem \ref{theorem: by isotropy and chern} and Theorem
\ref{theorem: by isotropy and chern for id} give a prove Theorem
\ref{main: by isotropy and chern}.
\end{proof}

Now, we deal with the case of $R/R_t = \D_1$ with
$A\bar{c}+\bar{c}=\bar{l}_0.$

\begin{proposition} \label{proposition: two bundles}
In the case of $R/R_t = \D_1$ with $A\bar{c}+\bar{c}=\bar{l}_0,$ the
homotopy $\pi_0 \big( \Omega_{\hat{D}_\rho, (W_{d^i})_{i \in
I_\rho^+}} \big)$ for each element $(W_{d^i})_{i \in I_\rho^+}$ in
$A_{G_\chi} ( \R^2/\Lambda, \chi )$ has two elements which give
equivariant vector bundles with the same Chern class.
\end{proposition}

\begin{proof}
First, we obtain that $\pi_0( \Omega_{\hat{D}_\rho, (W_{d^i})_{i \in
I_\rho^+}} )$ has two elements for each $(W_{d^i})_{i \in I_\rho^+}$
$\in$ $A_{G_\chi} ( \R^2/\Lambda, \chi )$ similarly to the proof of
Theorem \ref{theorem: by isotropy and chern for id}.

To prove the second statement of the proposition, let
$\Phi_{\hat{D}_\rho} = \varphi^2 \cup \varphi^3$ in
$\Omega_{\hat{D}_\rho, (W_{d^i})_{i \in I_\rho^+}}$ be the element
to satisfy
\begin{equation*}
\varphi^2 \equiv \id \text{ and } \varphi^3 = \sigma^k.\gamma_0
\end{equation*}
for some $k \in \Z.$ We would show that the Chern class of the
bundle $E^{\Phi_{\hat{D}_\rho}}$ is independent of $k$ as in the
proof of \cite[Theorem C.]{Ki}. To resemble the proof, we need merge
trivializations over a pair of faces as explained in the below.
Denote by $|\bar{f}^{-1}| \cup_c |\bar{f}^2|$ the quotient of the
disjoint union $|\bar{f}^{-1}| \amalg |\bar{f}^2|$ by the
identification of two points $p_{|\lineL|}(\hat{x})$ in
$|\bar{f}^{-1}|$ and $p_{|\lineL|}(|c|(\hat{x}))$ in $|\bar{f}^2|$
for each $\hat{x}$ in $|\hat{e}^2|.$ Since $\varphi^2 \equiv \id,$
trivializations of $F_{V_B}$ over $|\bar{f}^{-1}|$ and $|\bar{f}^2|$
merge into a trivialization over $|\bar{f}^{-1}| \cup_c
|\bar{f}^2|.$ Pick $g_k$'s in $G_\chi$ such that
\begin{equation*}
\R^2/\Lambda \quad = \quad \pi \Big( g_1 |\bar{f}^{-1}| \cup_c g_1
|\bar{f}^2| \Big)
  ~ \bigcup \cdots \bigcup ~
\pi \Big( g_{l_0} |\bar{f}^{-1}| \cup_c g_{l_0} |\bar{f}^2| \Big)
\end{equation*}
for $k = 1, \cdots, l_0$ where $l_0 = |\Lambda_t / \Lambda|$ and
$g_1 = \id.$ Similar to $|\bar{f}^{-1}| \cup_c |\bar{f}^2|,$
trivializations of $F_{V_B}$ over $g_k |\bar{f}^{-1}|$ and $g_k
|\bar{f}^2|$ are merged through $g_k \varphi^2 g_k^{-1}.$ By using
these trivializations, we calculate Chern class.

Let $\Phi$ in $\Omega_{V_B}$ be the extension of
$\Phi_{\hat{D}_\rho}.$ Since $|\hat{e}^3| \cup g_1^{-1} g_2
|c|(|\hat{e}^3|)$ give an edge of $|\bar{f}^{-1}| \cup_c
|\bar{f}^2|,$ we investigate $\Phi$ on it to calculate Chern class.
Recall that we parameterize $|\hat{e}^3|$ by $s \in [0,1]$ so that
$\hat{v}_{0,+}^3=0$ and $\hat{v}_{0,-}^0=1.$ By using this,
parameterize $g_1^{-1} g_2 |c|(|\hat{e}^3|)$ with $[-1, 0]$ to
satisfy $g_1^{-1} g_2 (|c|(\hat{x})) = s-1$ for each $\hat{x} = s$
in $|\hat{e}^3|$ with $s \in [0,1].$ Then, $\Phi(s-1)$ for $s \in
[0,1]$ is equal to
\begin{align*}
\Phi(g_1^{-1} g_2 |c|(s)) &= g_1^{-1} g_2 \Phi(|c|(s)) g_2^{-1} g_1 \\
&= g_1^{-1} g_2 \Phi(s)^{-1} g_2^{-1} g_1 \\
&= g_1^{-1} g_2 (\sigma^k.\gamma_0)(s)^{-1} g_2^{-1} g_1.
\end{align*}
Here, note that $\Phi(-1) = g_1^{-1} g_2 \id g_2^{-1} g_1.$ Since
$\Phi(s) = \sigma^k.\gamma_0(s)$ on $|\hat{e}^3|$ and $\Phi(s-1) =
g_1^{-1} g_2 (\sigma^k.\gamma_0)(s)^{-1} g_2^{-1} g_1$ on $g_1^{-1}
g_2 |c|(|\hat{e}^3|)$ merge to cancel each other regardless of $k$
in the level of
\begin{equation*}
\Big[ [-1,1], -1, 1; \Iso_H (V_{\bar{f}^{-1}}), g_1^{-1} g_2 \id
g_2^{-1} g_1  ,c_0 \Big],
\end{equation*}
Chern class of $E^{\Phi_{\hat{D}_\rho}}$ is independent of $k \in
\Z.$ This gives a proof.
\end{proof}

Since elements in $\pi_0( \Omega_{\hat{D}_\rho, (W_{d^i})_{i \in
I_\rho^+}} )$ are not discriminated by their Chern classes, we need
to consider $G_\chi$-isomorphisms by Proposition \ref{proposition:
equivalent condition for isomorphism}. In the case of $R/R_t = \D_1$
with $A\bar{c}+\bar{c}=\bar{l}_0,$ $(G_\chi)_x = H$ for each $x \in
\R^2/\Lambda$ and $G_\chi$ acts transitively on $B$ so that any
$H$-isomorphism of $(\res_H^{G_\chi}
 F_{V_B}) |_{|\bar{f}^{-1}|},$ which is just a function from
$|\bar{f}^{-1}|$ to $\Iso_H (V_{\bar{f}^{-1}}),$ can be
equivariantly extended to $G_\chi$-isomorphism of $F_{V_B}$ over the
whole $|\lineK_\rho|.$ In this extension, we need restrict our
arguments to $|\partial \bar{f}^{-1}|$ or $\lineL_\rho$ because
$G_\chi$-isomorphisms are related with equivariant clutching maps
which are defined on $\hatL_\rho.$ A map $\theta: |\partial
\bar{f}^{-1}| \rightarrow \Iso_H (V_{\bar{f}^{-1}})$ can be extended
to an $H$-isomorphism $\bar{\theta}$ of $(\res_H^{G_\chi}
F_{V_B})|_{|\bar{f}^{-1}|}$ if $[\theta]$ is trivial in $\pi_1
(\Iso_H (V_{\bar{f}^{-1}})).$ For such a $\theta,$ define $\theta^i
: |\hat{e}^i| \rightarrow \Iso_H (V_{\bar{f}^{-1}})$ as the function
$\theta^i(s) = \theta|_{|\bar{e}^i|}( p_{|\lineL|} (s) )$ for each
$i \in \Z_{i_\rho}$ and $s \in |\hat{e}^i|.$ For $\theta$ and an
extension $\bar{\theta}$ such that $\theta = \bar{\theta}
|_{|\partial \bar{f}^{-1}|},$ there exists the $G_\chi$-isomorphism
of $F_{V_B}$ which extends $\bar{\theta}$ to the whole
$|\lineK_\rho|.$ An $G_\chi$-isomorphism $\Theta$ of $F_{V_B}$
satisfying the condition $(p_{|\lineL|}^* \Theta) \Phi
(p_{|\lineL|}^* \Theta)^{-1} = \Phi^\prime$ of Proposition
\ref{proposition: equivalent condition for isomorphism} has a simple
form at $|\partial \bar{f}^{-1}|$ as follows:

\begin{lemma} \label{lemma: simple isomorphism}
Assume that $R/R_t = \D_1$ with $A\bar{c}+\bar{c}=\bar{l}_0.$ Let
$\Phi_{\hat{D}_\rho} = \varphi^2 \cup \varphi^3$ and
$\Phi_{\hat{D}_\rho}^\prime = \varphi^{\prime 2} \cup
\varphi^{\prime 3}$ be elements in $\Omega_{\hat{D}_\rho,
(W_{d^i})_{i \in I_\rho^+}}$ such that
\begin{equation}
\label{equation: simple form of clutching map} \varphi^2 =
\varphi^{\prime 2} \equiv \id, \quad \varphi^3(0)= \varphi^{\prime
3}(0)=\id, \quad \varphi^3(1)= \varphi^{\prime 3}(1)=c_0,
\end{equation}
and let $\Phi$ and $\Phi^\prime$ be the extensions of
$\Phi_{\hat{D}_\rho}$ and $\Phi_{\hat{D}_\rho}^\prime,$
respectively. Let $\Theta$ be a $G_\chi$-isomorphism of $F_{V_B}$
satisfying $(p_{|\lineL|}^* \Theta) \Phi (p_{|\lineL|}^*
\Theta)^{-1} = \Phi^\prime.$ Then, there exists a
$G_\chi$-isomorphism $\Theta^\prime$ of $F_{V_B}$ which still
satisfies $(p_{|\lineL|}^* \Theta^\prime) \Phi (p_{|\lineL|}^*
\Theta^\prime)^{-1} = \Phi^\prime$ and also satisfies $\Theta^\prime
(\bar{v}^i)=\id$ for each $i.$
\end{lemma}

\begin{proof}
First, we show that $\Theta(\bar{v}^0)$ determines
$\Theta(\bar{v}^i)$ for $i=1,2,3.$ Put $\theta=\Theta|_{|\partial
\bar{f}^{-1}|}.$ By equivariance of $\Theta,$ $p_{|\lineL|}^*
\Theta$ on $|c(\hat{e}^2)| \cup |c(\hat{e}^3)|$ is calculated as
follows:
\begin{align}
\label{equation: isomorphism on adjacent edges} p_{|\lineL|}^*
\Theta(s) &= p_{|\lineL|}^* \Theta(\hat{x}) = g_2 p_{|\lineL|}^*
\Theta(g_2^{-1} \hat{x}) g_2^{-1}
\quad \text{ for } s=\hat{x} \in |c(\hat{e}^2)| \\
\notag          &= g_2 \theta^0 (1-s) g_2^{-1}, \\
\notag p_{|\lineL|}^* \Theta(s) &= p_{|\lineL|}^* \Theta(\hat{x}) =
g_1^{-1} p_{|\lineL|}^* \Theta(g_1 \hat{x}) g_1 \quad \text{ for } s=\hat{x} \in |c(\hat{e}^3)| \\
\notag          &= g_1^{-1} \theta^1 (s) g_1
\end{align}
where we recall that we parameterize each edge $c(\hat{e}^i)$
linearly by $s \in [0,1]$ so that $|c|(s) = 1-s$ for $s \in
|\hat{e}^i|.$ On $\hat{D}_\rho,$ the condition $(p_{|\lineL|}^*
\Theta) \Phi (p_{|\lineL|}^* \Theta)^{-1} = \Phi^\prime$ is written
as
\begin{align}
\label{equation: endomorphism condition} g_2 \theta^0(s) g_2^{-1} ~
\varphi^2(s)  ~  \theta^2(s)^{-1} &=
\varphi^{\prime 2} (s), \\
\notag g_1^{-1} \theta^1(1-s) g_1  ~  \varphi^3(s)  ~
\theta^3(s)^{-1} &= \varphi^{\prime 3} (s)
\end{align}
by using (\ref{equation: isomorphism on adjacent edges}). By
applying (\ref{equation: simple form of clutching map}) and
substituting $s=0,1$ to these, we obtain
\begin{equation*}
\begin{array}{ll}
g_2 \theta^0 (0) g_2^{-1} = \theta^2(0), \quad & g_2 \theta^0 (1)
g_2^{-1} = \theta^2(1), \\
g_1^{-1} \theta^1(1) g_1 = \theta^3(0), \quad & g_1^{-1} \theta^1(0)
g_1 = c_0 \theta^3(1) c_0^{-1}.
\end{array}
\end{equation*}
That is,
\begin{equation}
\label{equation: endomorphism for s=0,1}
\begin{array}{ll}
g_2 \theta (\bar{v}^0) g_2^{-1} = \theta (\bar{v}^2), \quad & g_2
\theta
(\bar{v}^1) g_2^{-1} = \theta (\bar{v}^3), \\
g_1^{-1} \theta (\bar{v}^2) g_1 = \theta (\bar{v}^3), \quad &
g_1^{-1} \theta (\bar{v}^1) g_1 = c_0 \theta (\bar{v}^0) c_0^{-1}.
\end{array}
\end{equation}
These show that $\theta(\bar{v}^0)$ determines $\theta(\bar{v}^i)$
for $i=1,2,3.$ Let us express these relations by three maps. Define
$\delta^i : \Iso_H (V_{\bar{f}^{-1}}) \rightarrow \Iso_H
(V_{\bar{f}^{-1}})$ for $i=1,2,3$ as
\begin{equation*}
A \mapsto g_1 c_0 A c_0^{-1} g_1^{-1}, ~ g_2 A g_2^{-1}, ~ g_1^{-1}
g_2 A g_2^{-1} g_1,
\end{equation*}
respectively. Note that $\delta^i( \theta(\bar{v}^0) ) =
\theta(\bar{v}^i)$ for $i=1,2,3$ by (\ref{equation: endomorphism for
s=0,1}). Let $\gamma: [0,1] \rightarrow \Iso_H (V_{\bar{f}^{-1}})$
be a path such that $\gamma(0)=\id$ and
$\gamma(1)=\theta(\bar{v}^0).$ By using these, we would define a new
map $\theta^\prime : |\partial \bar{f}^{-1}| \rightarrow \Iso_H
(V_{\bar{f}^{-1}})$ from which an extension $\Theta^\prime$ will be
obtained. If we define $\theta^{\prime i} : |\hat{e}^i| \rightarrow
\Iso_H (V_{\bar{f}^{-1}})$ for $i=0,3$ as
\begin{equation*}
\begin{array}{ll}
\theta^{\prime 0} (s) = \gamma (3s)         & \text{for } s \in [0, 1/3], \\
\theta^{\prime 0} (s) = \theta^0 (3s-1)       & \text{for } s \in [1/3, 2/3], \\
\theta^{\prime 0} (s) = \delta^1 (\gamma(3-3s)) & \text{for } s \in [2/3, 1], \\
\theta^{\prime 3} (s) = \delta^3 (\gamma(3s))   & \text{for } s \in
[0, 1/3], \\
\theta^{\prime 3} (s) = \theta^3 (3s-1)       & \text{for } s \in [1/3, 2/3], \\
\theta^{\prime 3} (s) = \gamma (3-3s)         & \text{for } s \in
[2/3, 1],
\end{array}
\end{equation*}
then it is easily checked that both are well defined and
\begin{equation*}
\theta^{\prime 0} (s) = \id, \qquad \theta^{\prime 3} (s) = \id
\end{equation*}
for $s=0,1.$ If we define $\theta^{\prime 1}$ and $\theta^{\prime
2}$ by substituting four $\theta^{\prime i}$'s for $\theta^i$'s in
(\ref{equation: endomorphism condition}), then $\theta^{\prime i}$'s
give the map $\theta^\prime : |\partial \bar{f}^{-1}| \rightarrow
\Iso_H (V_{\bar{f}^{-1}})$ satisfying $\theta^{\prime i} (s) =
\theta^\prime |_{|\bar{e}^i|}( p_{|\lineL|} (s) )$ for each $i$ and
$s.$ Similarly, if we define $\theta_t^\prime : |\partial
\bar{f}^{-1}| \rightarrow \Iso_H (V_{\bar{f}^{-1}})$ by using
$\gamma_t (s) = \gamma \Big( 1-t(1-s) \Big)$ instead of $\gamma$ in
definition of $\theta^\prime,$ then it gives a homotopy between
$\theta$ and $\theta^\prime.$ So, $[\theta]=[\theta^\prime]$ in
$\pi_1 ( \Iso_H (V_{\bar{f}^{-1}}) )$ is trivial, and
$\theta^\prime$ can be extended to the whole $|\lineK_\rho|.$ If
$\Theta^\prime$ is an extension of $\theta^\prime,$ then
$\Theta^\prime (\bar{v}^i)=\id$ for each $i$ by definition of
$\theta^\prime$ and also $(p_{|\lineL|}^* \Theta^\prime) \Phi
(p_{|\lineL|}^* \Theta^\prime)^{-1} = \Phi^\prime$ because
$\theta^\prime$ satisfies \ref{equation: endomorphism condition}.
Therefore, $\Theta^\prime$ is a wanted $H$-isomorphism.
\end{proof}

Now, we can finish classification.
\begin{theorem} \label{theorem: D_1}
In the case of $R/R_t = \D_1$ with $A\bar{c}+\bar{c}=\bar{l}_0,$ the
preimage of each $(W_{d^i})_{i \in I_\rho^+}$ in $A_{G_\chi} (
\R^2/\Lambda, \chi )$ under $p_\vect$ has two elements.
\end{theorem}

\begin{proof}
By Proposition \ref{proposition: two bundles}, $\pi_0(
\Omega_{\hat{D}_\rho, (W_{d^i})_{i \in I_\rho^+}} )$ has two
elements $[\Phi_{\hat{D}_\rho}] \ne [\Phi_{\hat{D}_\rho}^\prime].$
We would show that they determine non-isomorphic equivariant
bundles. We may assume that $\Phi_{\hat{D}_\rho} = \varphi^2 \cup
\varphi^3,$ $\Phi_{\hat{D}_\rho}^\prime = \varphi^{\prime 2} \cup
\varphi^{\prime 3}$ with $\varphi^2 = \varphi^{\prime 2} \equiv
\id,$ $\varphi^3(s) = \sigma^k.\gamma_0,$ $\varphi^{\prime 3}(s) =
\sigma^{k^\prime}.\gamma_0$ for some $k, k^\prime \in \Z$ by Lemma
\ref{lemma: simple Phi D_1}. By Lemma \ref{lemma: simple homotopy
D_1}, $k$ is not congruent to $k^\prime$ modulo 2. Let $\Phi$ and
$\Phi^\prime$ be the extensions of $\Phi_{\hat{D}_\rho}$ and
$\Phi_{\hat{D}_\rho}^\prime,$ respectively.  Assume that two bundles
$E^\Phi$ and $E^{\Phi^\prime}$ are $G_\chi$-isomorphic, i.e.
$(p_{|\lineL|}^* \Theta) \Phi (p_{|\lineL|}^* \Theta)^{-1} =
\Phi^\prime$ for some $\Theta$ by Proposition \ref{proposition:
equivalent condition for isomorphism}. By Lemma \ref{lemma: simple
isomorphism}, we may assume that $\Theta(\bar{v}^i)=\id$ for each
$i.$ Put $\theta = \Theta |_{|\partial \bar{f}^{-1}|}.$ Formula
(\ref{equation: endomorphism condition}) says that
$[\theta^0(s)]=[\theta^2(s)]$ in $\pi_1(\Iso_H (V_{\bar{f}^{-1}}))$
because $\varphi^2 = \varphi^{\prime 2} \equiv \id.$ This gives
$[\theta^1(s)] \equiv [\theta^3(s)] \mod 2$ in $\pi_1(\Iso_H
(V_{\bar{f}^{-1}}))$ because $[\theta]$ is trivial in $\pi_1(\Iso_H
(V_{\bar{f}^{-1}})).$ From this, $k \equiv k^\prime \mod 2$ because
$g_1^{-1} \theta^1(1-s) g_1 ~ \varphi^3(s) ~ \theta^3(s)^{-1} =
\varphi^{\prime 3} (s)$ by (\ref{equation: endomorphism condition}).
This is a contradiction. Therefore, $E^\Phi$ and $E^{\Phi^\prime}$
are not $G_\chi$-isomorphic and we obtain a proof.
\end{proof}

In summary, we obtain a proof of Theorem \ref{main: not by isotropy
and chern class}.

\begin{proof}
Proposition \ref{proposition: two bundles} and Theorem \ref{theorem:
D_1} gives a proof of Theorem \ref{main: not by isotropy and chern
class}.
\end{proof}


\begin{thebibliography}{}



\bibitem[CKMS]{CKMS}
J.-H. Cho, S. S. Kim, M. Masuda, D. Y. Suh, {\em Classification of
Equivariant Complex Vector Bundles over a Circle}, J. Math. Kyoto Univ.
\textbf{41} (2001), 517-534.




\bibitem[Ki]{Ki}
M. K. Kim, {\em Classification of equivariant vector bundles over
two-sphere}, preprint.

\bibitem[Ki2]{Ki2}
M. K. Kim, {\em Classification of equivariant vector bundles over
Klein bottle}, preprint.







\end{thebibliography}
\end{document}